\documentclass[10pt]{article}

\usepackage{amssymb,amsbsy,graphicx,amsmath,amsthm,amscd,psfig,diagrams} 

\usepackage{times}

\begin{document}

\title{On the testability and repair of hereditary hypergraph properties}
\author{Tim Austin \and Terence Tao}
\date{}

\maketitle


\newenvironment{nmath}{\begin{center}\begin{math}}{\end{math}\end{center}}

\newtheorem{theorem}{Theorem}[section]
\newtheorem{thm}{Theorem}[section]
\newtheorem{lemma}[thm]{Lemma}
\newtheorem{lem}[thm]{Lemma}
\newtheorem{prop}[thm]{Proposition}
\newtheorem{proposition}[thm]{Proposition}
\newtheorem{corollary}[thm]{Corollary}
\newtheorem{cor}[thm]{Corollary}
\newtheorem{conj}[thm]{Conjecture}
\newtheorem{definition}[thm]{Definition}
\newtheorem{dfn}[thm]{Definition}
\newtheorem{example}[thm]{Example}
\newtheorem{examples}[thm]{Examples}
\newtheorem{remark}[thm]{Remark}
\newtheorem{prob}[thm]{Problem}
\newtheorem{ques}[thm]{Question}
\newtheorem{assertion}[thm]{Assertion}


\newcommand{\B}{\mathcal{B}}
\newcommand{\calH}{\mathcal{H}}
\renewcommand{\P}{\mathcal{P}}
\newcommand{\W}{\mathcal{W}}
\newcommand{\X}{\mathcal{X}}
\newcommand{\D}{\mathcal{D}}
\newcommand{\Y}{\mathcal{Y}}
\newcommand{\Q}{\mathcal{Q}}
\newcommand{\I}{\mathcal{I}}
\renewcommand{\X}{\mathcal{X}}
\newcommand{\calZ}{\mathcal{Z}}
\renewcommand{\L}{\Lambda}

\renewcommand{\Pr}{\operatorname{Pr}}
\newcommand{\Hom}{\operatorname{Inj}}
\newcommand{\Col}{\operatorname{Col}}
\newcommand{\id}{\operatorname{id}}
\renewcommand{\th}{{\operatorname{th}}}
\newcommand{\pt}{\operatorname{pt}}
\newcommand{\supp}{{\operatorname{supp}}}
\newcommand{\Stab}{{\operatorname{stab}}}

\newcommand{\E}{\mathbf{E}}
\newcommand{\N}{\mathbf{N}}
\newcommand{\R}{\mathbf{R}}
\newcommand{\Z}{\mathbf{Z}}
\newcommand{\Prob}{\mathbf{P}}
\newcommand{\eps}{\varepsilon}


\newcommand{\bb}[1]{\mathbb{#1}}
\renewcommand{\rm}[1]{\mathrm{#1}}
\renewcommand{\it}[1]{\mathit{#1}}
\renewcommand{\cal}[1]{\mathcal{#1}}
\renewcommand{\bf}[1]{\mathbf{#1}}
\renewcommand{\frak}[1]{\mathfrak{#1}}

\begin{abstract}
Recent works of Alon-Shapira \cite{AloSha2} and R\"odl-Schacht \cite{rs2} have demonstrated that every hereditary property of undirected graphs or hypergraphs is testable with one-sided error; informally, this means that if a graph or hypergraph satisfies that property ``locally'' with sufficiently high probability, then it can be perturbed (or ``repaired'') into a graph or hypergraph which satisfies that property ``globally''.

In this paper we make some refinements to these results, some of which may be surprising.  In the positive direction, we strengthen the results to cover hereditary properties of multiple directed polychromatic graphs and hypergraphs.  In the case of undirected graphs, we extend the result to continuous graphs on probability spaces, and show that the repair algorithm is ``local'' in the sense that it only depends on a bounded amount of data; in particular, the graph can be repaired in a time linear in the number of edges. We also show that local repairability also holds for monotone or partite hypergraph properties (this latter result is also implicitly in \cite{ishi}).  In the negative direction, we show that local repairability breaks down for directed graphs, or for undirected $3$-uniform hypergraphs.  The reason for this contrast in behavior stems from (the limitations of) Ramsey theory.
\end{abstract}

\parskip 0pt

\tableofcontents

\parskip 7pt

\section{Introduction}

The purpose of this paper is to investigate various generalisations of some recent graph and hypergraph property testing results of Alon-Shapira\cite{AloSha2}, R\"odl-Schacht\cite{rs2}, and others, when the graphs and hypergraphs are allowed to become coloured, non-uniform, directed and/or containing loops.  We also investigate a stronger property than local testability of such properties, which we call ``local repairability''.  Very briefly, our conclusions will be that the local testability results of R\"odl and Schacht extend to very general settings, but that the stronger local repairability results of Alon and Shapira are largely restricted to the setting of undirected graphs.

\subsection{Previous results}\label{prevsec}

Before discussing the general setting of coloured, non-uniform, directed hypergraphs in which our main results will take place, we first discuss the more familiar setting of monochromatic, uniform, undirected graphs and hypergraphs, which is the focus of most of the previous literature on this subject.  

We begin with the property testing theory for (monochromatic, undirected) graphs $G = (V,E)$, where $V$ is a finite vertex set and $E \subset \binom{V}{2}$ is\footnote{We use $\binom{V}{k} := \{e \subset V: |e|=k\}$ to denote the $k$-element subsets of $V$, and $|e|$ to denote the cardinality of a finite set $e$.} a set of edges in $V$.  One can also view such a graph as a map\footnote{The notational conventions in this section may seem somewhat odd, but will become clearer in the next section when we generalise these notions to coloured, non-uniform, and directed hypergraphs.  The subscript $2$, in particular, has to do with the $2$-uniform nature of graphs, i.e. that all edges consist of two vertices; the set $\{0,1\}$, meanwhile, is there to emphasise the monochromatic nature of the graph.} $G_2: \binom{V}{2} \to \{0,1\}$, where $G_2(\{v,w\})$ equals $1$ when $\{v,w\}$ lies in $E$ and equals zero otherwise.  The set of all graphs on a fixed vertex set $V$ will be denoted $2^{\binom{V}{2}}$.

A \emph{graph property} $\P$ is an assertion which holds true for some graphs and not for others.  More formally, such a property assigns to each vertex set $V$ a collection $\P^{(V)} \subset \{0,1\}_2^{(V)}$ of graphs on $V$, defined as the set of graphs on $V$ that obey $\P$.  Thus, for instance, if $\P$ is the property of being bipartite, then $\P^{(V)}$ is the collection of bipartite graphs on $V$.

We will restrict attention to two special types of graph properties, namely the monotone and hereditary properties.  A graph property $\P$ is \emph{hereditary} if, for every injection $\phi: V \to W$ between two finite sets $V, W$, and any graph $G \in \P^{(W)}$ on $W$ obeying $\P$, the pullback graph (or \emph{induced graph}) $\{0,1\}_2^{(\phi)}(G)$ on $V$ (defined by declaring an edge $\{v_1,v_2\}$ to lie in $\{0,1\}_2^{(\phi)}(G)$ if and only if $\{\phi(v_1),\phi(v_2)\}$ lies in $W$) also obeys $\P$; in other words, the pullback map $\{0,1\}_2^{(\phi)}$ maps $\P^{(W)}$ to $\P^{(V)}$.  In particular, this implies that the graph property is invariant with respect to graph isomorphism, and is also preserved by passing from a graph $G \in 2^{\binom{V}{2}}$ to an induced subgraph $G \downharpoonright_W \in 2^{\binom{V}{2}}$ for any $W \subset V$.  A \emph{monotone} graph property is a hereditary graph property with the additional property that if one takes a graph in $\P^{(V)}$ and removes one or more edges from it, then the graph continues to have the property $\P$.

\begin{example} The properties of being $4$-colourable, bipartite, or triangle-free are monotone (and hence hereditary).  Given any $k > 0$, the properties of being connected, or of avoiding either the empty graph on $k$ vertices or the complete graph on $k$ vertices are hereditary (but not monotone).  The property of having an odd number of edges, or containing a Hamiltonian cycle, are neither monotone nor hereditary.  It is not hard to show that a graph property $\P$ is monotone if and only if there is a (possibly infinite) family ${\mathcal F}$ of ``forbidden subgraphs'', such that a graph $G$ obeys $\P$ if and only if does not have any of the graphs in ${\mathcal F}$ as subgraphs, while $\P$ is hereditary if and only if there is a family of ${\mathcal F}$ of ``forbidden induced subgraphs'' such that a graph $G$ obeys $\P$ if and only if it does not have any of the graphs in ${\mathcal F}$ as \emph{induced} subgraphs.  For further discussion of monotone and hereditary graph properties, see \cite{AloSha2}.
\end{example}

We now come to the key notion of \emph{testability}.

\begin{definition}[Testability for graph properties]\label{testgraph}\cite{RS}  A graph property $\P$ is said to be \emph{locally testable with one-sided error}, or \emph{testable} for short, if for every $\eps > 0$ there exists $N \geq 1$ and a real number $\delta > 0$ with the following property: whenever $G = (V,E)$ is a graph with $N \leq |V| < \infty$ which locally almost obeys $\P$ in the sense that
\begin{equation}\label{voom}
\frac{1}{|\binom{V}{N}|} |\{ W \in \binom{V}{N}: G \downharpoonright_W \in \P^{(W)} \}| \geq 1-\delta
\end{equation}
(thus most $N$-element induced subgraphs of $G$ obey $\P$), then there exists $G' = (V,E')$ obeying $\P$ which is close to $G$ in the sense that $\frac{1}{\binom{V}{2}} |E \Delta E'| \leq \eps$.
\end{definition}

\begin{remark} See \cite{AloSha2} for a discussion as to why the above concept is equivalent to testability with one-sided error, as defined in \cite{RS}.
\end{remark}

The following is the main result of \cite{AloSha2}:

\begin{theorem}\label{as-basic}\cite{AloSha2} Every hereditary graph property is testable.
\end{theorem}

See \cite{AloSha2} for a history of this result and for a survey of the many prior results in this direction, including the earlier result in \cite{AloSha} that every monotone graph property is testable.  The proof of this theorem is rather intricate, involving repeated application of the Szem\'eredi regularity lemma, as well as Ramsey's theorem.

Theorem \ref{as-basic} has been generalised in two different ways.  Firstly, the work of R"odl and Schacht \cite{rs2} found a somewhat simpler (but more indirect) argument, avoiding Ramsey's theorem and using only a single instance of the (hypergraph) regularity lemma, which extended Theorem \ref{as-basic} to the setting of hypergraph properties for $k$-uniform hypergraphs $G = (V,E)$, where $k \geq 2$ and $E \subset \binom{V}{k}$.  It is straightforward to extend all of the above definitions to the $k$-uniform hypergraph setting, basically by replacing $2$ with $k$ in all the above definitions; we omit the details (and in any case, we will also make further generalisations of these definitions in the next section).

The main result of \cite{rs2} can now be stated as follows:

\begin{theorem}\label{rs-basic}\cite{rs2} Let $k \geq 2$.  Then every hereditary $k$-uniform hypergraph property is testable.
\end{theorem}

This builds upon a number of earlier hypergraph results which can be interpreted as testability results, such as the hypergraph removal lemma \cite{Gow1}, \cite{rs}, \cite{rodl2} or the induced $3$-uniform hypergraph removal lemma in \cite{KohNagRod}.  We refer the reader to \cite{rs2} for further references and discussion.

There is however a different way to generalise Theorem \ref{as-basic}, in which we stay in the setting of graphs, but instead replace testability by a stronger property which we call \emph{local repairability}, and which is analogous to the notion of \emph{local correctability} in the theory of error-correcting codes.  (Actually, we will eventually discuss two such properties, \emph{strong local repairability} and \emph{weak local repairability}, but we will only discuss the weak one for now.)  For simplicity we now restrict attention to graphs rather than $k$-uniform hypergraphs.

To motivate this concept, recall that if $G$ is a graph that locally almost obeys a testable property $\P$, then it is guaranteed that there is a way to modify a small number of edges of $G$ to create a graph $G'$ which truly does obey $\P$.  We will refer to the act of replacing $G$ by $G'$ as \emph{repairing} the graph $G$.  However, note that no \emph{algorithm} is provided in order to actually execute this repair; one can of course perform a brute force search among all possible candidate graphs $G'$, but this will take a time which is at least exponential in the number of vertices and is thus impractical.  It is thus of interest to determine whether a testable graph (or hypergraph) property $\P$ also comes with an ``efficient'' algorithm that can repair a graph $G$ quickly.  We will focus on a rather strong notion of efficiency here, namely that of a \emph{local} repair algorithm, in which any edge of the repaired graph $G'$ can be decided upon using only a \emph{bounded} number of queries to the original graph $G$ (which in particular implies that the entire graph can be repaired in time linear in the total number $\binom{|V|}{2}$ of possible edges).  More precisely, we seek repair algorithms which are given by a \emph{local modification rule}, which we will define shortly.  For technical reasons we will have to delete a small set $A$ of ``training'' vertices in order to perform this rule; thus the rule will start with a graph $G = (A \uplus V, E)$ almost obeying $\P$, where $A \uplus V$ is the disjoint union of a large vertex set $V$ and a small set $A$ of training vertices, and return a repaired graph $G' = (V,E')$ which obeys $\P$ exactly, but for which the training vertices $A$ have been deleted.

To motivate the concept of a local modification rule, let us discuss (somewhat informally) a specific example of repairability, in which $\P$ is the property of being a complete bipartite graph.  (For instance, one could think of the vertex set of a graph obeying $\P$ as a collection of positive and negative charges, with an edge between two vertices if they have opposite charge.)  Now consider a large graph $G_0 = (A \uplus V,E_0)$ obeying $\P$, and ``corrupt'' it to create a new graph $G = (A \uplus V,E)$ formed by adding or removing a small fraction of the edges to $E_0$.  (For instance, one could imagine a large collection of real-world charged particles $A \uplus V$, with an edge between two vertices $v,w$ in $E$ if the two particles are observed to attract each other in some (mostly reliable) measurement; in this case, the corruption between $E$ and the ``true'' graph $E_0$ would be caused by measurement error.)  Then $G$ approximately obeys $\P$.  If one is given $G$ (but not $G_0$), we now consider the task of repairing $G$ to form a graph $G' = (V,E')$ close to $G$ which obeys $\P$.  (Ideally, we would like $G'$ to recover the original uncorrupted graph $G_0$, but there is not enough information given to do so exactly, and will settle for obtaining a slightly different repaired graph $G'$ which is still complete and bipartite.)  Continuing our measurement example, this task would correspond to that of using the measured attraction and repulsion data to assign "charges" to various particles, thus attempting to correct for corrupted measurements and giving a prediction as to what the ``true'' attraction between any two particles are.

To do this, we first look at the restriction $G\downharpoonright_A$ of $G$ to the training vertices $A$.  If the training vertices were a sufficiently representative subset of the whole set $A \uplus V$ (which, in practice, we will ensure with high probability by drawing $A$ randomly from the vertex set of $G$), then we expect $G \downharpoonright_A$ to be very close to a complete bipartite graph.  By performing a brute force search on $A$ only, we can then find a complete bipartite graph $G'_A := (A, E'_A)$ on $A$ which is very close to $G\downharpoonright_A$ (and thus, presumably, also close to $G_0\downharpoonright_A$.  Note that while a brute force search on all of $V$ is exponentially expensive, if $A$ is bounded size then it will only take a bounded amount of time to locate $G'_A$.  Let $A = A_1 \uplus A_2$ be the partition of $A$ corresponding to the complete bipartite graph $G'_A$.  (This partition is only unique up to interchange of the labels $1,2$, but this will not concern us.)  We can then use this partition to create a partition $V = V_1 \uplus V_2$ of the larger vertex set $V$, by the following rule: a vertex $v$ will lie in $V_1$ if it is connected to more vertices of $A_2$ than to $A_1$, and in $V_2$ otherwise.  (Informally, $G'_A$ has ``decided'' which of the training vertices in $A$ are positively charged or negatively charged, and then one tests those charged particles against any other vertex $v \in V$ to decide whether $v$ should be classified as positive or negative only.)  We then define $G' = (V,E')$ to be the complete bipartite graph between $V_1$ and $V_2$.  Clearly $G'$ obeys $\P$; and it is intuitively clear that if $G$ is sufficiently close to $G_0$, and $A$ is sufficiently large (but still bounded) and drawn randomly from $G$, then $G'$ will be close to $G$ with high probability.  (In particular, if $G$ was exactly equal to $G_0$, one easily sees that $G'$ is equal to $G$.)

Now we make these concepts more precise. 

\begin{definition}[Local modification rule]\label{lacma}  A \emph{local modification rule} is a pair $(A,T)$, where $A$ is a finite set, and $T: 2^{\binom{A \uplus [2]}{2}} \to \{0,1\}$ is a map from graphs on $A \cup [2]$ to $\{0,1\}$, where $[2] := \{1,2\}$, which is symmetric with respect to interchange of the $1$ and $2$ labels.  Given any vertex set $V$, we define a \emph{modification map} $\overline{T}^{(V)}: 2^{\binom{A \uplus V}{2}} \to 2^{\binom{V}{2}}$ by declaring an edge $(v_1,v_2)$ in $V$ to lie in $\overline{T}^{(V)}(G)$ for some $G \in 2^{\binom{A \uplus V}{2}}$ if and only if $T( \{0,1\}_2^{(\id_A \oplus \phi)}(G) ) = 1$, where $\id_A \oplus \phi: A \uplus [2] \to A \uplus \{v_1,v_2\}$ is the map which is the identity on $A$ and maps $1,2$ to $v_1,v_2$ respectively.   
\end{definition}

\begin{example} The rule $G \mapsto G'$ defined in the preceding discussion can be viewed as a local modification rule, in which $T(G)$ for $G \in 2^{\binom{A \uplus [2]}{2}}$ is defined by first constructing the graph $G'_A$ and the partition $A=A_1 \cup A_2$ as above, and then $[2]$ is partitioned into $V_1 \cup V_2$, and $T(G) = 1$ if $1, 2$ lie in distinct partition classes, and $T(G)=0$ otherwise.
\end{example}

\begin{remark}\label{lacma2} Informally, a local modification rule only has to query $G$ between vertices in $\{v_1,v_2\} \cup A$ to decide how $v_1$ and $v_2$ are connected in $G' := \overline{T}^{(V)}(G)$; furthermore; all pairs $\{v_1,v_2\}$ are ``treated equally'' in the sense that the same modification function $T$ is applied to each of them.  There is also an equivalent category-theoretic definition of a local modification rule $(A,T)$, namely it is a finite set $A$ together with a \emph{natural transformation} $\overline{T}$, or more precisely a collection of maps $\overline{T}^{(V)}: 2^{\binom{A \uplus V}{2}} \to 2^{\binom{V}{2}}$ for every vertex set $V$ which obeys the natural transformation property
$$ \overline{T}^{(W)} \circ \{0,1\}_2^{(\id_A \oplus \phi)} = \{0,1\}_2^{(\phi)} \circ \overline{T}^{(V)}$$
for all injections $\phi: W \to V$ between two finite vertex sets $V, W$, where $\id_A \oplus \phi: A \cup W \to A \cup V$ is the extension of $\phi$ which is the identity on $A$.  This alternate characterisation of a local modification rule will be more convenient for us in later sections when we generalise to hypergraphs which may be multicoloured, non-uniform, directed, and/or infinite.
\end{remark}

\begin{definition}[Weak local repairability]\label{wlr}  Let ${\mathcal P}$ be a graph property.  We say that ${\mathcal P}$ is \emph{weakly locally repairable} if for every $\eps > 0$ there exists a finite set $A$, an integer $N \geq |A|+2$, and a $\delta > 0$ with the following property: if $G = (V,E)$ is a graph with $N \leq |V| < \infty$ which almost obeys ${\mathcal P}$ in the sense of \eqref{voom}, then there exists an embedding of $A$ in $V$ (thus identifying $V$ with $A \uplus V'$ for some $|V'| = |V| - |A|$) and a local modification rule $(A,T)$ such that $G' = (V',E') := T^{(V')}(G)$ obeys ${\mathcal P}$, and $G'$ is close to $G$ in the sense that 
$$|E' \Delta (E\downharpoonright_{V'})| \leq \eps |\binom{V'}{2}|$$
where $E\downharpoonright_{V'} := E \cap \binom{V'}{2}$.
\end{definition}

\begin{remark} Observe that weak local repairability stronger than local testability in the sense that the repaired graph $G'$ is given from $G$ by a local modification rule, but weaker because one had to remove a small number of vertices; see Remark \ref{easy-rem} for further discussion.
Also, observe that the embedding of $A$ in $V$ is not specified; also, the rule $(A,T)$ is only guaranteed to produce a graph $G'$ obeying ${\mathcal P}$ for the chosen input $G$.  Later on we shall introduce the notion of \emph{strong local repairability}, which roughly speaking is similar to weak local repairability, but the embedding of $A$ in $V$ is now chosen at random (and the algorithm has a small probability of failure), the rule $(A,T)$ now entails the property ${\mathcal P}$ for \emph{all} choices of input graph $G$, rather than being permitted to depend on $G$, and furthermore the graph $G$ is allowed to be infinite (or even ``continuous'') rather than just finite (or discrete).  However, to keep the discussion simple for now, we will not formally define strong local repairability until later sections.
\end{remark}

An inspection of the arguments in \cite{AloSha2} then reveals the following strengthening of Theorem \ref{as-basic}:

\begin{theorem}\label{as-2} Every hereditary graph property is weakly locally repairable.
\end{theorem}

Strictly speaking, this result is not explicitly stated in \cite{AloSha2}, but is an implicit consequence of their methods, together with the observation that Szemer\'edi partitions can be constructed using random neighbourhoods (see e.g. \cite{ishi-0}).  In any event we will establish a stronger version of this theorem in the next section.

\begin{example} We have informally discussed this result in the case when ${\mathcal P}$ is the property of being a complete bipartite graph.  Another illustrative example is the property of being triangle-free, which is a monotone property.  The local testability of this property is a well-known fact, often called the ``triangle-removal lemma'', and is due to Ruzsa and Szemer\'edi \cite{rsz}.  To repair an almost-triangle-free-graph into a genuinely triangle-free graph, the standard approach is to apply the Szemer\'edi regularity lemma \cite{szemeredi-reg} to the graph, and then delete all edges between pairs of cells of that partition that are too small, have too low an edge density, or too irregular.  This regularisation can be done in purely local fashion, by randomly selecting vertex neighbourhoods to create the partition (see e.g. \cite{ishi-0}), and this can be used to create a local modification rule to repair corrupted triangle-free graphs.
\end{example}

\subsection{General setup}

The prior results were restricted to properties for uniform monochromatic undirected graphs or hypergraphs without loops.  We now generalise much of the above discussion to a more general setting which allows for the hypergraphs to be non-uniform, directed, multi-coloured, and/or contain loops.  As such, there will be some overlap between the discussion here and that in the preceding section.

\begin{definition}[Vertex sets]\label{vertset}  A \emph{vertex set} is any set which is at most countable.  If $V$ and $W$ are vertex sets, we define a \emph{morphism} from $W$ to $V$ to be an injective map $\phi: W \to V$, and use $\Hom(W,V)$ to denote the space of such morphisms.  We use $\id_V \in \Hom(V,V)$ to denote the identity map from $V$ to itself, and if $W \subset V$, we use $\iota_{W \subset V} \in \Hom(W,V)$ to denote the inclusion map.  If $N$ is a non-negative integer, we use $[N] := \{ 1,\ldots,N\}$ to denote the vertex set of integers from $1$ to $N$.  If $v_1,\ldots,v_N$ are distinct vertices of $V$, we use $(v_1,\ldots,v_N) \in \Hom([N],V)$ to denote the morphism that sends $i$ to $v_i$ for all $i \in [N]$ (in particular, we canonically embed $\Hom([N],V)$ in $V^N$, and the unique element of $\Hom([0],V)$ is denoted $()$).  If $V$ is a set, we use $|V|$ to denote the cardinality of $V$, and for any $k \geq 0$ we let
$$\binom{V}{k} := \{ e \subset V: |e| = k \} \equiv \Hom([k],V)/\Hom([k],[k])$$
denote the $k$-element subsets of $V$.  If $V, W$ are vertex sets, we use $V \uplus W := (V \times \{0\}) \cup (W \times \{1\})$ to denote the disjoint union of $V$ and $W$.  We often abuse notation and view $V$ and $W$ as subsets of $V \uplus W$.  If $\phi_1 \in \Hom(W_1,V_1)$ and $\phi_2 \in \Hom(W_2,V_2)$, we use $\phi_1 \oplus \phi_2 \in \Hom(W_1 \uplus W_2, V_1 \uplus V_2)$ to denote the direct sum of $\phi_1$ and $\phi_2$.
\end{definition}

\begin{remark} One can view the collection of all vertex sets and their morphisms as a category.  We will make this category-theoretic perspective more explicit later in our analysis, as it contains a number of useful notions for us, most notably that of a \emph{natural transformation}.  However, readers who are not familiar with category theory can safely skip all remarks in this introductory section referring to this subject.
\end{remark}

\begin{definition}[Palettes] A \emph{finite palette} is a sequence $K = (K_j)_{j=0}^\infty$ of finite non-empty sets, all but finitely many of which are singleton sets.  We refer to the singleton components as \emph{points} and denote them by $\pt$.  We define the \emph{order} of $K$ to be the greatest integer $k$ for which $K_k$ is not a point (or $-1$ if all components are points).  We shall often abbreviate $K$ as $(K_0,\ldots,K_k)$ (thus discarding the trivial palettes $K_j = \pt$ for $j > k$).  For any $k \geq 0$, we define the \emph{monochromatic palette} $\{0,1\}_k$ of order $k$ to be the palette whose $k^{\th}$ component is $\{0,1\}$ and all other components are points.  If $j \in \Z$, we let $K_{\leq j}$ (resp. $K_{<j}$, $K_{\geq j}$, $K_{>j}$, $K_{=j}$) be the palette whose $i^{\th}$ component is $K_i$ when $i \leq j$ (resp. $i < j$, $i \geq j$, $i>j$, $i=j$), and is $\pt$ otherwise, thus for instance $K = K_{\geq 0} = K_{>-1}$.
\end{definition}

\begin{definition}[Hypergraphs]\label{hyperdef}  If $V$ is a vertex set, we define a \emph{$K$-coloured (directed) hypergraph} to be a tuple $G = (G_j)_{j=0}^\infty$, where each $G_j: \Hom([j],V) \to K_j$ is a function.  (Note that $G_j$ will be trivial when $K_j$ is a point, and so only finitely many of the $G_j$ are of any interest.  We will often abuse notation slightly by omitting the trivial components $G_j$ of a hypergraph.)  We let
$$K^{(V)} \equiv \prod_{j=0}^\infty K_j^{\Hom([j],V)}$$
denote the collection of all $K$-coloured hypergraphs on $V$.  We say that the hypergraph is \emph{undirected} if we have the symmetry property $G_j( \phi \circ \sigma ) = G_j(\phi)$ for all $j \geq 0$, all $\sigma \in \Hom([j],[j])$, and all $\phi \in \Hom([j],V)$.  If $\phi \in \Hom(W,V)$ is a morphism between vertex sets, we define the \emph{pullback map} $K^{(\phi)}: K^{(V)} \to K^{(W)}$ by defining $K^{(\phi)}(G)_j(\psi) := G_j(\phi \circ \psi)$ for all $G = (G_j)_{j=0}^\infty \in K^{(V)}$, $j \geq 0$, and $\psi \in \Hom([j],W)$. If $W$ is a subset of $V$, we write $G\downharpoonright_W$ for $K^{(\iota_{W \subset V})}(G)$, and refer to $G\downharpoonright_W$ as the \emph{restriction} of $G$ to $W$.
\end{definition}

\begin{example} An ordinary undirected graph $G = (V,E)$, where $E \subset \binom{V}{2}$ can be viewed as an undirected $\{0,1\}_2$-coloured hypergraph; similarly, a $k$-uniform hypergraph can be viewed as an undirected $\{0,1\}_k$-coloured hypergraph.  In particular, $2^{\binom{V}{2}}$ is nothing more than the hypergraphs in $\{0,1\}_2^{(V)}$ which are undirected.
More generally, if $G = (G_j)_{j=0}^\infty \in K^{(V)}$ is undirected, then the maps $G_j: \Hom([j],V) \to K_j$ can be viewed instead as maps from $\binom{V}{j}$ to $K_j$.  A bipartite graph can be viewed as an undirected $(\pt,\{0,1\},\{0,1\})$-coloured hypergraph, in which the order $1$ palette $\{0,1\}$ is used for the vertex partition, and the order $2$ palette $\{0,1\}$ is used to describe the edges of the graph.  One can similarly view partite hypergraphs using this framework; see also Definition \ref{partite} below.  Later on we will need to generalise the notion of a palette to allow the palettes $K_j$ to be sub-Cantor spaces instead of finite sets; see Definition \ref{subcantor}.
\end{example}

\begin{remark}  In the language of category theory, one can view the palette $K$ as a \emph{contravariant functor} $V \mapsto K^{(V)}$, $\phi \mapsto K^{(\phi)}$ between the category of  vertex sets $V$ (whose morphisms are the injective maps $\phi: W \to V$), and the category of sub-Cantor spaces (see  Definition \ref{subcantor} below), whose morphisms are the continuous maps (and more generally, the probability kernels, see Appendix \ref{prob}).  This category-theoretic language seems to be a natural framework to phrase many of our notions, such as local repairability, as we shall see in later sections.
\end{remark}

\begin{definition}[Hereditary hypergraph properties]\label{hered}  Let $K = (K_j)_{j=0}^\infty$ be a finite palette.  A \emph{hereditary $K$-property} is an assignment $\P: V \mapsto \P^{(V)}$ of a collection $\P^{(V)} \subset K^{(V)}$ of $K$-coloured hypergraphs on $V$ for every\footnote{Technically, the class of finite vertex sets is not itself a set, and so $\P$ is a class function rather than a function.  If one wishes to work with actual functions, one restricting attention to vertex sets which are (for instance) subsets of the integers.  As this issue does not make any actual impact on our arguments, we shall henceforth ignore it.}  finite vertex set $V$, such that
\begin{equation}\label{kphi}
K^{(\phi)}( \P^{(V)} ) \subset \P^{(W)}
\end{equation}
for every morphism $\phi \in \Hom(W,V)$ between finite vertex sets.  In particular, the $K$-property $\P$ is invariant under hypergraph isomorphism and preserved under hypergraph restriction\footnote{In category-theoretic language, one can view $\P$ (like $K$) as a contravariant functor, in which $\P^{(\phi)}: \P^{(V)} \to \P^{(W)}$ is the restriction of the pullback map $K^{(\phi)}$ to $\P^{(V)}$ for any injection $\phi: W \to V$; see Example \ref{propfunc}.}.  We say that the $K$-property $\P$ is \emph{undirected} if $\P^{(V)}$ consists entirely of undirected hypergraphs for each vertex set $V$.  We extend $\P$ to countably infinite vertex sets $V$ by declaring
$$\P^{(V)} := \{ G \in K^{(V)}: G\downharpoonright_W \in \P^{(W)} \hbox{ for all finite } W \subset V \}.$$
We say that a hypergraph $G \in K^{(V)}$ \emph{obeys} $\P$ if $G \in \P^{(V)}$.
\end{definition}

\begin{examples}  In the case of $\{0,1\}_2$-coloured hypergraphs (i.e. graphs), the properties of being undirected and connected, of being bipartite, of being undirected and free of triangles, of being planar, and of being four-colourable, are all hereditary $\{0,1\}_2$-properties.
\end{examples}

\begin{definition}[Testability]\label{Testdef}\cite{RS}  Let $K$ be a finite palette of some order $k \geq 0$, and let $\P$ be a hereditary $K$-property.  We say that $\P$ is \emph{testable with one-sided error} if, for every $\eps > 0$, there exists an integer $N \geq 1$ and a real number $\delta > 0$ with the following property: if $G \in K^{(V)}$ is a $K$-coloured hypergraph with $N \leq |V| < \infty$ which
locally almost obeys $\P$ in the sense that
\begin{equation}\label{injv}
\frac{1}{|\binom{V}{N}|} |\{ W \in \binom{V}{N}: G\downharpoonright_W \in \P^{(W)} \}| \geq 1-\delta,
\end{equation}
then there exists $G' \in \P^{(V)}$ which is close to $G$ in the sense that
\begin{equation}\label{gv}
\frac{1}{|\binom{V}{k}|} |\{ W \in \binom{V}{k}: G\downharpoonright_W \neq G'\downharpoonright_W \}| \leq \eps.
\end{equation}
\end{definition}

This definition of course generalises Definition \ref{testgraph}.

We can now state the main results of Alon-Shapira and R\"odl-Schacht again:

\begin{theorem}[Every hereditary undirected hypergraph property is testable]\label{as-thm}\cite{AloSha2},\cite{rs2}  If $k \geq 0$, then every hereditary undirected $\{0,1\}_k$-property is testable with one-sided error.
\end{theorem}

\begin{remark}  See \cite{AloSha2} for further discussion of this result, and why it is natural to restrict attention to hereditary properties.  The cases $k=0,1$ of this result are easy.  In the case of graphs $k=2$, this result was first obtained by \cite{AloSha2}, after building upon several earlier results in this direction; see \cite{GGR}, \cite{AloSha}, \cite{AloFisKriSze} and the references therein.  For general $k$, this result was first obtained in \cite{rs2}, with several earlier results in this direction in \cite{Gow1}, \cite{rs}, \cite{ars}, \cite{ishi}, \cite{KohNagRod}.  The special case of the above theorem in which $\P$ is the $\{0,1\}_k$-property of not containing any embedded copy of a fixed hypergraph is known as the \emph{hypergraph removal lemma} and is already a non-trivial result, implying for instance the multidimensional Szemer\'edi theorem; see \cite{Gow1}, \cite{rs} for further discussion.
\end{remark}

The Alon-Shapira argument \cite{AloSha2} that gave the $k=2$ case was somewhat intricate, using the Szemer\'edi regularity lemma three times and also using Ramsey's theorem for graphs.  The R\"odl-Schacht argument \cite{rs2}, in contrast, avoided Ramsey's theorem and used fewer applications of the (hypergraph) regularity lemma, leading to a simpler proof (though of course the fact that it dealt with general $k$ rather than $k=2$ lead to several notational complications).  On the other hand, the R\"odl-Schacht argument was more indirect than the Alon-Shapira one and did not yield explicitly quantitative bounds.  One of the purposes of this paper is to explain why this difference is in fact essential: the Alon-Shapira argument cannot extend to the case of general hypergraphs, for reasons which we shall explain below.

\subsection{New positive results}

In this paper we explore some generalisations and refinements of the above theorem, as well as counterexamples to some of these extensions.  Some obvious generalisations include that of allowing more general palettes $K$, allowing directed edges, allowing loops, and replacing the finite vertex set $V$ with a more general probability space such as $[0,1]$ with uniform measure.  Another direction to pursue is to determine the relationship between the original hypergraph $G$ in the above theorems and the ``repaired'' hypergraph $G'$.  For instance, the argument in \cite{AloSha2} gives an effective procedure to locate $G'$ (albeit one which requires heavy use of the regularity lemma); in contrast, the argument in \cite{rs2} is indirect (proceeding by contradiction) and does not obviously provide any algorithm for locating $G'$ other than brute force search.

In the positive direction we have three main results.  The first result extends Theorem \ref{as-thm} to the directed multicoloured case:

\begin{theorem}[Every hereditary directed hypergraph property is
testable]\label{rs-thm-dir}  Let $K$ be a finite palette, and let $\P$ be a hereditary $K$-property.  Then $\P$ is testable with one-sided error.
\end{theorem}

The proof of Theorem \ref{rs-thm-dir} follows the R\"odl-Schacht argument and is given in Section \ref{posi}.

\begin{remark}\label{directed} As is well known, one can identify a directed graph with an undirected bipartite graph on twice as many vertices, and similar identifications also exist for hypergraphs.  However, it does not appear possible to use such identifications to deduce the testability of directed hypergraph properties from the testability of undirected hypergraph properties, because one cannot canonically recover the directed graph from the undirected one without knowledge of the specific identification used.  Indeed, the negative result in Theorem \ref{negate} below shows that the directed and undirected cases are in fact quite different.  On the other hand, this distinction between directed and undirected hypergraphs disappears for partite properties; see Remark \ref{partite-directed}.
\end{remark}

The next result extends Theorem \ref{as-thm} (in the graph case $k=2$) in a different direction, namely showing that hereditary undirected graph properties are not only testable with one-sided error, but enjoy the stronger property of being \emph{locally repairable}.  Roughly speaking, local repairability (which is somewhat analogous to the concept of \emph{local correctability} in coding theory) shows that the repaired graph $G'$ can be (probabilistically) obtained from $G$ in a ``local'' manner, in that every edge of $G'$ can be determined using only knowledge of $O(1)$ edges of $G$.  Because of this locality, the testability theorem can in fact be extended from finite graphs $G$ to infinite graphs $G$ (with a probability measure on the vertices), and also one can allow the graphs to contain loops.  In fact this turns out to be a natural setting in which to study a certain strong form of local repairability.

To make this more precise we need more definitions, beginning with a continuous analogue of a graph or hypergraph.

\begin{definition}[Continuous hypergraphs]\label{contmap}  Let $K$ be a finite palette. A \emph{$K$-coloured continuous hypergraph} is a quadruplet $G = (V,\B,\nu,(G_j)_{j=0}^\infty)$, where $(V,\B,\nu)$ is a probability space, and $G_j: V^j \to K_j$ is a measurable map for each $j \geq 0$. If $W$ is an vertex set, we define the \emph{sampling map} $\overline{G}^{(W)}: V^W \to K^{(W)}$ by the formula
$$ \overline{G}^{(W)}( v )_j(\phi) = G_j(v \circ \phi)$$
for all $j \geq 0$, all $\phi \in \Hom([j],W)$, and all $v \in V^W$, where we view $v$ as a function from $W$ to $V$ (and identify $V^j$ with $V^{[j]}$).  If $\P$ is a $K$-property, we say that $G$ \emph{obeys} $\P$ if $\overline{G}^{(W)}(v) \in \P^{(W)}$ for all vertex sets $W$ and all $v \in V^W$.
\end{definition}

\begin{example} A $\{0,1\}_2$-coloured continuous hypergraph $G$ is essentially a probability space $(V,\B,\nu)$, together with a measurable subset $G_2$ of $V \times V$, which can be viewed as a continuous analogue of a set of edges on $V$.  In particular, if one takes $V$ to be the unit interval $[0,1]$ with the standard Borel $\sigma$-algebra $\B$ and Lebesgue measure $\nu$, a $\{0,1\}_2$-continuous hypergraph becomes a measurable subset $G_2$ of the unit square.  The sampling map $\overline{G}^{([n])}: [0,1]^n \to \{0,1\}_2^{([n])}$ then maps any $n$-tuple $v_1,\ldots,v_n \in [0,1]$ of ``sampling vertices'' to the directed graph $([n],E)$ on $n$ vertices, with $(i,j) \in E$ if and only if $(v_i,v_j) \in G_2$.  Thus, if one selects a point in $[0,1]^n$ uniformly at random, the image of this point under $\overline{G}^{([n])}$ is a randomly sampled graph of order $n$ from the continuous graph $G$. Note that we do not exclude the diagonal of $V \times V$ from $G_2$, and so we allow continuous hypergraphs to contain loops.
\end{example}

\begin{remark} In the language of category theory, one can view the map $\overline{G}: W \mapsto \overline{G}^{(W)}$ as a \emph{natural transformation} from the contravariant functor $W \mapsto V^W$ to the contravariant functor $W \mapsto K^{(W)}$.  If $G$ obeys $\P$, then the natural transformation $\overline{G}$ factors through the inclusion natural transformation from $\P$ to $K$.
\end{remark}

\begin{example}\label{extend}  Any ordinary hypergraph $G \in K^{(V)}$ on a finite set $V$ can be extended (somewhat arbitrarily) to a continuous hypergraph $\tilde G$, by endowing $V$ with the discrete $\sigma$-algebra $\B$ and the uniform probability measure $\nu$, and defining $\tilde G_j: V^j \to K_j$ to be an arbitrary extension of $G_j: \Hom([j],V) \to K_j$, where we view $\Hom([j],V)$ as a subset of $V^j$ in the obvious manner.  One can view $\tilde G$ as a looped version of the hypergraph $G$.  Observe that if any one of these extensions $\tilde G$ obeys a hereditary hypergraph property $\P$, then $G$ does also.  The framework of continuous hypergraph also allows for placing weights on the vertices by adjusting the probability measure $\nu$ accordingly.
\end{example}

\begin{example}[$0-1$ graphons]  Let $E \subset [0,1]^2$ be a measurable subset of the unit square.  Then the quadruplet $G = ([0,1],\B,\nu,\I(E))$, where $\B$ is the Borel $\sigma$-algebra on the unit interval $[0,1]$, $\nu$ is the uniform measure on $[0,1]$, and $\I(E): [0,1]^2 \to \{0,1\}$ is the indicator function of $E$, is a continuous $\{0,1\}_2$-coloured hypergraph (abusing notation slightly by dropping all the trivial components $G_j$ of the graph $G$ for $j \neq 2$).  If $\P$ is the $\{0,1\}_2$-property of being undirected and triangle-free, then $G$ obeys $\P$ if and only if $E$ is symmetric (thus $(x,y) \in E$ if and only if $(y,x) \in E$) and contains no sets of the form $\{ (x,y),(y,z),(z,x)\}$ for $x,y,z \in [0,1]$.
\end{example}

Now we generalise the local modification rules from Definition \ref{lacma} to more general hypergraphs (including continuous ones).  We give two equivalent definitions of this concept, a concrete one (resembling Definition \ref{lacma}) and a category-theoretic one (resembling Remark \ref{lacma2}):

\begin{definition}[Local modification rule, concrete definition]\label{concmod} Let $K = (K_j)_{j=0}^k$ be a finite palette.  A \emph{local modification rule} is a pair $T = (A, T)$, where $A$ is a finite vertex set, and $T$ is a collection of maps $T_j: K^{(A \uplus [j])} \to K_{=j}^{([j])}$ for $0 \leq j \leq k$ which obey the $\Hom([j],[j])$-equivariance condition
$$ K_{=j}^{(\phi)} \circ T_j = T_j \circ K^{(\id_A \uplus \phi)}$$
for all $\phi \in \Hom([j],[j])$.  Given such a rule, and given a vertex set $V$, we define the modification map $\overline{T}^{(V)}: K^{(A \uplus V)} \to K^{(V)}$ by the formula
$$ \overline{T}^{(V)}(G)_j(\phi) := T_j(K^{(\id_A \uplus \phi)}(G))(\phi)$$
for every vertex set $V$, all $0 \leq j \leq k$, all $G \in K^{(A \uplus V)}$, and all $\phi \in \Hom([j],V)$; the components $\overline{T}^{(V)}(G)_j$ for $j>k$ are of course trivial.  
\end{definition}

\begin{definition}[Local modification rule, categorical definition]\label{locmod}  Let $K$ be a finite palette.  A \emph{local modification rule} is a pair $T = (A, T)$, where $A$ is a finite vertex set, and $T$ is an assignment of a map $\overline{T}^{(V)}: K^{(A \uplus V)} \to K^{(V)}$ for every vertex set $V$ (where $A \uplus V$ denotes the disjoint union of $A$ and $V$), such that the diagram
\begin{equation}\label{natural}
\begin{CD}
K^{(A \uplus V)}                @>{\overline{T}^{(V)}}>>   K^{(V)} \\
@VV{K^{(\id_A \oplus \phi)}}V                             @VV{K^{(\phi)}}V            \\
K^{(A \uplus W)}                @>{\overline{T}^{(W)}}>>   K^{(W)}
\end{CD}
\end{equation}
commutes for any morphism $\phi \in \Hom(W,V)$ between two vertex sets $W,V$.  
\end{definition}

It is not difficult to see that the two definitions are equivalent.  For instance, given a modification rule $(A,T)$ defined by Definition \ref{locmod}, the corresponding maps $T_j$ for Definition \ref{concmod} can be defined by the formula
$$ T_j := \pi_{K \to K_{=j}}^{([j])} \circ \overline{T}^{([j])},$$
where $\pi_{K \to K_{=j}}^{([j])}: K^{([j])} \to K_{=j}^{([j])}$ is the projection map.   In our proofs, we shall adopt a category-theoretic viewpoint and rely on the latter definition rather than the former.  However, for the purpose of understanding the results, the reader may safely ignore the category-theoretic definition.

\begin{remark} The commutative diagram \eqref{natural} is asserting that $\overline{T}$ is a \emph{natural transformation} between the functors $V \mapsto K^{(A \uplus V)}$ and $V \mapsto K^{(V)}$.  It is this natural transformation property that makes the repair rule \emph{local} (and invariant under relabeling of vertices); it implies that the value of a modified edge $T_v(G)_j(\phi)$ for a continuous graph depends only on the edges that involve the vertices $v$ and the vertices of $\phi$, and similarly for the modified edges $T_\phi(G)_j(\psi)$ of finite graphs.
\end{remark}

We now use local modification rules to modify discrete and continuous hypergraphs in order to ensure (or \emph{entail}) certain properties $\P$.

\begin{definition}[Entailment and modification]\label{entail} Let $(A,T)$ be a local modofication rule.  We say that this rule \emph{entails} a $K$-property $\P$ if $\overline{T}^{(V)}( K^{(A \uplus V)} ) \subset \P^{(V)}$ for any vertex set $V$.
\begin{itemize}
\item If $G = (V,\B,\nu,(G_j)_{j=0}^\infty)$ is a continuous $K$-coloured hypergraph, and $v = (v_a)_{ a \in A} \in V^A$ is a collection of vertices in $a$, we define the \emph{modification} $T_v(G) = (V, \B, \nu, (G'_j)_{j=0}^\infty)$ of $G$ to be the continuous $K$-coloured hypergraph given by the requirement that
$$ T_v(G)^{(W)}( w ) = \overline{T}^{(W)}( \overline{G}^{(A \uplus W)}( v, w ) )$$
for all vertex sets $W$ and all $w \in V^W$; one can verify that this requirement uniquely defines a continuous $K$-coloured hypergraph $T_v(G)$.  Note that if $T$ entails $\P$, then $T_v(G)$ obeys $\P$ for every continuous $K$-coloured hypergraph $G$ on a vertex set $V$, and any $v \in V^A$.
\item If $G = (G_j)_{j=0}^\infty$ is a $K$-coloured hypergraph on a vertex set $V$, and $\phi \in \Hom(A,V)$, then we define the \emph{modification} $T_\phi(G)$ of $G$ to be the $K$-coloured hypergraph $G' = (G'_j)_{j=0}^\infty$ on $V \backslash \phi(A)$ defined by the formula
$$ T_\phi(G) = \overline{T}^{(V \backslash \phi(A))}( K^{(\phi \uplus \id_{V \backslash \phi(A)})}(G) ),$$
where $\phi \uplus \id_{V \backslash \phi(A)}: A \uplus (V \backslash \phi(A)) \to V$ is the bijection formed by the direct sum of $\phi: A \to \phi(A)$ and the identity map $\id_{V \backslash\phi(A)}: V \backslash \phi(A) \to V \backslash \phi(A)$.  Again, note that if $T$ entails $\P$, then $T_\phi(G)$ obeys $\P$ for every $K$-coloured hypergraph on a vertex set $V$, and any $\phi \in \Hom(A,V)$.
\end{itemize}
\end{definition}

\begin{example}\label{bipart} Let $K = \{0,1\}_2$, so that $K$-coloured hypergraphs are just directed graphs.  We define the local modification rule $T = (A,T)$ by setting $A := [1] = \{1\}$, and setting $\overline{T}^{(V)}(G) \in K^{(V)}$ for any vertex set $V$ and any directed graph $G \in K^{(A \uplus V)}$ (thus $G$ can be identified with a map $G_2: A \uplus V \times A \uplus V \to \{0,1\}$) to be the collection of all edges $(v,w) \in \Hom([2],V)$ such that $G_2(v,w)=G_2(w,1)=1$ and $G_2(v,1)=0$.  In words, $\overline{T}^{(V)}(G)$ creates a bipartite directed graph from $G$ by deleting all edges from $G$ except those which connect a vertex $V$ which do not have an edge to $1$, to a vertex of $V$ which does have an edge to $1$.  In particular, if $\P$ is the $\{0,1\}_2$-property of being bipartite, then it is clear that $T$ entails $\P$.  If $G = (V,\B,\nu,G_2)$ is a continuous $K$-coloured hypergraph (ignoring the trivial components $G_0,G_1$), and $v_1 \in V$, then the modified continuous graph $T_{v_1}(G) = (V,\B,\nu,G'_2)$ is given by requiring that $G'_2(v,w) = 1$ whenever $G_2(v,w)=G_2(w,v_1)=1$ and $G_2(v,v_1)=0$.  Similarly, if $G = (G_2)$ is a directed graph on a vertex set $V$, and $v_1$ is a vertex in $V$, then the modified directed graph $T_{v_1}(G) = G'_2$ is given by requiring that $G'_2(v,w) = 1$ whenever $G_2(v,w)=G_2(w,v_1)=1$ and $G_2(v,v_1)=0$.
\end{example}

We can now generalise Definition \ref{wlr}:

\begin{definition}[Local repairability]\label{locrep}  Let $K$ be a finite palette of some order $k$, and let $\P$ be a hereditary $K$-property.
\begin{itemize}
\item We say that $\P$ is \emph{strongly locally repairable} if for every $\eps > 0$ there exists a finite set $A$, an $N > 0$, and a real number $\delta > 0$ with the following property: Whenever $G = (V,\B,\nu,(G_j)_{j=0}^k)$ is a continuous $K$-coloured hypergraph which approximately locally obeys $\P$ in the sense that\footnote{We use $\I(E)$ to denote the indicator of an event $E$, thus $\I(E)=1$ when $E$ is true and $\I(E)=0$ otherwise.}
\begin{equation}\label{gp}
 \int_{V^{[N]}} \I\left(\overline{G}^{([N])}(v) \in \P^{([N])}\right)\ d\nu^{[N]}(v) \geq 1-\delta,
\end{equation}
where $\nu^{[N]}$ is the $N$-fold product measure of $\nu$ on $V^{[N]}$, then there exists a local modification rule $T = (A,T)$ that entails $\P$, which does not significantly modify $G$ in the sense that
\begin{equation}\label{intv}
 \int_{V^A} \int_{V^{[k]}} \I\left( \overline{T_v(G)}^{([k])}(w) \neq \overline{G}^{([k])}(w) \right)\ d\nu^A(v) d\nu^{[k]}(w)  \leq \eps.
\end{equation}
\item We say that $\P$ is \emph{weakly locally repairable} if for every $\eps > 0$ there exists a finite set $A$, an integer $N \geq |A|+k$, and a real number $\delta > 0$ with the following property: whenever $G$ is a $K$-coloured hypergraph on a vertex set $V$ with $N \leq |V| < \infty$ which approximately obeys $\P$ in the sense of \eqref{injv}, then there exists a local modification rule $T = (A,T)$ and $\phi \in \Hom(A,V)$ such that $T_\phi(G)$ obeys $\P$, and which is close to $G$ in the sense that
\begin{equation}\label{psieps}
 \frac{1}{\left|\binom{V \backslash \phi(A)}{k}\right|} \left|\left\{ W \in \binom{V \backslash \phi(A)}{k} : T_\phi(G)\downharpoonright_W \neq G\downharpoonright_W \right\}\right|  \leq \eps.
 \end{equation}
\end{itemize}
\end{definition}

\begin{example} Let $\P$ be the $\{0,1\}_2$-property of being a bipartite graph.  The local rule in Example \ref{bipart} entails $\P$, but is not strong enough by itself to show that $\P$ is strongly or weakly locally repairable, because it tends to delete far too many edges to force bipartiteness.  However, one can improve this rule by enlarging the set $A$ and using a rule closer to that discussed in Section \ref{prevsec}; we omit the details.
\end{example}

\begin{remark} Informally, local repairability is the assertion that if a hypergraph locally obeys $\P$ (in the sense that most hypergraphs of order $N$ obtained by randomly sampling $N$ vertices from $V$ will obey $\P$), then there is a modification rule which is guaranteed to produce a new hypergraph which obeys $\P$, and which is also close to the original hypergraph in the sense that most random $k$-element samples of the two hypergraphs will agree.  (Note that this implies automatically implies the same statement for random $j$-element samples for any $j < k$.)

The differences between strong and weak local repairability are that for strong local repairability, one can handle infinite hypergraphs, as well as hypergraphs with loops; one does not need to delete any vertices when repairing the hypergraph; and furthermore, the local modification rule $T$ modifies \emph{all} hypergraphs to obey $\P$, not just the original hypergraph $G$, and the repaired hypergraph is likely to stay close to $G$ for \emph{most} choices of $v \in V^A$, and not just for a \emph{single} $\phi \in \Hom(A,V)$.
\end{remark}

\begin{remark}
Suppose that $\P$ is weakly (or strongly) locally repairable. As stated, the repair algorithm $T$ appearing in the above definition depends on the hypergraph $G$ as well
as on the data $\P$ and $\eps$.  With a bit more effort, one can
show that there exists a repair algorithm $T$ which depends only on
$\P$ and $\eps$, and which works (with high probability) for
\emph{all} hypergraphs (or continuous hypergraphs) $G$ that obey \eqref{gp}.  To see this, observe
that as $A$ does not depend on $G$, the number of possible repair
algorithms $T$ that can arise is bounded (for fixed $\P$ and $\eps$). Thus one can simply try all of these algorithms
in turn on a large random portion of $G$ and verify empirically
whether any of them obey \eqref{intv}, and then use the ``winner'' to
then repair the rest of the hypergraph. We omit the details.
\end{remark}

We make the following simple observations:

\begin{proposition}[Easy implications]\label{easy}  Let $\P$ be a $K$-property for some finite palette $K$.  If $\P$ is strongly locally repairable, then it is weakly locally repairable, and also testable with one-sided error.
\end{proposition}

\begin{proof}(Sketch)  Let $k$ be the order of $K$.  To show that strong local repairability implies weak local repairability, we start with a large finite hypergraph $G$ on at least $N$ vertices obeying \eqref{injv} (for some $N$ and $\delta$ to be chosen later), extend it to a continuous hypergraph $\tilde G$ as in Example \ref{extend}, and apply strong local repairability to obtain a local repair rule $T = (A,T)$ entailing $\P$ and obeying \eqref{intv} with $\eps$ replaced by some slightly smaller quantity $\eps'$ depending on $k$ and $\eps$, and assuming that $N$ and $\delta$ were sufficiently large and small respectively depending on $\eps'$.  If $N$ is large enough, we can use \eqref{intv} and the pigeonhole principle to find $\phi \in \Hom(A,V) \subset V^A$ such that\footnote{Here and in the sequel we use $X \ll Y$ and $Y \gg X$ synonymously with $X = O(Y)$ or $Y = \Omega(X)$ for non-negative $X,Y$; if the implied constant depends on some parameters, we will indicate this by appropriate subscripting.}
$$ \frac{1}{|V^{[k]}|} \left|\left\{ w \in V^{[k]}: \overline{T_\phi(\tilde G)}^{([k])}(w) \neq \overline{\tilde G}^{([k])}(w) \right\}\right|  \ll_{k} \eps'$$
which then implies \eqref{psieps} if $N$ is large enough and $\eps'$ is sufficiently small depending on $k$ and $\eps$.  Also, since $T$ entails $\P$, $T_\phi(G)$ will obey $\P$, and we are done.  A similar argument gives testability with one-sided error, by setting $G'$ to be the hypergraph corresponding to $T_\phi(\tilde G)$ (basically, by reversing Example \ref{extend} and deleting all the loops); we omit the details.
\end{proof}

\begin{remark}\label{easy-rem} It is almost true that weak local repairability implies testability with one-sided error; the one problem is that the hypergraph obtained by weak local repairability was forced to delete a bounded number of vertices.  If one strengthens the notion of weak local repairability to allow $T$ to entail $\P$, rather than merely assume that $T_\phi(G)$ obeys $\P$, then one can easily fix the problem by adding a bounded number of ``dummy'' vertices to $G$ to create a slightly enlarged graph $G'$, so that $T_\phi(G')$ still obeys $\P$ and has the same number of vertices as $G$; we leave the details to the reader.  On the other hand, this strengthened notion of weak local repairability becomes equivalent to the strong notion of local repairability, as one can see by viewing a continuous hypergraph as the limit of a sequence of finite hypergraphs (and using the fact that for fixed $A$, the number of possible modification rules $T$ is finite); we omit the details.  Indeed we do not know any example of a hereditary property which is weakly locally repairable but not strongly locally repairable.
\end{remark}

We can now quickly state our next main theorem.

\begin{theorem}[Every hereditary undirected graph property is locally repairable]\label{lgr} Let $K$ be a finite palette of order at most $2$, and let $\P$ be a hereditary undirected $K$-property.  Then $\P$ is strongly locally repairable (and hence also weakly locally repairable).
\end{theorem}

The proof of Theorem \ref{lgr} follows the Alon-Shapira argument and is given in Section \ref{posi}.

\begin{remark} Theorem \ref{lgr} implies the existence of a probabilistic algorithm that can generate each edge of the graph $G'$ in Theorem \ref{as-thm} in time $O_{\P,\eps}(1)$ (and using $O_{\P,\eps}(1)$ queries to $G$), i.e. in a time bounded by a quantity depending only\footnote{We caution however that our result, which is proven by indirect means, is \emph{ineffective} or \emph{non-uniform} in the sense that we do not provide a way to explicitly compute this bound $O_{\P,\eps}(1)$ given $\P$ and $\eps$.  Indeed, given the discussion in \cite{AloSha}, \cite{AloSha2}, it is extremely likely that the bound here is \emph{uncomputable} from that data in general, even when $\P$ itself is computable; the issue seems to be related to that of solving various halting problems associated to $\P$.  In particular, we have a somewhat subtle distinction: for any \emph{fixed} $\P$, $\eps$, and $G$, the repair algorithm $T$ can be described in a finite (but uncomputable) amount of time, but we do not have an algorithm to \emph{compute} this description from $\P$, $\eps$, and $G$.}  on $\P$ and $\eps$.  In particular, the entire graph $G'$ can be reconstructed in time $O_{\P,\eps}(|V|^2)$ (of course, one needs to query the entire graph $G$ to do this).  Similar remarks apply to Theorems \ref{monotone}, \ref{part} below.
\end{remark}

Another way to contrast local repairability with testability is to observe that Theorem \ref{lgr} also easily implies Ramsey's theorem:

\begin{corollary}[Ramsey's theorem]\label{ramsey}  Let $K$ be a finite palette of order at most $2$ and let $n \geq 1$.  If $N'$ is sufficiently large depending on $K$ and $n$, then for every undirected graph $G \in K^{([N'])}$ there exists a set $W \subset [N']$ with $|W| = n$ such that the induced graph $G\downharpoonright_W \in K^{(W)}$ is monochromatic, or equivalently that $K^{(\phi)}( G\downharpoonright_W ) = G\downharpoonright_W$ for all $\phi \in \Hom(W,W)$.
\end{corollary}

\begin{remark} Ramsey's theorem is of course also true for palettes $K$ of order greater than $2$, but Theorem \ref{lgr} turns out to fail in this case, due to the failure of a generalised version of Ramsey's theorem: see Theorem \ref{negate} below.
\end{remark}

\begin{proof} Let $\P$ be the $K$-property of being undirected and not containing any monochromatic induced sub-hypergraph on $n$ vertices.  This is clearly a hereditary $K$-property, and hence strongly locally repairable by Theorem \ref{lgr}.  On the other hand, it is impossible for any non-empty $K$-coloured continuous graph $G = (V,\B,\nu,G_2)$ to obey this property\footnote{For closely related reasons, it is also impossible to find a local repair rule $T$ which entails $\P$.}, since if $v \in V^n$ is any $n$-tuple with all coordinates equal then $\overline{G}^{([n])}(v)$ is a monochromatic hypergraph on $n$ vertices.  Applying Theorem \ref{lgr} in the contrapositive, we conclude the existence of an $N \geq 1$ and $\delta > 0$ such that
$$ \frac{1}{|V|^N} \left|\left\{ v \in V^{[N]}: \overline{G}^{([N])}(v) \hbox{ obeys } \P \right\}\right| < 1-\delta.$$
On the other hand, if $G$ contained no induced monochromatic sub-hypergraphs on $n$ vertices, the left-hand side would be\footnote{We use subscripts on the $O()$ notation to indicate that the implied constants in that notation depend on the variables in the subscripts.} $1 - O_{N,n,K}(1/|V|)$.  The claim then follows by taking $N'$ sufficiently large depending on $N, n, K, \delta$.
\end{proof}

It does not appear possible to similarly deduce Ramsey's theorem just from Theorem \ref{as-thm}.  One indirect piece of evidence for this claim is that the arguments in \cite{rs2} do not invoke Ramsey-theoretic arguments anywhere, but are still able to obtain Theorem \ref{as-thm}.  On the other hand, the Alon-Shapira arguments used to prove Theorem \ref{as-thm} in the $k=2$ case crucially relies on Ramsey's theorem.  Similarly, our proof of Theorem \ref{lgr} will also invoke Corollary \ref{ramsey} at a key juncture (see Section \ref{tech}).
The arguments used to prove Theorem \ref{lgr} can also be used (after some modification) to establish local repairability of monotone hypergraph properties and partite hypergraph properties.  More precisely, we have the following two results.

\begin{definition}[Monotonicity]\label{mono}  An \emph{ordered finite palette} is a finite palette $K = (K_j)_{j=0}^\infty$, together with a partial ordering $<_j$ on each component $K_j$ which is a \emph{meet-semilattice}, in the sense that any two elements $c_j, c'_j$ in $K_j$ have a unique meet\footnote{We say that $z = x \wedge y$ is the \emph{meet} of two elements $x,y$ of a partially ordered set if $z \leq x,y$, and if $z \geq z'$ for any $z' \leq x,y$.} $c_j \wedge c'_j$; note that this is automatically a commutative and associative operation.  

Now let $K$ be an ordered finite palette and $\P$ a hereditary $K$-property.
\begin{itemize}
\item We say that $\P$ is \emph{monotone} if if given any vertex set $V$ and any $K$-coloured hypergraphs $G \in \P^{(V)}$, any hypergraph $G' \in K^{(V)}$ with the property that $G'_j(\phi) \leq G_j(\phi)$ for all $j \geq 0$ and $\phi \in \Hom([j],V)$, will obey $\P$.  (Informally: ``deleting''  edges (or lowering the colour of edges) will preserve the property $\P$.)
\item We say that $\P$ is \emph{weakly monotone} if given any vertex set $V$ and any $K$-coloured hypergraphs $G,G' \in \P^{(V)}$, the hypergraph $G \wedge G' \in K^{(V)}$ defined by $(G \wedge G')_j(\phi) := G_j(\phi) \wedge G'_j(\phi)$ for all $j \geq 0$ and $\phi \in \Hom([j],V)$, also obeys $\P$.  (Informally, the ``intersection'' (or color-meet) of two hypergraphs obeying $\P$, continues to obey $\P$.)
\end{itemize}
\end{definition}

\begin{example} Suppose we are in the ``boolean'' case where $K = \{0,1\}_k$ is the monochromatic finite palette of some order $k \geq 0$, so that a $K$-coloured hypergraph on a vertex set $V$ can be identified with a set $E \subset \Hom([k],V)$ of morphisms from $[k]$ to $V$.  A hereditary $K$-property $\P$ is then monotone if, given any $E \in \Hom([k],V)$ which obeys $\P$, the hypergraph associated to any subset of $E$ also obeys $\P$.  Similarly, $\P$ is weakly monotone if, given any two $E, E' \subset \Hom([k],V)$ which obey $\P$, the hypergraph associated to $E \cap E'$ also obeys $\P$.  Note that any directed monotone or undirected monotone hypergraph property is weakly monotone.  However, one can easily concoct examples of weakly monotone properties which are not monotone (e.g. the property of being a complete hypergraph is weakly monotone).
\end{example}

\begin{theorem}[Every weakly monotone directed hypergraph property is locally repairable]\label{monotone} Let $K$ be an ordered finite palette, and let $\P$ be a weakly monotone $K$-property.  Then $\P$ is strongly locally repairable (and hence also weakly locally repairable).
\end{theorem}

\begin{definition}[Partiteness]\label{partite}  Let $K$ be a palette of order $k \geq 1$.  If $G \in K^{(V)}$ is a $K$-coloured hypergraph, $0 \leq j \leq k$, and $\phi \in \Hom([j],V)$, we say that $\phi$ is a \emph{partite edge} of $G$ if the map $G_1: V \to K_1$ is injective on $\phi([j])$.  If $G, G' \in K^{(V)}$, we say that $G, G'$ are \emph{partite equivalent} if $G_1 = G'_1$ and if $G_j(\phi) = G'_j(\phi)$ for every $0 \leq j \leq k$ and every partite edge $\phi \in \Hom([j],V)$ of $G$ (and thus of $G'$).  We say that a hereditary $K$-property $\P$ is \emph{partite} if it is preserved under partite equivalence, thus if $G \in \P^{(V)}$ and $G'$ is partite equivalent to $G$, then $G' \in \P^{(V)}$.
\end{definition}

\begin{example}[Tripartite triangle-freeness]  Let $K$ be the finite palette
$K := (\pt, \{1,2,3\}, \{0,1\})$ of order $2$.
Thus a $K$-coloured graph $G \in K^{(V)}$ on a vertex set $V$ can be viewed as a vertex colouring $G_1: V \to \{1,2,3\}$, together with a set $E_2 \subset \Hom([2],V)$ of edges.  Let $\P$ be the $K$-property of being undirected (thus $(v,w) \in E_2$ if and only if $(w,v) \in E_2)$, partite (thus $(v,w) \in E_2$ only if $G_1(v) \neq G_1(w)$), and triangle-free (thus there do not exist $u,v,w \in V$ such that $(u,v), (v,w), (w,u) \in E_2$).  With our definitions, $\P$ is hereditary but is not a partite $K$-property, because it is not preserved under partite equivalent operations, such as adding edges $(v,w)$ within a single vertex colour class $G_1^{-1}(\{i\})$.  However, if we define $\P'$ to be the $K$-property that $G'$ obeys $\P$, where $G'$ is the $K$-coloured graph with the same vertex colouring $G'_1 := G_1$ as $G$, and whose edge set $E'_2$ consists of those edges $(v,w) \in E_2$ for which $G_1(v) \neq G_1(w)$, then $\P'$ is a hereditary partite $K$-property.
\end{example}

\begin{remark}\label{partite-directed} In Remark \ref{directed} we commented that property testing of directed hypergraph properties could not be easily reduced to the property testing of undirected hypergraph properties.  However, in the case of partite properties one can canonically convert directed hypergraphs into undirected hypergraphs in a manner which allows one to transfer property testing results back and forth between the directed and undirected cases.  For instance, given a bipartite directed graph $G = (V,E)$ (so the palette here is $(\pt, \{0,1\},\{0,1\})$), one can lift $G$ to an undirected bipartite $(\pt, \{0,1\}, \{ \emptyset, (0,0), (0,1), (1,0), (1,1)\})$-coloured graph $G'$, by declaring the colour of an undirected edge $\{v_0,v_1\}$ in $G'$, where $v_0$ and $v_1$ are in the $0$-vertex and $1$-vertex classes respectively, to be the ordered pair consisting of the colour of the directed edges $(v_0, v_1)$ and $(v_1,v_0)$ in $G$ respectively (and all edges not connecting a $0$-vertex to a $1$-vertex can be assigned the colour $\emptyset$).  It is then not hard to see that a partite property $\P$ of directed bipartite graphs $G$ can be lifted to an equivalent partite property $\P'$ on undirected bipartite graphs $G'$, and that local testability or repair results for $\P$ are equivalent to those for $\P'$.  More generally, if $K$ is any finite palette and $G \in K^{(V)}$ is a directed $K$-coloured hypergraph, one can create an undirected $K$-coloured hypergraph $G' \in (K')^{(V)}$, where the finite palette $K' = (K'_j)_{j=0}^\infty$ is defined by setting $K'_j := K_j$ for $j =0,1$ and $K'_j := \binom{\Hom([j],K_1) \times K_j}{j!} \cup \{\emptyset\}$ for $j > 1$, by setting $G'_j := G_j$ for $j=0,1$, and setting $G'_j(\phi) := \{ (G_1 \circ \phi \circ \psi, G_j(\phi \circ \psi)): \psi \in \Hom([j],[j]) \}$ when $j \geq 2$ and $\phi$ is a partite edge, and $G'_j(\phi) := \emptyset$ when $j \geq 2$ and $\phi$ is not a partite edge.  Then one can identify each directed partite $K$-property $\P$ with a undirected partite $K'$-property $\P'$, such that $G$ obeys $\P$ if and only if $G'$ obeys $\P'$; we omit the details.
\end{remark}

\begin{theorem}[Every partite hypergraph property is locally repairable]\label{part} Let $K$ be an finite palette of order $k \geq 1$, and let $\P$ be a partite hereditary $K$-property.  Then $\P$ is strongly locally repairable (and hence also weakly locally repairable).
\end{theorem}

\begin{remark} A similar result to Theorem \ref{part} implicitly appears in \cite{ishi}.  It is also quite likely that Theorem \ref{monotone} can be deduced from the methods in \cite{ars}, although this is not done explicitly in that paper.
\end{remark}

Theorems \ref{rs-thm-dir}, \ref{lgr}, \ref{monotone}, and \ref{part} will all be proven in Section \ref{posi}.  The arguments have many features in common (and in fact share many key propositions) and so will be proven concurrently.  To do this, we will use a version of the hypergraph correspondence principle \cite{Tao3}, combined with a structure theorem \cite{Aus1} for exchangeable random hypergraphs, to convert these problems into an infinitary\footnote{There are a number of advantages in working in the infinitary framework.  One is that there are fewer epsilons that one needs to manage in the argument.  Another is that one gains access to a number of useful infinitary tools, such as the Lebesgue dominated convergence theorem, Littlewood's principle that measurable functions are almost continuous, and the Lebesgue-Radon-Nikodym theorem.  While each of these infinitary tools does have some sort of finitary analogue, these analogues are significantly messier to use (and are less well known) than their infinitary counterparts.} one concerning the testability and repairability of certain ``infinitely regular'' exchangeable random hypergraphs (or more precisely, for exchangeable ``recipes'' for producing such hypergraphs, whose palettes are sub-Cantor sets rather than finite sets).  This conversion, which is completed in Section \ref{reduce-sec} is analogous to the exploitation of graph and hypergraph limits in \cite{LovSze2}, \cite{EleSze}, with the infinitely regular exchangeable random hypergraphs being closely related to the graphons and hypergraphons from those papers.

The (infinitary versions of) three local repairability results (Theorem \ref{lgr}, \ref{monotone}, and \ref{part}) will then be deduced from a single ``non-exchangeable'' local repairability result, Proposition \ref{repair}, in Section \ref{tech}.  It is at this stage that a certain amount of Ramsey theory is needed, and assumptions such as undirectedness, monotonicity, or partiteness become crucial.  On the other hand, the result in Proposition \ref{repair} does not require any Ramsey theory, and works for arbitrary hereditary properties.

Proposition \ref{repair}, as well (the infinitary version of) Theorem \ref{rs-thm-dir}, is then deduced from two discretisation results, Propositions \ref{disc-ident} and \ref{disc-ident2}, which construct certain discretisation transformations from continuous palettes to discrete palettes that converge in certain technical senses to the identity as the discrete palette becomes increasingly fine.  These propositions form the heart of the paper and are proven in Sections \ref{disc-sec}, \ref{disc2-sec}.  Proposition \ref{disc-ident}, which underlies the local testability result in Theorem \ref{rs-thm-dir} (and is also used in the proof of Proposition \ref{repair}) follows the R\"odl-Schacht approach and is relatively easy, whereas Proposition \ref{disc-ident2}, which is needed only for the repairability results, uses the Alon-Shapira method and is significantly more technical due to the breakdown of independence caused by ``indistinguishable'' edges\footnote{In the setting of \cite{AloSha2}, this corresponds to the difficulty of repairing edges that connect a single cell in a Szemer\'edi partition to itself.  Once one considers (not necessarily undirected) hypergraphs of higher order, more complicated forms of indistinguishability also appear.}.

\subsection{New negative results}\label{subs:negative}

The above positive results are fairly unsurprising, given the prior work in this direction such as \cite{AloSha2}, \cite{rs2}, and \cite{ishi}.  On the other hand, the following negative results seem to be somewhat more unexpected.

\begin{theorem}[Negative results]\label{negate}
\begin{itemize}
\item[(a)] (Directed graph properties are not locally repairable) There exists a hereditary $\{0,1\}_2$-property which is not weakly locally repairable.
\item[(b)] (Undirected $\leq 3$-uniform hypergraph properties are not locally repairable) There exists a hereditary undirected $(\pt,\{0,1\},\{0,1\},\{0,1\})$-property which is not weakly locally repairable.
\item[(c)] (Undirected $3$-uniform hypergraph properties are not locally repairable)  There exists a hereditary undirected $\{0,1\}_3$-property which is not weakly locally repairable.
\end{itemize}
\end{theorem}

\begin{remark} Combining this theorem with Theorem \ref{rs-thm-dir} we see that there exist hereditary undirected hypergraph properties $\P$ which are testable with one-sided error, but not weakly or strongly locally repairable.  Informally, what this means is that for hypergraphs $G$ which almost obey such properties $\P$, there do exist nearby hypergraphs $G'$ which genuinely obey $\P$, but such hypergraphs cannot be obtained from $G$ by purely local modifications.  We will make this more precise in Section \ref{negchap}, when we prove more refined versions of Theorem \ref{negate}.
\end{remark}

\begin{remark}  There are analogous results\footnote{We are indebted to Luca Trevisan for this remark.} in the coding theory literature.  For instance, in \cite{GS} one finds constructions of locally testable codes which map messages of length $k$ to strings of length $k^{1+o(1)}$, but such codes cannot be locally correctable due to the lower bound results in \cite{katz}.
\end{remark}

We prove Theorem \ref{negate} in Section \ref{negchap}.  For part (a), the directed graph property is actually very simple\footnote{This example is of course closely related to the example of the \emph{half-graph}, which is a familiar counterexample to many overly strong assertions about graph regularity or graph property testing.} - it is the property that a directed graph determines a total ordering on $V$.  The theorem is thus asserting that a lightly corrupted total ordering on an extremely large vertex set cannot be ``cleaned up'' by a purely local algorithm.  The failure in (a) can ultimately be traced back to the simple fact that directed graphs do not obey the Ramsey theorem (which in turn reflects the basic fact that the two directed edges connecting two vertices $v$ and $w$ may well have distinct colours).  Parts (b) and (c) are derived from the counterexample in (a) and some \emph{ad hoc} combinatorial constructions, which ``encode'' the property of being a directed graph as a $\leq 3$-uniform undirected property, and then as a $3$-uniform undirected property.  It is somewhat surprising that one has failure of local repairability in these undirected cases, since Ramsey's theorem is known to be true for hypergraphs.  The problem is rather subtle, and lies in the fact that in the $3$-uniform case, Ramsey's theorem fails for a certain generalisation of a hypergraph known as a \emph{hypergraphon}, in which the colour of a given $3$-uniform edge is not completely determined by its three vertices, but is also dependent on the colour of the $\binom{3}{2}$ $2$-uniform edges between those vertices, which are in turn not completely determined by the vertices themselves.

\begin{remark}  In \cite{KohNagRod}, a positive property testing result for $3$-uniform hypergraphs was proven in the case that $\P$ was the $\{0,1\}_3$-property of not containing a fixed hypergraph as an induced subhypergraph.  This argument relied on Ramsey theory and it seems likely that the repaired hypergraph $G'$ given by this argument could be generated by a local modification rule, though we were unable to fully verify that the arguments in \cite{KohNagRod} would yield this conclusion.  If this is the case, it illustrates an interesting contrast with Theorem \ref{negate}(c), in that arbitrary hereditary properties can in fact behave differently from the properties formed from forbidding a single hypergraph.  Unsurprisingly, the counterexample for local repair of $3$-uniform hypergraphs can be modified to also be a counterexample for local repair of $k$-uniform hypergraphs for any $k \geq 3$, but we will not detail this here.
\end{remark}

\subsection{Acknowledgements}

The second author is supported by NSF grant CCF-0649473 and a grant from the MacArthur Foundation.  The authors thank Asaf Shapira for bringing these problems to our attention and for encouragement, and also thank Noga Alon, Oded Goldreich, Luca Trevisan, and Avi Wigderson for helpful conversations, especially with regard to references and terminology.  We are especially indebted to the anonymous referees for their careful reading and useful suggestions for improving the exposition of the paper.

\subsection{Summary of notation}

For the readers convenience we summarise some of the key notation used in this paper.

The cardinality of a finite set $E$ is denoted $|E|$, and we write $\binom{V}{j} := \{ e \subset V: |e|=j\}$.  For any positive integer $N$, we write $[N] := \{1,\ldots,N\}$.  For any event $E$, we write $\I(E)$ for the indicator function of $E$.  The injections $\phi$ from $V$ to $W$ are denoted $\Hom(V,W)$.  The notation $X \ll Y$, $Y \gg X$, $X=O(Y)$, or $Y = \Omega(X)$ is used to denote $X \leq CY$ for some absolute constant $C$; if $C$ needs to depend on some additional parameters such as $\eps$, we will denote this by subscripting, e.g. $X \ll_\eps Y$ or $X=O_\eps(Y)$.

Hypergraphs and pullback maps $K^{(\phi)}$ are defined in Definition \ref{hyperdef}.  Hereditary properties are defined in Definition \ref{hered}.  Testability is defined in Definition \ref{Testdef}, and local repairability is defined in Definition \ref{locrep}, after introducing the notions of a continuous hypergraph (Definition \ref{contmap}), a local modification rule (Definition \ref{concmod} or \ref{locmod}), and entailment and modification (Definition \ref{entail}).

In Appendix \ref{prob} a number of key probabilistic concepts are defined, including the conditioning $(\mu|E)$ of a probability measure to an event of positive probability, the notion of a probability kernel $P: Y \rightsquigarrow X$, and the composition of $P \circ Q$ of two such kernels.

\section{Proofs of negative results}\label{negchap}

We begin with the proofs of the various counterexamples to local repairability in Theorem \ref{negate}.  The material here is largely independent of those of the positive results, which will be given in Section \ref{posi}.

\subsection{The counterexample for directed graphs}\label{lrs}

In this section we construct a counterexample that will demonstrate part (a) of Theorem \ref{negate}.  In this section we set $K := \{0,1\}_2$ and $k := 2$.  Note in this case that we can identify a $K$-coloured hypergraph $G$ on a vertex set $V$ with a \emph{directed graph} $G = (V,<_G)$, where $<_G$ is a binary relation $<_G: V \times V \to \{\hbox{true}, \hbox{false}\}$ on $V$.  We let $\P$ be the $\{0,1\}_2$-property that $<_G$ is a total ordering, then this is clearly a hereditary $K$-property.  It will suffice to show that $\P$ is not weakly locally repairable.

In order to illustrate some of the ideas involved, let us first demonstrate the much simpler fact that $\P$ is not \emph{strongly} locally repairable.  Consider the continuous $K$-coloured hypergraph $G$ in which $V$ is the unit interval $[0,1]$ with the Borel $\sigma$-algebra $\B$ and Lebesgue measure $\mu$, and $<_G$ is the usual ordering relation on $[0,1]$.  Then we certainly have \eqref{gp}; in fact we can take $\delta=0$ in this case.  On the other hand, it is impossible to repair $G$ to a new continuous hypergraph $G'$ that obeys $\P$, because if $W$ is any finite set with at least two elements, then $\overline{G'}^{(W)}(v)$ cannot obey $\P$ whenever $v$ has a repeated coefficent (thus $v_w = v_{w'}$ for some for some distinct $w,w' \in W$), since the statements $w <_{\overline{G'}^{(W)}(v)} w'$ and $w' <_{\overline{G'}^{(W)}(v)} w$ would have the same truth value, which is inconsistent with $\P$.  Thus $\P$ is not strongly locally repairable.

Now we disprove weak local repairability for the same property $\P$.  This counterexample will be so strong that the parameter $\eps$ in Definition \ref{locrep} (and the estimate \eqref{psieps}) will play no role whatsoever.  (However, we will take advantage of \eqref{psieps} for some suitably small $\eps$ when proving parts (b) and (c) of Theorem \ref{negate}.)

Let $A$ be an arbitrary finite non-empty set, let $N > 0$ be an integer, and let $\delta > 0$ be an arbitrary small number, which we can assume to be small compared to $A,N$.  Let $\sigma > 0$ be an even smaller number (depending on these parameters) to be chosen later, and then let $M$ be an enormous number (depending on all previous parameters), again to be chosen later.  We set $V := [M]$.

To prove Theorem \ref{negate}(a), it will suffice to construct a directed graph $G = (V,<_G)$ obeying \eqref{injv} for which there does \emph{not} exist a local modification rule $T = (A,T)$ and $\phi \in \Hom(A,V)$ such that the repaired graph $T_{\phi}(G)$ obeys $\P$ (note that by construction, our counterexample $V$ can be larger than any specified size).  Our construction will be probabilistic in nature.

To define $G$, we first define an ``uncorrupted'' directed graph $G^{(0)} = (V, <_{G^{(0)}})$ by letting $<_{G^{(0)}} = <$ be the standard total ordering on $V=[M]$, thus $G^{(0)}$ obeys $\P$.  Now let $G = ([M], <_G)$ be a corrupted version of $G^{(0)}$, in which for any $(v,w) \in \Hom([2],[M])$, the statements $v <_G w$ and $v < w$ have the same truth value with probability $1-\sigma$ and the opposite truth value with probability $\sigma$, with these events being independent as $(v,w)$ varies.

Since the uncorrupted $G^{(0)}$ obeys $\P$, and $G$ is a random corruption of $G^{(0)}$, it is easy to see that for each fixed morphism $\phi \in \Hom([N],V)$, that $K^{(\phi)}( G )$ will obey $\P$ with probability $1 - O_N(\sigma)$.  By the first moment method and linearity of expectation, we conclude that \eqref{injv} holds with probability $1 - O_{N,\delta}(\sigma)$.  Let us now condition on the event that \eqref{injv} holds.

Now suppose for contradiction that there exists a local modification rule $T = (A,T)$ and $\phi \in \Hom(A,V)$ such that the repaired graph $G' = (V \backslash \phi(A), <') := T_{\phi}(G)$ obeys $\P$.

Let us say that two distinct vertices $v_1, v_2 \in V \backslash \phi(A)$ are \emph{indistinguishable} if the graph $K^{(\phi \uplus (v_1,v_2))}(G) \in K^{(A \uplus \{1,2\})}$ is invariant under permutation of the $1$ and $2$ indices; more explicitly, $v_1, v_2$ are indistinguishable whenever one has the symmetries
$$ \I( v_1 <_G a ) = \I( v_2 <_G a )$$
and
$$ \I( a <_G v_1 ) = \I( a <_G v_2 )$$
for all $a \in A$, as well as the symmetry
$$ \I( v_1 <_G v_2 ) = \I( v_2 <_G v_1 ).$$
Note that if an indistinguishable pair $v_1, v_2$ of vertices exists, then by \eqref{natural} (applied to the map from $V$ to itself interchanging $v_1$ and $v_2$) the statements $v_1 <_{G'} v_2$ and $v_2 <_{G'} v_1$ have the same truth value, which implies that $G'$ cannot obey $\P$, a contradiction.  Thus, in order to establish Theorem \ref{negate}(a), it will suffice (by the probabilistic method) to show

\begin{lemma}\label{indis}  Suppose $M$ is sufficiently large (depending on $N,\delta,\sigma,A$).  Then with probability $1-O_{A}(\sigma)$, it is true that for every $\phi \in \Hom(A,V)$, there exists at least one pair $(v_1,v_2)$ of distinct but indistinguishable vertices in $V \backslash \phi(A)$.
\end{lemma}

\begin{proof} Let $c > 0$ be a small constant depending on $A$ to be chosen later (actually, one can take $c := 100^{-|A|}$).  Let $B$ be an arbitrary subset of $V$ of cardinality at least $cM$.  We assume $M$ is large enough that $cM > 2$.  Call $B$ \emph{corrupted} if there exists distinct $v_1, v_2 \in B$ such that $v_1 <_G v_2$ and $v_2 <_G v_1$ have the same truth value.  Observe from construction of $G$ for any distinct $v_1, v_2 \in V$ that $v_1 <_G v_2$ and $v_2 <_G v_1$ have the same truth value with probability $\gg \sigma$.  By independence, we conclude that $B$ will be corrupted with probability at least $1 - \exp( - \Omega( \sigma c^2 M^2 ) )$.  On the other hand, the total number of sets $B$ is at most $2^M$.  Also, the total number of choices for $\phi$ can be crudely bounded by $M^{|A|}$.
By the union bound, we conclude that with probability at least $1 - 2^M M^{|A|} \exp( - \Omega( \sigma c^2 M^2 ) )$, \emph{every} set of cardinality at least $cM$ is corrupted, for all choices of $\phi$.
If $M$ is large enough depending on $A, \sigma$, we thus see that this event holds with probability $1-O_{A}(\sigma)$.

Let us condition on the above event, and let $\phi \in \Hom(A,V)$ be arbitrary.  Let $\Omega := 2^A$ denote the power set of $A$.  We can then partition
$$ V = A \cup \bigcup_{U, U' \in \Omega} V_{U,U'}$$
where for each $U \in \Omega$, $V_{U,U'}$ is the set of all $v \in V \backslash A$ such that
$$ U = \{ a \in A: v <_G \phi(a) \} \hbox{ and } U' = \{ a \in A: \phi(a) <_G v \}.$$
The total number of pairs $(U,U')$ is $O_{A}(1)$.  Thus by the pigeonhole principle (and taking $M$ large enough), we can find $U, U'$ such that $|V_{U,U'}| \geq cM$, if $c$ is sufficiently small depending on $A$.  In particular, $V_{U,U'}$ is corrupted and we can find distinct $v_1,v_2 \in V_{U,U'}$ such that $v_1 <_G v_2$ and $v_2 <_G v_1$ have the same truth value.  By construction, we see that $v_1,v_2$ are indistinguishable, and the claim follows.
\end{proof}

The proof of Theorem \ref{negate}(a) is now complete.

\subsubsection{Further remarks}
We close this section with some further remarks about Theorem \ref{negate}(a).  Informally, the above result asserts that there does not exist a repair algorithm to convert a corrupted total ordering $G$ on a large finite set into an exact total ordering $<'$, in which the order relationship of two vertices $v_1, v_2$ of the set is repaired by inspecting the corrupted relationship between $v_1, v_2$ and a bounded number of other vertices, selected in advance.  It is likely that this result can be strengthened to allow for a more adaptive repair algorithm in which the other vertices that one queries need not be selected in advance, and for which the probability of the algorithm successfully obtaining a total ordering is lowered to, say, $2/3$ rather than $1$.  One should also be able to obtain a similar result even if the algorithm is allowed to retain a bounded amount of ``memory'' between repairing one edge and the next.  However, we will not pursue such strengthenings here.

On the other hand, once one relaxes the requirement of locality (or bounded memory), it becomes very easy to repair the corrupted total ordering $G$ used in the above proof to obtain an exact total ordering $<'$, while only modifying a proportion $O(\eps)$ of the edges.  We sketch the details as follows.  Fix $\eps > 0$, let $A = [N']$ for some large integer $N'$, and select $\phi \in \Hom(A,V)$ at random.  With probability $1-O_{A}(\delta)$, the directed graph $K^{(A)}(G)$ is totally ordered; we condition on this event, and then without loss of generality (relabeling $A$ if necessary) we may assume that the total ordering on $K^{(A)}(G)$ is the usual ordering on $A$.

For each $0 \leq i \leq N'$, let $V_i$ be the set of all vertices $v \in V \backslash \phi(A)$ such that
$$ \{ j \in A: i < j \leq N' \} = \{ 1 \leq j \leq N': v <_G \phi(j) \} $$
and
$$
\{ j \in A: 1 \leq j \leq i \} = \{ 1 \leq j \leq N': \phi(j) <_G v \};$$
roughly speaking, $V_i$ is the set of those vertices which $\phi(A)$ "predicts" should lie in the interval between $\phi(i)$ and $\phi(i+1)$.

These sets are clearly disjoint, and using the first moment method one can show that with probability $1 - O_{N',\eps}(\delta)$, these sets cover a proportion $1-O(\eps |V|)$ of the vertices in $V$.  We then define the total order $<'$ by declaring $v_i <' v_j$ whenever $v_i \in V_i, v_j \in V_j$, and $i < j$, and placing an arbitrary total ordering $<'$ on each of the $V_i$ separately, and also completing the total ordering to the complement of $\bigcup_i V_i$ (these are the non-local components of the repair algorithm).  It is not difficult to show that for $\delta$ sufficiently small, and then $M$ sufficiently large, that with probability $1 - O_{A,\eps}(\delta)$, this total order $<'$ will differ from $G$ on only $O(\eps)$ of the edges; we omit the details.  Note that the run time of this algorithm will be linear in the number of edges (i.e. the run time will be $O(|V|^2)$).

\subsection{The counterexample for undirected $\leq 3$-uniform hypergraphs}\label{leq3}

We now prove Theorem \ref{negate}(b).  We fix $k = 3$ and $K = \{ \pt, \{0,1\}, \{0,1\}, \{0,1\} \}$.  Note that a $K$-coloured undirected hypergraph $G$ on a vertex set $V$ can thus be viewed as a quadruplet $G = (V, E_1, E_2, E_3)$, where $E_1 \subset V$ is a set of vertices, $E_2 \subset \binom{V}{2}$ is a set of undirected $2$-edges, and $E_3 \subset \binom{V}{3}$ is a set of undirected $3$-edges.  The basic idea will be to ``encode'' the notion of a total ordering using the undirected data $E_1, E_2, E_3$.

Let us introduce the following notation.  Given a $K$-coloured undirected hypergraph $G$ and vertices $r, b, b' \in V$, we say that
\begin{itemize}
\item $b$ is \emph{$G$-blue} if $\{b\} \in E_1$;
\item $r$ is \emph{$G$-red} if $\{r\} \not \in E_1$;
\item $r$ \emph{$G$-likes} $b$ if $r$ is $G$-red, $b$ is $G$-blue, and $\{r,b\} \in E_2$;
\item $r$ \emph{$G$-prefers} $b$ to $b'$ if $r$ is $G$-red, $b, b'$ are $G$-blue, $r$ $G$-likes $b$, and $r$ does not $G$-like $b$;
\item $r$ \emph{ranks $\{b,b'\}$ $G$-correctly} if $r$ is $G$-red, $b,b'$ are $G$-blue, and $\{r,b,b'\} \in E_3$;
\item We write $b >_{G,r} b'$ if $r$ either (a) $G$-prefers $b$ to $b'$ and ranks $\{b,b'\}$ $G$-correctly, or (b) $G$-prefers $b'$ to $b$ and does not rank $\{b,b'\}$ $G$-correctly;
\item The hypergraph $G$ is \emph{consistently orderable} if there exists a total ordering $>_G$ on $V$ such that $b >_G b'$ whenever $r, b, b'$ are such that $b >_{G,r} b'$.
\end{itemize}

We let $\P$ be the $K$-property of being undirected and consistently orderable.  One easily verifies that $\P$ is a hereditary undirected $K$-coloured hypergraph property.  To show Theorem \ref{negate}(b), it suffices to show that $\P$ is not weakly locally repairable.

Let $\eps > 0$ be a small absolute constant (one could take $\eps = \frac{1}{1000}$ for concreteness),
let $A$ be an arbitrary finite non-empty set, let $N > 0$ be an integer, and let $\delta > 0$ be an arbitrary small number, which we can assume to be small compared to $A,N$.  Let $\sigma > 0$ be an even smaller number (depending on these parameters) to be chosen later, and then let $M$ be an enormous number (depending on all previous parameters), again to be chosen later.  We set $V := [M]$.

To prove Theorem \ref{negate}(b), it will suffice to construct a $K$-coloured undirected graph $G = (V,E_1,E_2,E_3)$ obeying \eqref{injv}, for which there does \emph{not} exist a local modification rule $T = (A,T)$ and $\phi \in \Hom(A,V)$ such that the repaired hypergraph $T_{\phi}(G)$ obeys $\P$ and \eqref{psieps}. (Again, note that by construction that $V$ can be made larger than any specified number.)

As before, to define $G$ we first define a (random) ``uncorrupted'' $K$-coloured hypergraph
$$G^{(0)} = (V, E^{(0)}_1, E^{(0)}_2, E^{(0)}_3)$$
by the following construction:

\begin{itemize}
\item $E^{(0)}_1 := [M/2]$ (thus vertices between $1$ and $M/2$ are $G^{(0)}$-blue, and vertices between $M/2+1$ and $M$ are $G^{(0)}$-red);
\item $E^{(0)}_2$ is a random graph on $V$, with each edge $\{v_1,v_2\}$ lying in $E^{(0)}_2$ with a probability of $1/2$, with these events being jointly independent.  (Thus, a given $G^{(0)}$-red vertex will $G^{(0)}$-like a given $G^{(0)}$-blue vertex with a probability of $1/2$, independently of all other instances of the $G^{(0)}$-like relation.)
\item $E^{(0)}_3$ is the set of unordered triples $\{r,b,b'\}$ such that $r$ is $G^{(0)}$-red, $b, b'$ are $G^{(0)}$-blue, and one of the following statements hold:
\begin{itemize}
\item[(i)] $r$ $G^{(0)}$-likes both $b$ and $b'$;
\item[(ii)] $r$ does not $G^{(0)}$-like either $b$ or $b'$;
\item[(iii)] $r$ $G^{(0)}$-prefers $b$ to $b'$, and $b > b'$.
\end{itemize}
\end{itemize}

\begin{figure}[tb]
\centerline{\psfig{figure=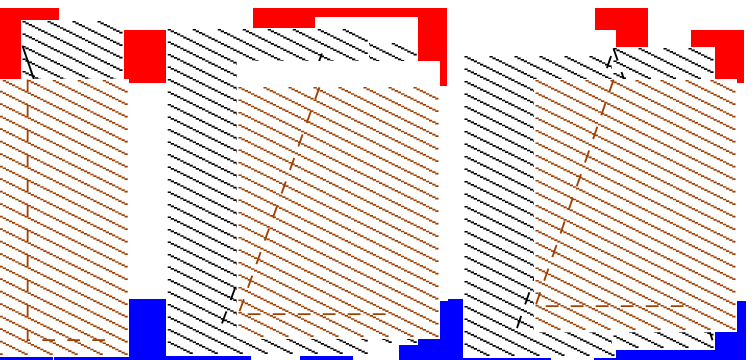}}
\caption{The three types of triples (indicated by shaded triangles) which lie in $E^{(0)}_3$.  Solid lines indicate edges in $E_2 = E^{(0)}_2$, while dashed lines indicate edges not in $E_2 = E^{(0)}_2$.  The vertices on the top row are red, while the bottom vertices are blue; the blue points are ordered so that the larger points are on the right.}
\label{fig1}
\end{figure}

It is not hard to verify that $G^{(0)}$ is consistently orderable (with $>_{G^{(0)}}$ being the usual ordering $>$ on $[M]$) and so obeys $\P$.

Next, we define the ``corrupted'' $K$-coloured undirected hypergraph
$$G = (V, E) = (V, E_1, E_2, E_3)$$
as follows:

\begin{itemize}
\item $V = [M]$;
\item $E_j = E^{(0)}_j$ for $j=1,2$ (thus, $G$ and $G^{(0)}$ have the same notions of red, blue, like, and prefer).
\item For each $e \in \binom{V}{3}$, the statements $e \in E^{(0)}_3$ and $e \in E_3$ have the same truth value with probability $1-\sigma$, and have opposite truth value with probability $\sigma$, independently of each other and of the random graph $E^{(0)}_2$.  (Thus the relations $>_{G,r}$ will be a slight corruption of $>_{G^{(0)},r}$.)
\end{itemize}

Since $G^{(0)}$ obeys $\P$, we can use the first moment method as in the preceding section to conclude that \eqref{injv} holds with probability $1 - O_{N,\delta}(\sigma)$.  Let us now condition on the event that \eqref{injv} holds.

Suppose for contradiction that there exists a local modification rule $T = (A,T)$ and a morphism $\phi: A \to V$ such that the repaired hypergraph $G' = (V \backslash \phi(A), E'_1,E'_2,E'_3) := T_{\phi}(G)$ obeys $\P$ and \eqref{psieps}.  From \eqref{psieps} we see in particular that
\begin{equation}\label{lowcorrupt}
|E'_j \Delta E_j| \ll \eps M^j
\end{equation}
for $j=1,2,3$, where $\Delta$ denotes the symmetric difference operator.

Fix $T,\phi$.  Call a quadruplet $(r_1,r_2,b_1,b_2)$ of distinct vertices in $V \backslash \{\phi(A)\}$ \emph{inconsistent} (relative to $T$ and $\phi$) if the following properties hold:

\begin{itemize}
\item[(i)] $r_1,r_2$ are both $G$-red and $G'$-red, and $b_1,b_2$ are both $G$-blue and $G'$-blue.
\item[(ii)] $r_1$ $G'$-prefers $b_1$ to $b_2$, and $r_2$ $G'$-prefers $b_2$ to $b_1$.
\item[(iii)]  The undirected hypergraph  $K^{(\phi \uplus (r_1,r_2,b_1,b_2))}(G) \in K^{(A \uplus [4])}$ is invariant under the morphism $\id_A \oplus (2,1,4,3) \in \Hom(A \cup [4], A \cup [4])$, where $(2,1,4,3) \in \Hom([4],[4])$ is the permutation which switches $1$ and $2$, and also switches $3$ and $4$.   More explicitly, for any $a \in A$, we have the $E_2$ symmetries $\I( \{ b_1, \phi(a) \} \in E_2 ) = \I( \{ b_2, \phi(a) \} \in E_2 )$ and $\I( \{ r_1, \phi(a) \} \in E_2 ) = \I( \{ r_2, \phi(a) \} \in E_2 )$, as well as the $E_3$ symmetries
\begin{equation}\label{symmetry}
\begin{split}
\I( \{r_1,r_2,b_1\} \in E_3 ) &= \I( \{r_1,r_2,b_2\} \in E_3 ) \\
\I( \{b_1,b_2,r_1\} \in E_3 ) &= \I( \{b_1,b_2,r_2\} \in E_3 ) \\
\I( \{r_1,b_1,\phi(a)\} \in E_3 ) &= \I( \{r_2,b_2,\phi(a)\} \in E_3 ) \hbox{ for all } a \in A\\
\I( \{r_1,b_2,\phi(a)\} \in E_3 ) &= \I( \{r_2,b_1,\phi(a)\} \in E_3 ) \hbox{ for all } a \in A\\
\I( \{r_1,\phi(a),\phi(a')\} \in E_3 ) &= \I( \{r_2,\phi(a),\phi(a')\} \in E_3 ) \hbox{ for all } \{a,a'\} \in \binom{A}{2}\\
\I( \{b_1,\phi(a),\phi(a')\} \in E_3 ) &= \I( \{b_2,\phi(a),\phi(a')\} \in E_3 ) \hbox{ for all } \{a,a'\} \in \binom{A}{2}.
\end{split}
\end{equation}
\end{itemize}

\begin{figure}[tb]
\centerline{\psfig{figure=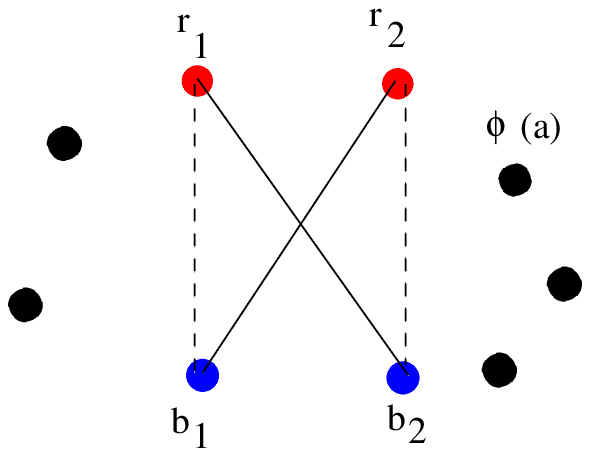}}
\caption{A partial depiction of an inconsistent quadruple $(r_1,r_2,b_1,b_2)$, surrounded by a number of vertices $\phi(a)$ with $a \in A$.  The connectivity between $(r_1,r_2,b_1,b_2)$ and $\phi(A)$ needs to be symmetric with respect to the ``reflection map'' $\id_A \oplus (2,1,4,3)$ which swaps $r_1$ and $r_2$, and swaps $b_1$ and $b_2$, but leaves the vertices in $\phi(A)$ unchanged.}
\label{fig2}
\end{figure}

Observe that if $(r_1,r_2,b_1,b_2)$ are inconsistent, then from properties (iii) and Definition \ref{locmod} we conclude that
$$ \I( \{ r_1, b_1, b_2 \} \in E'_3 ) = \I( \{ r_2, b_2, b_1 \} \in E'_3 ).$$
By properties (i) and (ii), this implies either that $b_1 <_{G',r_1} b_2$ and $b_2 <_{G',r_2} b_1$ are both true, or that $b_2 <_{G',r_1} b_1$ and $b_1 <_{G',r_2} b_2$ are both true.  But this implies that $G'$ is not consistently orderable and thus does not obey $\P$, a contradiction.  Thus to conclude the proof of Theorem \ref{negate}(b), it suffices to show

\begin{lemma}\label{triple} Suppose $\eps > 0$ is sufficiently small, and $M$ is sufficiently large (depending on $N,\sigma,A,\eps$).  Then with probability $1-O_{A,\eps}(\sigma)$, it is true that for all morphisms $\phi \in \Hom(A,V)$ and all local modification rules $T$ obeying \eqref{lowcorrupt}, there exists at least one quadruplet $(r_1,r_2,b_1,b_2)$ of inconsistent vertices in $V \backslash \phi(A)$.
\end{lemma}

\begin{proof}
Let $c > 0$ be a small number depending on $\eps, A$ to be chosen later.
Recall that the $2$-uniform graph $E_2 \subset \binom{V}{2}$ was selected to be a random graph on $V = [M]$, with edge density $1/2$.  By standard arguments (similar\footnote{In other words, one shows that \eqref{regular} holds for each pair $X, Y$ with super-exponentially high probability $1 - \exp(\Omega_c(|V|^2))$, and then applies the union bound.  See also \cite{wilson} for a proof that random graphs are regular.} to that used to prove Lemma \ref{indis}), we thus see that if $M$ is sufficiently large depending on $c, \sigma$, with probability $1 - O(\sigma)$, the graph $E_2$ is \emph{$c$-regular} in the sense that
\begin{equation}\label{regular}
|\{ (a,b) \in X \times Y: \{a,b\} \in E_2 \}| = ( \frac{1}{2} + O(c) ) |X| |Y|
\end{equation}
for all disjoint $X, Y \subset V$ with cardinality $|X|, |Y| \geq c |V|$.  Let us now condition on the event that we have this $c$-regularity, and freeze $E_2$ (and hence $G^{(0)}$).

Next, by paying a factor of $M^{|A|}$ in all future probability upper bounds, we may freeze the morphisms $\phi$. The total number of possible modification rules $T$ is clearly $O_A(1)$, so by paying this factor as well we may also freeze $T$.

We now freeze the set $E_3 \backslash \binom{ V \backslash \phi(A) }{3}$, which describes all the edges of $E_3$ which contain at least one vertex from $\phi(A)$.  Now that we have frozen these edges, as well as $E_2$ and $T$, we see from Definition \ref{locmod} that $E'_1$ and $E'_2$ are also frozen.

The only randomness that remains after all this freezing comes from the random variables $\I(e \in E_3 \Delta E^{(0)}_3)$ for $e \in \binom{ V \backslash \phi(A) }{3}$, which are jointly independent (even after all the freezing) and equal $1$ with probability $\sigma$ each.  From \eqref{natural} we conclude that if $e \in \binom{V}{3}$ intersects $\phi(A)$ then the quantity $\I(e \in E'_3)$ is now deterministic, whereas if $e$ does not intersect $\phi(A)$ then the quantity $\I(e \in E'_3)$ depends only on the quantity $\I(e \in E_3 \Delta E^{(0)}_3)$ (as well as all the frozen data, of course).

Since $E'_1$ and $E'_2$ are frozen, we may condition on the event that \eqref{lowcorrupt} holds for $j=1,2$ without difficulty.  (We will not attempt to condition on the event that \eqref{lowcorrupt} holds for $j=3$, because this creates the technical problem that such a conditioning will disrupt the joint independence of the events $e \in E_3 \Delta E^{(0)}_3$, which we will need to exploit later.)

Let $V_R$ denote the set of vertices in $V \backslash \phi(A)$ which are both $G$-red and $G'$-red, and similarly let $V_B$ denote the set of vertices in $V \backslash \phi(A)$ which are both $G$-blue and $G'$-blue.  From \eqref{lowcorrupt} for $j=1$ we have
\begin{equation}\label{vrb}
|V_R|, |V_B| \geq M/4
\end{equation}
if $\eps$ is small enough.

Let $E^*_2 \subset V_R \times V_B$ be the set of all pairs $(r,b) \in V_R \times V_B$ such that $\{r,b\} \in E_2 \Delta E'_2$.  From \eqref{lowcorrupt} for $j=2$ and \eqref{vrb}, we have
\begin{equation}\label{vrbb}
 |E^*_2| \ll \eps |V_R| |V_B|.
\end{equation}

Let $\Omega = 2^{A}$ be the power set of $A$.  If $U_R \in \Omega$, define $V_{R,U_R}$ to be the set of all vertices $r \in V_R$ such that
$$ U_R = \{ a \in A: \{\phi(a),r\} \in E_2 \}.$$
Similarly, for any $U_B \in \Omega$, define $V_{B,U_B}$ to be the set of all $b \in V_B$ such that
$$ U_B = \{ a \in A: \{\phi(a),b\} \in E_2 \}.$$
Then we have the partitions
$$ V_R = \bigcup_{U_R \in \Omega} V_{R,U_R}; \quad V_B = \bigcup_{U_B \in \Omega} V_{B,U_B}$$
and thus
$$ V_R \times V_B = \bigcup_{U_R, U_B \in \Omega} V_{R,U_R} \times V_{B,U_B}.$$
The number of pairs $(U_R, U_B)$ is $O_{A}(1)$.  By the pigeonhole principle (first discarding all small pairs $V_{R,U_R} \times V_{B,U_B}$) we can choose a pair $(U_R,U_B)$ such that
\begin{equation}\label{verb}
|V_{R,U_R}|, |V_{B,U_B}| \gg_{\eps,A} M
\end{equation}
and
\begin{equation}\label{es2}
 |E^*_2 \cap (V_{R,U_R} \times V_{B,U_B})| \ll \eps |V_{R,U_R}| |V_{B,U_B}|.
 \end{equation}
Fix this pair $(U_R, U_B)$ (if there are multiple pairs available, choose one arbitrarily).

By \eqref{regular} and standard ``counting lemma'' arguments (see e.g. \cite{wilson}), we see that
there exist $\gg |V_{R,U_R}|^2 |V_{B,U_B}|^2$ quadruplets $(r_1,r_2,b_1,b_2)$ with $r_1, r_2 \in V_{R,U_R}$ and $b_1, b_2 \in V_{B,U_B}$ such that $r$ $G$-prefers $b_1$ to $b_2$, and $r_2$ $G$-prefers $b_2$ to $b_1$.  In view of \eqref{es2}, we conclude (if $\eps$ is small enough) that the same assertion holds with ``$G$-prefers'' replaced by ``$G'$-prefers''.

Call a quadruplet $(r_1,r_2,b_1,b_2)$ \emph{admissible} if it is of the above form, thus $r_1, r_2 \in V_{R,U_R}$ and $b_1, b_2 \in V_{B,U_B}$ such that $r_1$ $G'$-prefers $b$ to $b_2$, and $r_2$ $G'$-prefers $b_2$ to $b_1$.  From \eqref{verb} we thus see that there are $\gg_{\eps,A} M^4$ admissible quadruplets.

From chasing all the definitions, we see that if an admissible quadruplet $(r_1,r_2,b_1,b_2)$ obeys \eqref{symmetry}, then it is inconsistent.  Thus, it will suffice to upper bound the probability that no admissible quadruplet obeys \eqref{symmetry} for any choice of $\phi$.

Since $E_2$ and $E'_2$ are already frozen, so are the set of admissible quadruplets $(r_1,r_2,b_1,b_2)$.  Observe from construction of $E_3$ that for any admissible quadruplet $(r_1,r_2,b_1,b_2)$, the probability that this quadruplet obeys\footnote{Note from construction that only the first two conditions in \eqref{symmetry} are in doubt; the remaining conditions, which involve at least one element from $\phi(A)$, are automatic due to $r_1, r_2$ and $b_1, b_2$ lying in the same cells $V_{R,U_R}$ and $V_{B,U_B}$ respectively.} \eqref{symmetry} is $\Omega_{\sigma,A}(1)$, and thus the probability that it does \emph{not} obey \eqref{symmetry} is $\exp( - \Omega_{\sigma,A}(1) )$.  Furthermore, the events that a family of quadruplets do not obey \eqref{symmetry} will be jointly independent as long as no two of these quadruplets share a vertex in common (recall that we are freezing all the edges of $E_3$ which intersect $\phi(A)$).  Since there are $\gg_{\eps,A} M^4$ admissible quadruplets, an easy greedy algorithm argument allows us to find $\gg_{\eps,A} M$ admissible quadruplets for which no two share three vertices in common.  Thus the probability that no admissible quadruplet is corrupted is at most $\exp( - \Omega_{\sigma,\eps,A}(M) )$.  Combining this with our previous factors of $M^{|A|}$ and $O_{A}(1)$ introduced earlier, we obtain the claim if $M$ is sufficiently large.
\end{proof}

The proof of Theorem \ref{negate}(b) is now complete.

\subsection{The counterexample for undirected $3$-uniform hypergraphs}\label{eq3}

We now adapt the methods of the previous section to prove Theorem \ref{negate}(c).  The main challenge is to find analogues of $G^{(0)}$ and $\P$ in the $3$-uniform setting rather than the $\leq 3$-uniform setting.  This will be done in a rather artificial and \emph{ad hoc} fashion, encoding a $\leq 3$-uniform hypergraph property in a $3$-uniform one.

We fix $k = 3$ and $K = \{ 0,1\}_3$.  Note that a $K$-coloured undirected hypergraph $G$ on a vertex set $V$ can thus be viewed as a pair $G = (V, E)$, where $E_3 \subset \binom{V}{3}$ is a set of $3$-edges.

In order to motivate the property $\P$ that we will need here, we first construct the uncorrupted $K$-coloured hypergraph $G^{(1)} = (V, E^{(1)})$ which will play the role of $G^{(0)}$ in the previous section.

Let $M$ be a large integer.  Then we define the $(\pt,\{0,1\},\{0,1\},\{0,1\})$-coloured undirected hypergraph $G^{(0)} = ([M], E^{(0)}_1, E^{(0)}_2, E^{(0)}_3)$ as in the previous section.  We then define the notions of ``red'', ``blue'', ``likes'', ``prefers'', ``ranks correctly'' as before (dropping the $G^{(0)}$ prefix).  We then let $V := [2M]$.  We call the vertices in $[2M] \backslash [M]$ \emph{green} (thus every vertex in $V$ is either red, blue, or green).  We then define $G^{(1)} = (V, E^{(1)})$ to be the $3$-uniform graph, where $E^{(1)}$ consists of all triples $\{x,y,z\} \in \binom{V}{3}$ for which one of the following statements are true:
\begin{itemize}
\item $x,y,z$ are all green.
\item $\{x,y,z\}$ consists of a red vertex, a blue vertex, and a green vertex, and the red vertex likes the blue vertex.
\item $\{x,y,z\}$ consists of a red vertex and two blue vertices, and the red vertex ranks the two blue vertices correctly.
\end{itemize}

Note how $E^{(1)}$ involves the three components $E^{(0)}_1, E^{(0)}_2, E^{(0)}_3$ of $E^{(0)}$.

\begin{figure}[tb]
\centerline{\psfig{figure=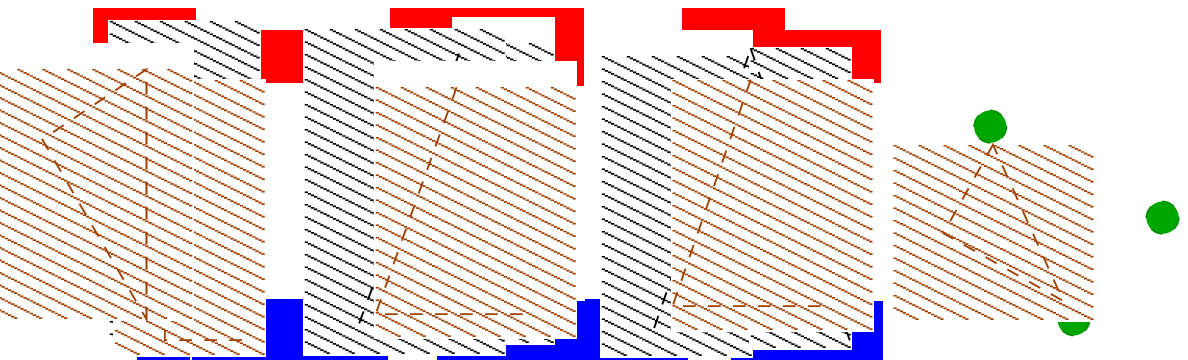}}
\caption{The various types of triples that make up $E^{(1)}$.  In addition to the triples that are inherited from $E^{(0)}_3$, one also has triples that connect three green vertices together, or else connect a green vertex to a red vertex that likes a blue vertex.  Note that the four green vertices on the right will in fact form a tetrahedron (and thus be $G$-green), whereas any quadruple of vertices which is not entirely green cannot form such a tetrahedron.}
\label{fig3}
\end{figure}

Now we define $\P$.  For any $K$-coloured undirected hypergraph $G = (V,E)$, we introduce the following notation:

\begin{itemize}
\item We call an element $g_1 \in V$ \emph{$G$-green} if there exists $\{g_2,g_3,g_4\} \in \binom{V \backslash \{g_1\}}{3}$ such that $\binom{\{g_1,g_2,g_3,g_4\}}{3} \subset E$.
\item We call an element $x \in V$ \emph{$G$-nongreen}\footnote{We allow for the possibility that a vertex is both $G$-green and $G$-nongreen, or is neither $G$-green nor $G$-nongreen.  However, these situations will not occur for the model graph $G^{(1)}$.} if there exist distinct $G$-green vertices $g, g'$ such that $\{x,g,g'\} \not \in E$.
\item If $x, y \in V$ are distinct, we say that $x$ \emph{$G$-likes} $y$ if they are both $G$-nongreen, and there exists a $G$-green vertex $g$ such that $\{ x, y, g \} \in E$.
\item Two vertices $x, x' \in V$ are \emph{$G$-similar} if there exists $y$ such that $x, x'$ both $G$-like $y$.
\item If $r, b \in V$ are distinct, we say that $r$ \emph{$G$-dislikes} $b$ if $r,b$ are both $G$-nongreen, and there exists a $G$-green vertex $g$ such that $\{ x, y, g \} \not \in E$.
\item If $b, b', r$ are distinct elements of $V$, we say that $r$ \emph{$G$-prefers $b$ to $b'$} if $r,b,b'$ is $G$-nongreen, $b,b'$ are $G$-similar, $r$ $G$-likes $b$, and $r$ $G$-dislikes $b'$.
\item If $b, b', r$ are distinct elements of $V$, we write $b >_{G,r} b'$ if either (a) $r$ $G$-prefers $b$ to $b'$ and $\{r,b,b'\} \in E$; or (b) $r$ $G$-prefers $b'$ to $b$ and $\{r,b,b'\} \not \in E$.
\item The hypergraph $G$ is \emph{consistently orderable} if there exists a total ordering $>_G$ on $V$ such that $b >_G b'$ whenever $r, b, b'$ are such that $b >_{G,r} b'$.
\end{itemize}

We say that a $K$-coloured hypergraph obeys $\P$ if it is undirected and consistently orderable.  One can verify with some tedious effort that $\P$ is an undirected $K$-property.  One can also verify that when $G = G^{(1)}$, the $G^{(1)}$-green vertices are precisely the green vertices, the $G^{(1)}$-nongreen vertices the red and blue vertices, and $G^{(1)}$-similar vertices are either both red or both blue.  From this one can easily verify that $G^{(1)}$ obeys $\P$ (using the usual ordering $>$ on $[2M]$ for $>_{G^{(1)}}$).

We set $\eps > 0$ to be small ($\eps := \frac{1}{1000}$ will do).    Let $A$, $N > 0$, and $\delta > 0$ be arbitrary, and let $\sigma > 0$ be sufficiently small depending on all these parameters.  We then let $M$ be a large integer (depending on all previous parameters).
We define the ``corrupted'' $3$-uniform hypergraph $G = (V, E)$ by declaring $\I( e \in E ) := \I( e \in E^{(1)} )$ with probability $1-\sigma$ and $\I( e \in E ) := 1-\I( e \in E^{(1)} )$ with probability $\sigma$ for each $e \in \binom{V}{3}$, independently for each choice of $e$.

Since $G^{(1)}$ obeys $\P$, we can use the first moment method as in the preceding two sections to conclude that \eqref{injv} holds with probability $1 - O_{N,\delta}(\sigma)$.  Let us now condition on the event that \eqref{injv} holds.

To prove Theorem \ref{negate}(c), it will suffice to show that there does \emph{not} exist a local modification rule $T = (A,T)$ and a morphism $\phi \in \Hom(A,V)$ such that the repaired hypergraph $T_{\phi}(G)$ obeys $\P$ and \eqref{psieps}.

Suppose for contradiction that $T$ and $\phi$ exists with the above properties.  We write $G' = (V \backslash \phi(A),E') := T_\phi(G)$.  From \eqref{psieps} we thus have
\begin{equation}\label{lowc}
|E' \Delta E| \ll \eps M^3.
\end{equation}
Let us call an $9$-tuple
\begin{equation}\label{rgb}
(r_1, r_2, r_3, b_1, b_2, g_1, g_2, g_3, g_4 )
\end{equation}
of distinct vertices in $V \backslash \phi(A)$ \emph{inconsistent} if the following properties hold:

\begin{itemize}
\item[(i)] $\binom{\{g_1,g_2,g_3,g_4\}}{3} \subset E'$.
\item[(ii)] For $x \in \{r_1,r_2,r_3,b_1,b_2\}$ we have $\{x,g_1,g_2\} \not \in E'$.
\item[(iii)] For $i \in \{1,2,3\}$ and $j \in \{1,2\}$ we have $\{ r_i, b_j, g_1 \} \in E'$ if and only if $(i,j) \not \in \{ (1,1), (2,2) \}$.
\item[(iv)] We have the symmetries \eqref{symmetry} (with $E_3$ replaced by $E$).
\end{itemize}

Suppose that we can locate an inconsistent $9$-tuple \eqref{rgb}. From property (i) we see that $g_1,g_2,g_3,g_4$ are $G'$-green.  From property (ii) we then conclude that $r_1,r_2,r_3,b_1,b_2$ are $G'$-nongreen.  From property (iii) we conclude that for $i \in \{1,2,3\}$ and $j \in \{1,2\}$, that $r_i$ $G'$-likes $b_j$ whenever $(i,j) \not \in \{ (1,1), (2,2)\}$.  In particular, $b_1, b_2$ are  $G'$-similar (thanks to $r_3$).  From property (iii) again we also see that $r_1$ $G'$-dislikes $b_1$ and $r_2$ $G'$-dislikes $b_2$.  Thus $r_1$ $G'$-prefers $b_2$ to $b_1$, and $r_2$ $G'$-prefers $b_1$ to $b_2$.  On the other hand from property (iv) and Definition \ref{locmod} as in the previous section we see that $\I( \{ r_1, b_1, b_2\} \in E' ) = \I( \{ r_2, b_2, b_1\} \in E' )$.  Thus either $b_1 >_{G',r_1} b_2$ and $b_2 >_{G',r_2} b_1$ are both true, or $b_2 >_{G',r_1} b_1$ and $b_1 >_{G',r_2} b_2$ are both true, and so $G'$ is not consistently orderable and thus does not obey $\P$, a contradiction.  Thus to conclude the proof of Theorem \ref{negate}(c), it will suffice to show

\begin{lemma} Suppose $\eps > 0$ is sufficiently small, and $M \geq N_*$ is sufficiently large (depending on $N,\delta,\sigma,A,N_*,\eps$).  Then with probability $1-O_{A,\delta,\eps}(\sigma)$, there will exist at least one $9$-tuple \eqref{rgb} of inconsistent vertices in $V \backslash \phi(A)$, for all choices of morphism $\phi$ and modification rule $T$ for which \eqref{lowc} holds.
\end{lemma}

\begin{proof}
Let $c > 0$ be a small number depending on $\eps, A$ to be chosen later.
Recall the $2$-uniform random graph $E_2$ on $[M]$ used to construct $G^{(0)}$.  By arguing exactly as in the proof of Lemma \ref{triple}, we see (for $M$ large enough) that with probability $1-O(\sigma)$,
we have the regularity property \eqref{regular} for all disjoint $X, Y \subset [M]$ with $|X|, |Y| \geq cM$.  Let us condition on the event that this regularity property holds.  We now freeze $E_2$, which in turn freezes $G^{(1)}$ and $E^{(1)}$.

As in the proof of Lemma \ref{triple}, we pay a factor of $O_{A}(M^{|A|})$ in all future probability upper bounds in order to freeze $\phi$ and $T$.

From construction, we have
for any $\{ v_1, v_2, v_3 \} \in \binom{V}{3}$, that $\{v_1, v_2, v_3 \} \in E \Delta E^{(1)}$ with an independent probability $\delta$.  From Chernoff's inequality, we conclude that for each $v_1, v_2 \in V$, that
\begin{equation}\label{v3}
 | \{ v_3 \in V \backslash \{v_1,v_2\}: \{v_1,v_2,v_3\} \in E \Delta E^{(1)} \} | \leq \delta^{1/2} M
\end{equation}
with probability at least $1 - \exp( - \Omega_{\delta}(M) )$.  For technical reasons (related to the reason we did not condition on \eqref{lowcorrupt} for $j=3$ in the previous section), we will weaken \eqref{v3} to
\begin{equation}\label{v3-weak}
 | \{ v_3 \in V \backslash \{v_1,v_2\}: \{v_1,v_2,v_3\} \in (E \Delta E^{(1)}) \backslash \binom{[M] \backslash \phi(A)}{3} \} | \leq \delta^{1/2} M
\end{equation}
in order not to destroy the independence of the events $\{v_1,v_2,v_3\} \in E \Delta E^{(1)}$ when $v_1,v_2,v_3$ lie in $[M] \backslash \phi(A)$.

By the union bound, we thus see that (if $M$ is sufficiently large) that with probability
$1 - O( M \exp( - \Omega_{\delta}(M) ) ) = 1 - O(\sigma)$, the assertion \eqref{v3-weak} holds for \emph{all} $v_1,v_2 \in V$.  We now condition on the event that this holds.

We now freeze the restriction $E \backslash \binom{[M] \backslash \phi(A)}{3}$ of $E$ to those edges which are not contained in $[M] \backslash \phi(A)$.  Thus the only randomness remaining comes from the random variables $\I(e \in E^{(1)} \Delta E)$ for $e \in \binom{[M] \backslash \phi(A)\}}{3}$, which are jointly independent with probability $\delta$ each.  Note (from Definition \ref{locmod}) that the quantity $\I(e \in E')$ for $e \in \binom{[M]}{3}$ is now deterministic unless $e \in \binom{[M] \backslash \phi(A)\}}{3}$, in which case it depends only on the quantity $\I(e \in E^{(1)} \Delta E)$ (as well as frozen data, of course).

We would like to condition on the event that \eqref{locmod} holds, but this would destroy the joint independence of the events $e \in E^{(1)} \Delta E$, which will be important later.  So we shall be content to condition on the slightly weaker statement
\begin{equation}\label{small-delta-weak}
 \frac{1}{|\binom{V}{3}|} \left|(E \Delta E') \backslash \binom{[M] \backslash \phi(A)}{3}\right| \ll \eps
\end{equation}
as this is a deterministic statement that does not depend on the truth value of any of the events
$e \in E^{(1)} \Delta E$ for $e \in \binom{[M] \backslash \phi(A)}{3}$.

The next step is to select some good vertex sets to work with.  From \eqref{small-delta-weak} we have
$$ \sum_{v_1 \in V \backslash \phi(A)} \left| \left\{ \{ v_2, v_3 \} \in \binom{V \backslash \{v_1\}}{2}: \{v_1,v_2,v_3 \} \in (E \Delta E') \backslash \binom{[M] \backslash \phi(A)}{3} \right\} \right| \ll \eps M^3$$
and so (for $M$ large enough) by Markov's inequality we can find a subset $V' \subset V \backslash \phi(A)$ with $|V \backslash V'| \ll \eps^{1/2} M$ such that
\begin{equation}\label{vsting}
\left| \left\{ \{ v_2, v_3 \} \in \binom{V \backslash \{v_1\}}{2}: \{v_1,v_2,v_3 \} \in (E \Delta E') \backslash \binom{[M] \backslash \phi(A)}{3} \right\} \right| \ll \eps^{1/2} M^2
\end{equation}
for all $v_1 \in V'$.

Set
\begin{align*}
V_B &:= [M/2] \cap V'\\
V_R &:= ([M] \backslash [M/2]) \cap V'\\
V_G &:= ([2M] \backslash [M]) \cap V'.
\end{align*}
In particular (for $\eps$ small enough) we have $|V_B|, |V_R|, |V_G| \gg M$.

For $b \in V_B$ and $r \in V_R$, define
\begin{equation}\label{eep}
f(r,b) := | \{ v \in V \backslash \{r,b\}: \{r,b,v\} \in (E \Delta E') \backslash \binom{[M] \backslash \phi(A) }{3} \}|,
\end{equation}
thus $0 \leq f(r,b) \ll M$.  From \eqref{lowc} we observe that
$$\sum_{r \in V_R} \sum_{b \in V_B} f(r,b) \ll \eps |V_R| |V_B| M.$$
Thus if we define
\begin{equation}\label{es2-again}
E^*_2 := \{ (r,b) \in V_R \times V_B: f(r,b) \geq \sqrt{\eps} M \}
\end{equation}
then by Markov's inequality we have
\begin{equation}\label{fiber}
|E^*_2| \ll \sqrt{\eps} |V_R| |V_B|.
\end{equation}

Let $\Omega_R := 2^{\binom{A}{2}}$ and $\Omega_B := 2^{\binom{A}{\leq 2}}$
be the power sets of $\binom{A}{2}$ and $\binom{A}{\leq 2} := \bigcup_{j \leq 2} \binom{A}{j}$ respectively.  If $U_R \in \Omega_R$, define $V_{R,U_R}$ to be the set of all vertices $r \in V_R$ such that
$$ U_R = \{ \{a,a'\} \in \binom{A}{2}: \{\phi(a),\phi(a'),r\} \in E_2 \}.$$
Similarly, for any $U_B \in \Omega_B$, define $V_{B,U_B}$ to be the set of all $b \in V_{B}$ such that
$$ U_B = \{ \{ a \} \in \binom{A}{1}: b < \phi(a) \} \cup
\{ \{a,a'\} \in \binom{A}{2}: \{\phi(a),\phi(a'),b\} \in E_2 \}.$$
The $V_{R,U_R}$ and $V_{B,U_B}$ partition $V_R$ and $V_B$ respectively. Since $|\Omega_R|, |\Omega_B| \ll_{A} 1$, we thus see from \eqref{fiber} and the pigeonhole principle that there exists $U_R \in \Omega_R$ and $U_B \in\Omega_B$ with
\begin{equation}\label{verb-again}
 |V_{R,U_R}|, |V_{B,U_B}| \gg_{A} M
 \end{equation}
and
\begin{equation}\label{estar2}
 |E^*_2 \cap (V_{R,U_R} \times V_{B,U_B})| \ll \sqrt{\eps} |V_{B,U_B}| |V_{R,U_R}|.
\end{equation}
Henceforth we fix $U_B$ and $U_R$ so that \eqref{verb-again}, \eqref{estar2} hold.

To locate inconsistent $9$-tuples \eqref{rgb} we shall constructed a nested sequence $\Sigma_0 \supset \ldots \supset \Sigma_7$ of candidate $9$-tuples as follows.  We let $\Sigma_0$ be the collection of all $9$-tuples \eqref{rgb} such that $\{r_1,r_2,r_3\} \in \binom{V_{R,U_R}}{3}$, $\{b_1,b_2\} \in \binom{V_{B,U_B}}{2}$, and $\{g_1,g_2,g_3,g_4\} \in \binom{V_G}{4}$.  Clearly we have $|\Sigma_0| \gg |V_{R,U_R}|^3 |V_{B,U_B}|^2 M^4$.

Let $\Sigma_1$ be the collection of all $9$-tuples \eqref{rgb} in $\Sigma_0$ such that for all $i \in \{1,2,3\}$ and $j \in \{1,2\}$, we have $\{ r_i, b_j \} \in E_2$ if and only if $(i,j) \not\in \{ (1,1), (2,2) \}$.  Using \eqref{regular} and standard ``counting lemma'' arguments we see that if $c$ is sufficiently small (depending on $N'$), then $|\Sigma_1| \gg |V_{R,U_R}|^3 |V_{B,U_B}|^2 M^4$.

Let $\Sigma_2$ be the collection of all $9$-tuples \eqref{rgb} in $\Sigma_1$ such that $(r_i,b_j) \not \in E^*_2$ for all $i \in \{1,2,3\}$ and $j \in \{1,2\}$.  From \eqref{estar2} we have $|\Sigma_1 \backslash \Sigma_2| \ll \sqrt{\eps} |V_{R,U_R}|^3 |V_{B,U_B}|^2 M^4$.  Thus if $\eps$ is sufficiently small we have $|\Sigma_2| \gg |V_{R,U_R}|^3 |V_{B,U_B}|^2 M^4$.

Let $\Sigma_3$ be the collection of all $9$-tuples \eqref{rgb} in $\Sigma_2$ such that $\{r_i,b_j,g_k\} \not \in E \Delta E'$ for all $i \in \{1,2,3\}$, $j \in \{1,2\}$ and $k \in \{1,2,3,4\}$.  From \eqref{eep}, \eqref{es2-again} we see that $|\Sigma_2 \backslash \Sigma_3| \ll \sqrt{\eps} |V_{R,U_R}|^3 |V_{B,U_B}|^2 M^4$. Thus if $\eps$ is sufficiently small we have $|\Sigma_3| \gg |V_{R,U_R}|^3 |V_{B,U_B}|^2 M^4$.

Let $\Sigma_4$ be the collection of all $9$-tuples \eqref{rgb} in $\Sigma_3$ such that $\{x, y, z \} \not \in E \Delta E'$ for all $x \in \{r_1,r_2,r_3,b_1,b_2\}$ and distinct $y,z \in \{g_1,g_2,g_3,g_4\}$.  From \eqref{vsting} we see that $|\Sigma_3 \backslash \Sigma_4| \ll \sqrt{\eps} |V_{R,U_R}|^3 |V_{B,U_B}|^2 M^4$. Thus if $\eps$ is sufficiently small we have $|\Sigma_4| \gg |V_{R,U_R}|^3 |V_{B,U_B}|^2 M^4$.

Let $\Sigma_5$ be the collection of all $9$-tuples \eqref{rgb} in $\Sigma_4$ such that $\{x, y, z \} \not \in E \Delta E'$ for all $\{x,y,z\} \in \binom{\{g_1,g_2,g_3,g_4\}}{3}$.  From \eqref{small-delta-weak} we observe that $|\Sigma_4 \backslash \Sigma_5| \ll \eps |V_{R,U_R}|^3 |V_{B,U_B}|^2 M^4$. Thus if $\eps$ is sufficiently small we have $|\Sigma_5| \gg |V_{R,U_R}|^3 |V_{B,U_B}|^2 M^4$.  In particular, by \eqref{verb-again} we have $|\Sigma_5| \gg_{A} M^{9}$.

Let $\Sigma_6$ be the collection of all $9$-tuples \eqref{rgb} in $\Sigma_5$ such that $\{x,y,z\} \not \in E \Delta E^{(1)}$ for all
\begin{equation}\label{xyz}
 \{x,y,z\} \in \binom{\{ r_1, r_2, r_3, b_1, b_2, g_1,g_2,g_3,g_4 \} \cup \phi(A)}{3} \backslash
\left(\binom{\phi(A)}{3} \cup \binom{\{r_1,r_2,b_1,b_2\}}{3} \right).
\end{equation}
From \eqref{v3} we see that $|\Sigma_5 \backslash \Sigma_6| \ll_{N'} \delta^{1/2} M^{9}$.  Thus if $\delta$ is sufficiently small (depending on $N'$) then $|\Sigma_6| \gg_{A} M^{9}$.

Let $\Sigma_7$ be the collection of all $9$-tuples \eqref{rgb} in $\Sigma_6$ such that
$$ \I( \{r_1,b_1,b_2\} \in E ) = \I( \{r_2,b_1,b_2\} \in E ) \hbox{ and }
\I( \{b_1,r_1,r_2\} \in E ) = \I( \{b_2,r_1,r_2\} \in E ) $$
To estimate the size of $\Sigma_7$ we will need a slightly different type of argument than those used in previous paragraphs, namely a probabilistic argument.  Let $\Hom([9],[2M])$ denote the space of all $9$-tuples \eqref{rgb}.  From the lower bound $|\Sigma_6| \gg_{A} M^{9}$ we see that for each fixed $n$, a randomly selected $9$-tuple \eqref{rgb} in $\Hom([9],[2M])$ would lie in $\Sigma_6$ with probability $\gg_{A} 1$.

Now observe from construction of $\Sigma_6$ and $E$ that the event that \eqref{rgb} lies in $\Sigma_6$ is independent\footnote{Note how it is important here that $\binom{\{r_1,r_2,b_1,b_2\}}{3}$ is excluded in \eqref{xyz}.} of the events $\{x,y,z\} \in E \Delta E^{(1)}$ for $\{x,y,z\} \in \binom{\{r_1,r_2,b_1,b_2\}}{3}$, which each occur with an independent probability of $\delta$.  Thus, regardless of the truth values of $\{x,y,z\} \in E^{(1)}$ for $\{x,y,z\} \in \binom{\{r_1,r_2,b_1,b_2\}}{3}$, we see that if one conditions on the event \eqref{rgb} lies in $\Sigma_6$, then \eqref{rgb} will lie in $\Sigma_7$ with probability $\gg_{\delta} 1$.  Undoing the conditioning on $\Sigma_6$, we see that a randomly chosen $9$-tuple \eqref{rgb} in $\Hom([9],[2M])$ lies in $\Sigma_7$ with probability $\gg_{\delta,A} 1$.

Let $A := \lfloor M^{0.1} \rfloor$.
We pick $A$ $9$-tuples $t_1,\ldots,t_A \in \Hom([9],[2M])$ independently at random (and independently of $E_2$ and $E$).  With probability $1 - O( M^{-0.8} )$, these tuples will be disjoint; we condition on this event.  Now we make the crucial observation the events $t_i \in \Sigma_7$ are jointly independent for $i=1,\ldots,A$.  Indeed, in view of all the frozen data, the event that $t_i$ lies in $\Sigma_7$ depends only on the truth value of the events $\{x,y,z\} \in (E \Delta E^{(1)}) \cap \binom{[M] \backslash \phi(A)}{3}$, where $\{x,y,z\}$ lies $t_i$, and the independence assertion follows.  (It is for this reason that we have jealously guarded the joint independence of the edge events associated to $\binom{[M] \backslash \phi(A)}{3}$.)

Now for any $1 \leq i \leq A$, if we condition on $t_1,\ldots,t_{i-1}$ then each $t_i$ will lie in $\Sigma_7$ with probability $\gg_{\delta,N'} 1$ (the constraint that $t_1,\ldots,t_A$ are all disjoint only distorts this probability by $O(M^{-0.8})$, which is negligible if $M$ is large enough).  Multiplying this together we see that with probability at least $1 - \exp( \Omega_{\delta,N'}(M^{0.1}) )$, at least one of the $t_i$ will lie in $\Sigma_7$.  Unfreezing $\phi$ and $T$, we conclude from the union bound that with probability $1 - O_{A,\delta}( M^{|A|} \exp( \Omega_{\delta,A}(M^{0.1}) ) )$, we have $\Sigma_7$ non-empty for \emph{all} choices of $\phi(A)$ and $T$.  In particular, this event occurs with probability $1 - O(\sigma)$ if $M$ is large enough.

To conclude the lemma, it will suffice to show that every $9$-tuple in $\Sigma_7$ is inconsistent.  Let \eqref{rgb} be a tuple in $\Sigma_7$.  By definition of $\Sigma_0$, we have $g_1,g_2,g_3,g_4 \in V_G$, and thus by definition of $E^{(1)}$
$$ \binom{\{g_1,g_2,g_3,g_4\}}{3} \subset E^{(1)}.$$
From the definition of $\Sigma_6$ we thus have
$$ \binom{\{g_1,g_2,g_3,g_4\}}{3} \subset E$$
and then by definition of $\Sigma_5$ we have
$$ \binom{\{g_1,g_2,g_3,g_4\}}{3} \subset E'$$
which is part (i) of the definition of inconsistency.

Similarly, by definition of $\Sigma_0$ and $E^{(1)}$ we have $\{x,g_1,g_2\} \not \in E^{(1)}$ for all $x \in \{r_1,r_2,r_3,b_1,b_2\}$.  By definition of $\Sigma_6$ we then have $\{x,g_1,g_2\} \not \in E$, and by definition of $\Sigma_4$ we have $\{x,g_1,g_2\} \not \in E'$. This is part (ii) of the definition of inconsistency.

From the definition of $\Sigma_0$, $g_1$ is green.  From the definitions of $\Sigma_1$ and $E^{(1)}$ we then have for every  $i \in \{1,2,3\}$ and $j \in \{1,2\}$ that $\{ r_i, b_j, g_1 \} \in E^{(1)}$ if and only if $(i,j) \not \in \{ (1,1), (2,2) \}$.  By definition of $\Sigma_6$, the same statement holds with $E^{(1)}$ replaced by $E$, and by definition of $\Sigma_2$ the same statement holds with $E$ replaced by $E'$.  This is part (iii) of the definition of inconsistency.

It remains to verify \eqref{symmetry} (with $E_3$ replaced by $E$).
The first two symmetries follow from the definition of $\Sigma_7$.
The last two symmetries follow from the definitions of $\Sigma_0$
and $V_{R,U_R}$, $V_{B,U_B}$.  To verify the middle two symmetries,
we see from definition of $\Sigma_3$ that it suffices to show that
$\I( \{r_1,b_1,\phi(a)\} \in E^{(1)} ) = \I( \{r_2,b_2,\phi(a)\} \in
E^{(1)} )$ and $\I( \{r_1,b_2,\phi(a)\} \in E^{(1)} ) = \I(
\{r_2,b_1,\phi(a)\} \in E^{(1)} )$ for all $a \in A$.

Fix $a$.  There are several cases.  If $\phi(a)$ is green, then the claim follows from the definitions of $E^{(1)}$ and $\Sigma_1$.  If $\phi(a)$ is red, then by definition of $E^{(1)}$, none of the $\{r_j,b_k,\phi(a)\}$ lie in $E^{(1)}$, and the claim follows.  Finally, suppose that $\phi(a)$ is blue.  By definition of $V_{B,U_B}$ and $V_{R,U_R}$, we see that $\I( b_1 < \phi(a) ) = \I( b_2 < \phi(a) )$ and $\I( (r_1,\phi(a)) \in E_2 ) = \I( (r_2,\phi(a)) \in E_2 )$.  Also, by definition of $\Sigma_1$ we have $\I( (r_1,b_1) \in E_2 ) = \I( (r_2,b_2) \in E_2 )$ and $\I( (r_1,b_2) \in E_2 ) = \I( (r_2,b_1) \in E_2 )$.  The claim then follows from the definition of $E^{(1)}$.
\end{proof}

This concludes the proof of Theorem \ref{negate}(c).

\begin{remark}  One does not need the full strength of consistent orderability to define $\P$; it is enough that there do not exist $r,r',b,b'$ such that $b >_{G,r} b'$ and $b' >_{G,r'} b$.  With this modification, the property $\P$ can now be expressed as a single first-order sentence\footnote{Equivalently, there exists a finite collection of ``forbidden'' hypergraphs which describe $\P$, in the sense that $G$ obeys $\P$ if and only if it contains no induced copy of any of the forbidden hypergraphs.  In contrast, hereditary properties are associated to an \emph{at most countable} family of forbidden hypergraphs.} using only the universal quantifier $\forall$, which is a slightly stronger statement than saying that $\P$ is hereditary.  This gives a slight strengthening to Theorem \ref{negate}(c).
\end{remark}

\section{Proofs of the positive results}\label{posi}

We now begin the proofs of the positive results.  Except in side remarks and examples, the material here is independent of that in Section \ref{negchap}.

\subsection{An infinitary setting: exchangeable random hypergraphs and their structure}

In order to prove our new positive results, it will be helpful to
recast the graphs and hypergraphs that we are studying into a more infinitary form (although the actual arguments will still be
structured much as in the finitary presentations in Alon and
Shapira~\cite{AloSha} and elsewhere). The formalism we will use is
that of `exchangeable random hypergraphs', which have already
appeared in the study of single hypergraph removal lemmas
in~\cite{Tao3} and whose structure is examined in more detail
in~\cite{Aus1}.  In addition to providing a reasonably clean
language for handling continuous graphs, these also admit their own
versions of the theorem we shall prove, in whose statement the
existence of an $\varepsilon$-modification of a given graph or
hypergraph to another that satisfies a certain property is replaced
by that of a near-diagonal joining of a given exchangeable random
graph or hypergraph to another that satisfies the relevant property
almost surely.

The infinitary setting offers several advantages.  Firstly, it conceals from view many quantitative parameters such as $\eps$ and $N$ which would otherwise have to be managed directly by hand; the process of taking a limit sends most (though not all) of these parameters to zero or infinity, and the remaining parameters often just need to be controlled qualitatively (e.g. knowing that they are finite) rather than quantitatively (i.e. with an explicit bound).  Secondly, it allows one to use the standard tools and intuition from basic infinitary theories, most notably topology, measure theory, and probability theory.  For instance, the well-known fact that measurable functions can be approximated by continuous ones will form a partial substitute for the Szemer\'edi regularity lemma.

The purpose of this section is to review the relevant theory from \cite{Aus1} which we will need here.  To begin with we shall work with undirected graphs, and then discuss the (minor) modifications needed to handle directed graphs later in this section.

\subsubsection{The category of sub-Cantor spaces}

Our infinitary analysis will take place in the category of \emph{sub-Cantor spaces}, which we now pause to define.

\begin{dfn}[Sub-Cantor spaces]\label{subcantor}  A \emph{sub-Cantor space} is a topological space $Z$ which is homeomorphic to a compact subset of the standard Cantor space $\{0,1\}^\N$. We always endow sub-Cantor spaces with their Borel $\sigma$-algebra generated by the open sets (or compact sets).  We say that a sub-Cantor space is \emph{trivial} or a \emph{point} if $Z$ is a singleton set, and write $Z = \pt$ in this case.
\end{dfn}

\begin{examples} Any finite set is a sub-Cantor space, a closed subspace of a sub-Cantor space is again a sub-Cantor space, and any at most countable product of a sub-Cantor space is again a sub-Cantor space.  In particular, $K^{(V)}$ is a sub-Cantor space for any finite palette $K$ and any vertex set $V$.
\end{examples}

\begin{remark}\label{borel} By a theorem of Borel, a space is a sub-Cantor space if and only if it is totally disconnected, compact, and metrisable.  However, we will not need that characterisation here.  We also make the useful observation that the topology of a sub-Cantor space can be generated from a countable algebra of clopen sets, as this property can be easily verified for the Cantor space $\{0,1\}^\N$ and is preserved under passage to compact subspaces.
\end{remark}

We will view the class of sub-Cantor spaces as a category, where the morphisms are the \emph{probability kernels} $P: X \rightsquigarrow Y$ between sub-Cantor spaces $X,Y$; see Appendix \ref{prob} for a definition of a probability kernel and their relevant properties.  (This is distinct from the category of vertex sets, defined in Definition\ref{vertset}.) Informally, one can think of a probability kernel as a stochastic analogue of a function from $X$ to $Y$, mapping points in $X$ to probability distribtions in $Y$ rather than to deterministic points. We distinguish several special types of probability kernels between sub-Cantor spaces:

\begin{itemize}
\item A probability kernel $P: X \rightsquigarrow Y$ is \emph{deterministic} if we have $P(x) = \delta_{\phi(x)}$ for all $x \in X$ and some measurable $\phi: X \to Y$;
\item A probability kernel $P: X \rightsquigarrow Y$ is \emph{deterministically continuous} if we have $P(x) = \delta_{\phi(x)}$ for all $x \in X$ and some continuous $\phi: X \to Y$;
\item A probability kernel $P: X \rightsquigarrow Y$ is \emph{weakly continuous} if the function $x \mapsto \int_Y f(y)\ P(x,dy)$ is continuous for every continuous function $f: Y \to \R$.
\end{itemize}

\begin{remark} Recall from Remark \ref{borel} that a sub-Cantor space has a countable base of clopen sets.  Because of this, one can easily verify that a probabilistic kernel is deterministically continuous if and only if it is both deterministic and weakly continuous.  As we will show later (see Proposition \ref{tprop} and Definition \ref{infitest}), the concept of weak continuity will correspond to testability with one-sided error, while deterministic continuity will correspond to strong local repairability.  Roughly speaking, weak continuity is the minimal amount of regularity necessary for one to be able to transfer infinitary results back to the finitary setting, while strong continuity, in view of the sub-Cantor structure, means that the relevant continuous maps $\phi: X \to Y$ between sub-Cantor spaces ``depend on only finitely many coordinates'' and will thus define a local modification rule.
\end{remark}

Rather than work on an individual sub-Cantor space, it will be useful to conduct our analysis on \emph{families} of sub-Cantor spaces indexed by vertex sets, with various morphisms between these spaces.  The most convenient way to handle these families is via the notion of a \emph{contravariant functor} from category theory.

\begin{definition}[Contravariant functor]  A \emph{contravariant functor} $Z$ is an assignment of a sub-Cantor space $Z^{(V)}$ to every vertex set $V$, together with a probability kernel $Z^{(\phi)}: Z^{(V)} \to Z^{(W)}$ for every morphism\footnote{Recall that in the category of vertex sets (as opposed to that of sub-Cantor spaces), the morphisms are just the (deterministic) injective maps between vertex sets.} $\phi \in \Hom(W,V)$ between vertex sets, such that $Z^{(\id_V)}: Z^{(V)} \to Z^{(V)}$ is the identity probability kernel on $Z^{(V)}$ for every vertex sets $V$, and such that $Z^{(\phi \circ \psi)} = Z^{(\psi)} \circ Z^{(\phi)}$ for any morphisms $\phi \in \Hom(W,V)$ and $\psi \in \Hom(V,U)$ between vertex sets.  We say that the contravariant functor is \emph{deterministically continuous} (resp. \emph{weakly continuous}) if all the probability kernels $Z^{(\phi)}$ are deterministically continuous (resp. weakly continuous).  If $z \in Z^{(V)}$ and $W \subset V$, we write $z\downharpoonright_W \in Z^{(W)}$ for $Z^{(\iota_{W \subset V})}(z)$, and refer to $z\downharpoonright_W$ as the \emph{restriction} of $z$ to $W$. Similarly, if $\mu \in \Pr(Z^{(V)})$ and $W \subset V$, we write $\mu\downharpoonright_W \in \Pr(Z^{(W)})$ for the projected measure $Z^{\iota_{W \subset V}} \circ \mu$.

If $Z$ is a contravariant functor and $S$ is a vertex set, we define the \emph{shift} $Z^{\uplus S}$ to be the contravariant functor given by requiring that
$$ (Z^{\uplus S})^{(V)} := Z^{(V \uplus S)}$$
for all vertex sets $V$ and
$$ (Z^{\uplus S})^{(\phi)} := Z^{(\phi \oplus \id_S)}$$
for all morphisms $\phi$.  One easily verifies that $Z^{\uplus S}$ is a contravariant functor, which is deterministically continuous (resp. weakly continuous) if $Z$ is.
\end{definition}

\begin{remark}  Intuitively, a contravariant functor is a recipe for generating a space of objects $Z^{(V)}$ to every vertex set $V$, to which one can meaningfully perform operations such as relabeling $V$, or restricting $V$ to a subset $W$.  A typical example of such a space $Z^{(V)}$ would be $K^{(V)}$, the space of $K$-coloured hypergraphs on $V$.  Note however that we allow the relabeling and restriction operations to be stochastic rather than deterministic.
\end{remark}

In this paper we will only be dealing with either deterministically continuous or weakly continuous contravariant functors.  One such functor is the \emph{trivial functor} $\pt$, which maps every vertex set to a point (and every morphism to the unique probability kernel between two points).  More generally, an important source of such functors for us will come from \emph{sub-Cantor palettes}.

\begin{definition}[Sub-Cantor palettes]  A \emph{sub-Cantor palette} is a tuple $Z = (Z_j)_{j=0}^\infty$ of sub-Cantor spaces, all but finitely many of which are trivial.  We define the \emph{order} of $Z$ to be the largest $k$ for which $Z_k$ is non-trivial, or $-1$ if all components $Z_j$ are trivial.  We identify $Z$ with a deterministically continuous contravariant functor by defining
$$ Z^{(V)} := \prod_{j=0}^\infty Z_j^{\Hom([j],V)}$$
for all vertex sets $V$, and defining $Z^{(\phi)}: Z^{(V)} \to Z^{(W)}$ for all morphisms $\phi \in \Hom(W,V)$ by the formula
$$ Z^{(\phi)}( ( ( z_j(\psi) )_{\psi \in \Hom([j],V)} )_{j=0}^\infty ) = ( ( z_j(\phi \circ \psi) )_{\psi \in \Hom([j],W)} )_{j=0}^\infty $$
for all $( ( z_j(\psi) )_{\psi \in \Hom([j],V)} )_{j=0}^\infty \in Z^{(V)}$.  One easily verifies that $Z$ is indeed a deterministically continuous contravariant functor.

If $j$ is an integer, we write $Z_{\leq j}$ (resp. $Z_{<j}$, $Z_{\geq j}$, $Z_{>j}$, $Z_{=j}$) for the sub-Cantor palette whose $i^{\th}$ component is $Z_i$ when $i \leq j$ (resp. $i<j$, $i\geq j$, $i>j$, $i=j$) and a point otherwise.
\end{definition}

\begin{example} The finite palettes in Definition \ref{hyperdef} are sub-Cantor palettes.
\end{example}

\begin{example}[Sub-Cantor spaces as contravariant functors]\label{trivial} A sub-Cantor space $X$ can be viewed as a sub-Cantor palette of order $0$, and can therefore be viewed as a contravariant functor, in which $X^{(V)} = X$ and $X^{(\phi)} = \id_X$ for all vertex sets $V$ and morphisms $\phi$.
\end{example}

\begin{example}[Hypergraph properties as contravariant functors]\label{propfunc}  If $K$ is a finite palette and $\P$ is a hereditary $K$-property, one easily verifies for every vertex set $V$ that $\P^{(V)}$ is a closed subspace of $K^{(V)}$ and is therefore itself a sub-Cantor space.  From this and the hereditary nature of $\P$ we see that $\P$ is in fact a contravariant functor.
\end{example}

We will also need to deal with families of probability kernels between one family of sub-Cantor spaces and another.  The most convenient way to handle such a concept is using the notion of a \emph{natural transformation} from category theory.

\begin{definition}[Natural transformation]\label{natdef}  A \emph{natural transformation} $N: Z \to Y$ between two contravariant functors $Z, Y$ is an assignment of a probability kernel $N^{(V)}: Z^{(V)} \rightsquigarrow Y^{(V)}$ for every vertex set $V$, such that the diagram
\begin{equation}\label{natprop}
\begin{CD}
Z^{(V)}             @>{N^{(V)}}>>   Y^{(V)} \\
@VV{Z^{(\phi)}}V                   @VV{Y^{(\phi)}}V            \\
Z^{(W)}             @>{N^{(W)}}>>   Y^{(W)}
\end{CD}
\end{equation}
commutes for every morphism $\phi \in \Hom(W,V)$ between vertex sets (the horizontal arrows here being probability kernels rather than continuous maps).  We say that the natural transformation is \emph{deterministically continuous} (resp. \emph{weakly continuous}) if all the probability kernels $N^{(V)}$ are deterministically continuous (resp. weakly continuous).

An \emph{exchangeable $Z$-recipe} on a contravariant functor $Z$ is a natural transformation $\mu: \pt \to Z$ from the trivial functor to $Z$, or equivalently an assignment of a probability measure $\mu^{(V)} \in \Pr(Z^{(V)})$ to every vertex set $V$, such that one has the exchangeability property
\begin{equation}\label{zphi}
Z^{(\phi)} \circ \mu^{(V)} = \mu^{(W)}
\end{equation}
for all morphisms $\phi \in \Hom(W,V)$ between two vertex sets.  If $S$ is a vertex set, we define the exchangeable $Z^{\uplus S}$-recipe $\mu^{\uplus S}: \pt \to Z^{\uplus S}$ by the formula $(\mu^{\uplus S})^{(V)} := \mu^{(V \uplus S)}$.
\end{definition}

\begin{remark} The condition \eqref{natprop} can be divided into two sub-conditions, namely \emph{equivariance} (or \emph{exchangeability})
\begin{equation}\label{exchange}
 Y^{(\phi)} \circ N^{(V)} = N^{(V)} \circ Z^{(\phi)} \hbox{ for all } \phi \in \Hom(V,V)
 \end{equation}
and \emph{locality}
\begin{equation}\label{local}
 N^{(V)}(z)\downharpoonright_W = N^{(W)}(z\downharpoonright_W) \hbox{ for all } W \subset V \hbox{ and } z \in Z^{(V)}.
\end{equation}
Similarly, if $\mu$ is an exchangeable $Z$-recipe, then $\mu^{(V)}$ is an $\Hom(V,V)$-invariant measure on $Z^{(V)}$, and the pushforward of $\mu^{(V)}$ under the restriction map to a subset $W$ of $V$ is the measure $\mu^{(W)}$.

Intuitively, a natural transformation $N: Z \to Y$ is a rule (which may be either deterministic or stochastic) for converting $Z$-type objects on a given vertex set $V$ to $Y$-type objects on the same vertex set, in a manner which is both local (in the sense of \eqref{local}) and exchangeable (in the sense of \eqref{exchange}).  We will shortly give a number of examples of natural transformations, such as recolouring maps, and local modification rules.

If $Z$ is a palette, one can view an exchangeable $Z$-recipe as a means for constructing a random $Z$-coloured hypergraph on any vertex set $V$, which is exchangeable with respect to relabeling of $V$, and also respects restriction from one vertex set to a subset.
\end{remark}

\begin{remark}  For future reference we observe the obvious fact that the composition $N_1 \circ N_2: Z \to X$ of two natural transformations $N_1: Y \to X$ and $N_2: Z \to Y$, defined by $(N_1 \circ N_2)^{(V)} := N_1^{(V)} \circ N_2^{(V)}$, is again a natural transformation.  
\end{remark}

Many important combinatorial operations on hypergraphs can be interpreted as natural transformations\footnote{Informally, any operation on hypergraphs which is both local (the effect of an operation on a subset $W$ of the vertex set $V$ depends only on the restriction of the hypergraph to $W$) and exchangeable (the operation respects hypergraph isomorphism) will have an interpretation as a natural transformation.}.  We list some examples of relevance to our applications here.

\begin{definition}[Colouring as a natural transformation]\label{colour}  Let $Z = (Z_j)_{j=0}^\infty$ be a sub-Cantor palette.  A \emph{colouring} $\alpha: Z \to A$ of $Z$ is a tuple $\alpha = (\alpha_j)_{j=0}^\infty$ of \emph{continuous}\footnote{Informally, this means that the colour assigned to any point in $Z$ depends only on ``finitely many coordinates'' of that point.} maps $\alpha_j: Z_j \to A_j$, where $A = (A_j)_{j=0}^\infty$ is a finite palette.  Each individual map $\alpha_j$ can be interpreted as a deterministically continuous natural transformation $\overline{\alpha_j}: Z_j \to A_j$ defined by the formula
$$ \overline{\alpha_j}^{(V)}( (z(\phi))_{\phi \in \Hom([j],V)} ) := ( \alpha_j(z(\phi)))_{\phi \in \Hom([j],V)}$$
and then the entire colouring can be viewed as a deterministically continuous natural transformation $\overline{\alpha}: Z \to A$ by
$$ \overline{\alpha}^{(V)}( (z_j)_{j=0}^\infty ) := ( \overline{\alpha_j}^{(V)}(z_j) )_{j=0}^\infty.$$
One easily verifies that $\overline{\alpha_j}$ and $\overline{\alpha}$ are indeed deterministically continuous natural transformations.  We say that a colouring $\alpha: Z \to A$ \emph{refines} or \emph{is finer than} another $\kappa: Z \to K$ if we have $\kappa = \sigma \circ \alpha$ for some colouring $\sigma: A \to K$.
\end{definition}

\begin{example}[Probability measures as exchangeable recipes]\label{trivial-2} If $X$ is a sub-Cantor space (which we can view as a palette of order 0 and thus as a contravariant functor, by Example \ref{trivial}), then an exchangeable $X$-recipe $\mu$ is nothing more than just a probability measure $\mu \in \Pr(X)$ on $X$.
\end{example}

\begin{definition}[Sampling as an exchangeable recipe]\label{sampling}  Let $K = (K_j)_{j=0}^k$ be a finite palette, and let $G = (V,\B,\nu,(G_j)_{j=0}^k)$ be a continuous $K$-coloured hypergraph.  For any vertex set $S$, the sampling map $\overline{G}^{(S)}: V^S \to K^{(S)}$ is a measurable map, and $\overline{\nu}^{(S)} := \nu^S$ is a probability measure on $V^S$.  Thus the pushforward measure $\overline{G}^{(S)} \circ \overline{\nu}^{(S)}$ is a probability measure on $K^{(S)}$, which can be viewed as a probability kernel from $\pt$ to $K^{(S)}$.  We can then define the exchangeable $K$-recipe $\overline{G} \circ \overline{\nu}: \pt \to K$ by letting $(\overline{G} \circ \overline{\nu})^{(S)} := \overline{G}^{(S)} \circ \nu^S$; one easily verifies that this is indeed an exchangeable $K$-recipe.  (If $V$ was a sub-Cantor space, and thus identifiable with a sub-Cantor palette of order $1$, one could interpret $\overline{\nu}: \pt \to V$ as an exchangeable $V$-recipe, and $\overline{G}: V \to K$ as a deterministic natural transformation; however, we will not need to adopt this perspective here.)
\end{definition}

\begin{example}[Inclusion as a natural transformation]  If $K$ is a finite palette and $\P$ is a hereditary $K$-property, then the inclusion natural transformation $\iota: \P \to K$ is a deterministically continuous natural transformation.
\end{example}

\begin{example}[Local modification rule as natural transformation]\label{lmr-nat} A local modification rule $T = (T,A)$ on a finite palette $K$ can be viewed as a deterministically continuous natural transformation $\overline{T}: K^{\uplus A} \to K$, with the maps $\overline{T}^{(V)}: K^{(A \uplus V)} \to K^{(V)}$ given by either Definition \ref{concmod} or Definition \ref{locmod}; the locality condition \eqref{local} reflects the fact that the colour assigned to an edge $\phi \in \Hom([j],V)$ by such a rule only depends on the restriction of the original graph to $A \cup \phi([j])$.  If $\P$ is a hereditary $K$-property $\P$, then $T$ entails $\P$ if and only if the associated natural transformation $\overline{T}$ factors through the inclusion natural transformation $\iota: \P \to K$.
\end{example}



\begin{definition}[Direct sum of natural transformations] If $Y_1$ and $Y_2$ are contravariant functors, we define the \emph{Cartesian product} $Y_1 \times Y_2$ to be the contravariant functor defined by $(Y_1 \times Y_2)^{(V)} :=Y_1^{(V)} \times Y_2^{(V)}$ for all vertex sets $V$, and $(Y_1 \times Y_2)^{(\phi)}(y_1,y_2) := (Y_1^{(\phi)}(y_1), Y_2^{(\phi)}(y_2))$ for all morphisms $\phi \in \Hom(W,V)$ and points $y_1 \in Y_1^{(V)}$, $y_2 \in Y_2^{(V)}$; one easily verifies that $Y_1 \times Y_2$ is indeed a contravariant functor.  If $N_1: Z_1 \to Y_1$ and $N_2: Z_2 \to Y_2$ are natural transformations, we define the \emph{direct sum} $N_1 \oplus N_2: Z_1 \times Z_2 \to Y_1 \times Y_2$ to be the natural transformation defined by $(N_1 \oplus N_2)^{(V)}(z_1,z_2) = (N_1^{(V)}(z_1), N_2^{(V)}(z_2))$ for all vertex sets $V$ and points $z_1 \in Z_1^{(V)}$ and $z_2 \in Z_2^{(V)}$; one easily verifies that $N_1 \oplus N_2$ is indeed a natural transformation.
\end{definition}

\begin{example} If $Z = (Z_j)_{j=0}^k$ is a sub-Cantor palette, then we have $Z = Z_{=0} \times \ldots \times Z_{=k}$ as contravariant functors.  If $\alpha = (\alpha_j)_{j=0}^k: Z \to A$ is a colouring, then we have $\overline{\alpha} = \overline{\alpha_{0}} \oplus \ldots \oplus \overline{\alpha_{k}}$.
\end{example}

We now turn to an important weak compactness property of recipes, which in fact is the main reason why we have set up all this infinitary machinery in the first place.

\begin{definition}[Vague convergence of recipes]  Let $Z$ be a sub-Cantor palette, let $\mu_n: \pt \to Z$ be a sequence of exchangeable $Z$-recipes, and let $\mu: \pt \to Z$ be another exchangeable $Z$-recipe.  We say that $\mu_n$ \emph{converges vaguely} to $\mu$ if $\mu_n^{(V)}$ converges vaguely to $\mu^{(V)}$ for every vertex set $V$ (see Appendix \ref{prob} for a definition of vague convergence of measures).
\end{definition}

\begin{lemma}[Vague sequential compactness of recipes]\label{compact}  Let $Z$ be a sub-Cantor palette, and let $\mu_n: \pt \to Z$ be a sequence of exchangeable $Z$-recipes.  Then there exists a subsequence $\mu_{n_j}: \pt \to Z$ which converges vaguely to another exchangeable $Z$-recipe $\mu: \pt \to Z$.
\end{lemma}

\begin{proof} Let $S$ be a countably infinite vertex set.  Then by Lemma \ref{seq}, we can find a subsequence $\mu_{n_j}: \pt \to Z$ such that the probability measures $\mu_{n_j}^{(S)} \in \Pr(Z^{(S)})$ converge vaguely to a measure $\mu^{(S)} \in \Pr(Z^{(S)})$.  Observe from \eqref{zphi} that $Z^{(\phi)} \circ \mu_{n_j}^{(S)} = \mu_{n_j}^{(S)}$ for all $\phi \in \Hom(S,S)$.  Since $Z^{(\phi)}$ is continuous, we can use vague convergence and conclude that
\begin{equation}\label{zamu}
Z^{(\phi)} \circ \mu^{(S)} = \mu^{(S)}.
\end{equation}
We can then define the exchangeable $Z$-recipe $\mu: \pt \to Z$ by defining $\mu^{(V)} := Z^{(\phi)} \circ \mu^{(S)}$ for any vertex set $V$ and any morphism $\phi \in \Hom(V,S)$; one easily verifies from \eqref{zamu} that $\mu$ is well-defined and is an exchangeable $Z$-recipe.  Also, as $\mu_{n_j}^{(S)}$ converges vaguely to $\mu^{(S)}$, one can see (by pulling back by an arbitrary morphism $\phi \in \Hom(V,S)$) that $\mu_{n_j}^{(V)}$ converges vaguely to $\mu^{(V)}$ for all vertex sets $V$.  The claim follows.
\end{proof}

\subsubsection{A structure theorem for exchangeable random hypergraphs}


In the infinitary framework, graphs and hypergraphs will be modeled
by exchangeable recipes, via the sampling operation in Definition
\ref{sampling}.  In order to use this formalism, we will need a classification of all the possible exchangeable recipes that one could
associate with a given palette $Z$.  Such a classification is
analogous to the Szemer\'edi and hypergraph regularity lemmas in the
finitary setting, or to the description of `limit objects' of
certain sequences of finite graphs or hypergraphs in terms of
`graphons' and `hypergraphons' in the works of Lov\'asz and
Szegedy~\cite{LovSze2} and Elek and Szegedy~\cite{EleSze}.  (The $k=1$ version of this classification is essentially de Finetti's theorem, a foundational result in the study of exchangeable probability measures.)

In fact, the classification that we need has been available in the
probabilistic literature for quite some time, appearing first in the
study of `exchangeable arrays of random variables' in the work of
Hoover~\cite{Hoo1,Hoo2}, Aldous~\cite{Ald1,Ald2,Ald3} and
Kallenberg~\cite{Kal}.  Their formalism is slightly removed from the
more combinatorial set-up and demands of the present paper, and so
we refer the reader to~\cite{Aus1} for a description of the
relationship between them and versions of these results suited to
our present purposes.

Let us first give some illustrative examples of exchangeable
$Z$-recipes that provide simple instances of the general result to
follow.

\begin{example}[Random vertex colouring]\label{rvc} Let $Z = (Z_0,Z_1)$ be a palette of order $1$, let $P_0 \in
\Pr(Z_0)$ be a probability measure (and thus identifiable with a
$Z_{\leq 0}$-recipe $P_0: \pt \to Z_{\leq 0}$), and let $Q_1: Z_0
\rightsquigarrow Z_1$ be a probability kernel.  If we then define
the probability kernel $P_1: Z_{\leq 0} \to Z$ by the formula
$P_1^{(V)}(z) := Q_1(z)^V$ for all vertex sets $V$ and $z \in Z_0$,
then $\mu := P_1 \circ P_0$ is an exchangeable $Z$-recipe.  This
recipe colours a given vertex set $V$ by first assigning a colour $z
\in Z_0$ at random with law $P_0$ to the empty set, and then
assigning a colour in $Z_1$ to each vertex independently at random
with law $Q_1(z)$.

A classical theorem of de Finetti asserts (in this language) that if
$Z_1$ is a sub-Cantor space and $\mu_{=1}$ is an exchangeable
$Z_{=1}$-recipe, then there exists $Z_0, P_0, Q_1, \mu$ as above
such that $\mu_{=1} = \pi \circ \mu$, where $\pi: Z \to Z_{=1}$ is
the projection map.  This theorem gives a satisfactory
classification of exchangeable recipes on palettes of order $1$, and
the later work of Hoover, Aldous and Kallenberg was motivated by an effort to generalise this special case.
\end{example}

\begin{example}[Erd\H{o}s-Renyi hypergraphs]\label{eph}
Let $Z = \{0,1\}_k$ for some $k \geq 1$, and let $0 < p < 1$.  Then
we can define the exchangeable $Z$-recipe $\mu: \pt \to Z$ by
setting $\mu^{(V)} = \prod_{e \in \binom{V}{k}} \mu_{p,e}$ for all
vertex sets $V$, where we identify $\{0,1\}_k^{(V)}$ with $\prod_{e
\in \binom{V}{k}} \{0,1\}_k^{(e)}$, and $\mu_{p,e} \in \Pr(
\{0,1\}_k^{(e)} )$ is the law of the random hypergraph of order $k$
on $e$ which is complete with probability $p$ and empty with
probability $1-p$; thus $\mu^{(V)}$ is the law of a random
undirected hypergraph of Erd\H{o}s-Renyi type.
\end{example}

\begin{example}[Random complete bipartite graph]\label{rcbg} 
Let $Z = (\pt,\{0,1\},\{0,1\})$, and let $Q_1 \in \Pr(\{0,1\})$ be
the uniform measure on $\{0,1\}$.  From Example \ref{rvc}, $Q_1$
induces a $Z_{\leq 1}$-exchangeable recipe $P_1: \pt \to Z_{\leq
1}$.  We also define a natural transformation $P_2: Z_{\leq 1} \to
Z$ by the formula $P_2^{(V)}(z) := \delta_z \times \prod_{e \in
\binom{V}{2}} Q_2^{(e)}(z\downharpoonright_e)$ for all vertex sets
$V$, where we identify $Z^{(V)}$ with $Z_{\leq 1}^{(V)} \times
\prod_{e \in \binom{V}{2}} Z_{=2}^{(e)}$, and for any $e = \{v,w\}$
and $z = (z_v,z_w) \in Z_{\leq 1}^{(e)}$, $Q_2^{(e)}(z)$ is the law
of the random graph on $e$ which is complete when $z_v \neq z_w$ and
empty otherwise.  The recipe $\mu := P_2 \circ P_1$ is then an
exchangeable $Z$-recipe, which describes a random complete bipartite
graph on any given vertex set $V$.
\end{example}

\begin{example}[Erd\H{o}s-Renyi graphs with random density]  Let $Z = (Z_0,\pt,\{0,1\})$, let $P_0 \in \Pr(Z_0)$ be a probability measure, and let $p: Z_0 \to [0,1]$ be a measurable function.  We can view $P_0$ as a natural transformation $P_0: \pt \to Z_{\leq 0}$.  We can then define the natural transformation $P_2: Z_{\leq 1} \to Z$ by setting $P_2^{(V)}(z_0) = \delta_{z_0} \times \prod_{e \in \binom{V}{2}}\mu_{p(z_0),e}$ for all vertex sets $V$ and all $z_0 \in Z_0 \equiv Z_{\leq 1}^{(V)}$, where we identify $Z^{(V)}$ with $Z_{\leq 1}^{(V)} \times \prod_{e \in \binom{V}{2}} Z_{=2}^{(e)}$, and $\mu_{p,e}$ are the measures defined in Example \ref{eph}.  Then $\mu := P_2 \circ P_0$ is an exchangeable $Z$-coloured hypergraph, which describes an Erd\H{o}s-Renyi random graph whose expected edge density $p$ is itself a random variable.
\end{example}

\begin{example}[Random directed complete graph]  Let $Z = \{0,1\}_2$ and let $P_2: \pt \to Z$ be exchangeable $Z$-recipe $P_2^{(V)} = \prod_{e \in \binom{V}{2}} Q_2^{(e)}$, where for each $e = \{v,w\}$, $Q_2^{(e)} \in \Pr( \{0,1\}_2^{(e)} )$ is the law of the random directed graph $G_2: \Hom([2],e) \to \{0,1\}$ such that $G_2(v,w) = 1$ and $G_2(w,v)=0$ with probability $1/2$, and $G_2(v,w)=0$ and $G_2(w,v)=1$ with probability $1/2$.  Thus $P_2^{(V)}$ is the law of a random directed complete graph on $V$, on which given any two vertices $v$ and $w$, exactly one of the directed edges $(v,w)$ and $(w,v)$ will lie in the graph, with an equal probability $1/2$ of each.
\end{example}

\begin{example}[Random 3-uniform hypergraphs]\label{r3h} We now consider a somewhat more general example than those above. Let $Z = (\pt,Z_1,Z_2,\{0,1\})$, let $Q_1 \in \Pr(Z_1)$ be a
probability measure, let $Q_2: Z_1 \times Z_1 \rightsquigarrow Z_2$
be a symmetric probability kernel, and let $p: Z_{\leq 2}^{([3])} \to
[0,1]$ be a measurable function which is symmetric with respect to
the $\Hom([3],[3])$ action on the base $Z_{\leq 2}^{([3])} \equiv Z_1^3
\times Z_2^6$.  (Actually, for this construction, only the values of
$p$ on \emph{undirected} hypergraphs in $Z_{\leq 2}^{([3])}$ - a set
which is identifiable with $Z_1^3 \times Z_2^3$ - will be relevant.)
From Example \ref{rvc} with $Z_0 = \pt$, $Q_1$ induces a natural
transformation $P_1: \pt \to Z_{\leq 1}$.  Similarly, the map $Q_2$
induces a natural transformation $P_2: Z_{\leq 1} \to Z_{\leq 2}$
defined by $P_2^{(V)}(z) := \delta_z \times \prod_{e \in
\binom{V}{2}} Q_2^{(e)}(z\downharpoonright_e)$ for all vertex sets
$V$ and $z \in Z_{\leq 1}^{(V)}$, where for each $e = \{v,w\}$,
$Q_2^{(e)}(z_v,z_w)$ is the law of the random hypergraph $G_e$ in
$Z_{=2}^{e}$ which is symmetric (thus $G_e(v,w)=G_e(w,v)$) and such
that $G_e(v,w)$ has law $Q_2(v,w)=Q_2(w,v)$.  The function $p$ also
induces a natural transformation $P_3: Z_{\leq 2} \to Z$, defined by
$P_3^{(V)}(z) := \delta_z \times \prod_{e \in \binom{V}{3}}
Q_3^{(e)}(z\downharpoonright_e)$ for all vertex sets $V$ and $z \in
Z_{\leq 2}^{(V)}$, where $Q_3^{(\{v_1,v_2,v_3\})}(y)$ is the law of
the random hypergraph in $\{0,1\}_3^{(\{v_1,v_2,v_3\})}$ which is
complete with probability $p( Z_{\leq 2}^{(v_1,v_2,v_3)}(y) )$ and
empty otherwise (note that the exact ordering of $\{v_1,v_2,v_3\}$
is irrelevant due to the symmetry assumptions on $p$).  This generates an exchangeable $Z$-recipe $\mu := P_3 \circ P_2 \circ P_1$, which creates a hypergraph on any vertex set $V$ by first using $Q_1$ to colour the vertices, then $Q_2$ to colour $2$-edges, and finally $Q_3$ to colour $3$-edges.  This sort of recipe has also appeared, for example, in the different
formalism of `hypergraphons' studied in~\cite{EleSze}, and can be viewed as the infinitary analogue of the regularisations of finitary hypergraphs given for instance in \cite{gowers-3} or \cite{rodl}.
\end{example}

These examples can be generalized to create exchangeable recipes of
any given order.  To do this, we introduce some more notation.

\begin{definition}[Independence]  Let $X$ be a sub-Cantor space with a probability measure $\mu \in \Pr(X)$, and let $\pi_\alpha: X \to Y_\alpha$, $\alpha \in A$ be a collection of measurable maps to other sub-Cantor spaces $Y_\alpha$.  We say that the maps $\pi_\alpha$ are \emph{jointly independent relative to $\mu$} if we have
$$ \int_X (\prod_{\alpha \in A'} f_\alpha \circ \pi_\alpha)\ d\mu = \prod_{\alpha \in A'} \int_X f_\alpha \circ \pi_\alpha\ d\mu$$
for all finite subsets $A'$ of $A$ and all bounded measurable functions $f_\alpha: Y_\alpha \to \R$.
\end{definition}

\begin{remark} If we choose $x \in X$ at random with law $\mu$, then the $\pi_\alpha$ are jointly independent relative to $\mu$ if and only if the random $\pi_\alpha(x) \in Y_\alpha$ are jointly independent in the usual probabilistic sense.
\end{remark}

\begin{definition}[$j$-independence]\label{j-indep}  Let $N: Z \to Y$ be a natural transformation, and let $j \geq 0$.  We say that $N$ is \emph{$j$-independent} if for every vertex set $V$ and every $z \in Z^{(V)}$, the restriction maps $\pi_W: Y^{(V)} \to Y^{(W)}$ for $W \in \binom{V}{j}$ are jointly independent relative to the measure $N^{(V)}(z) \in \Pr(Y^{(V)})$.
\end{definition}

\begin{remark} Informally, $j$-independence asserts that for any fixed $z \in Z^{(V)}$, the $j$-edges of the random element of $Y^{(V)}$ drawn using the law $N^{(V)}(z)$ are jointly independent random variables.  For instance in Example \ref{rvc}, once one fixes $z \in Z_0$, $P_1$ colours the vertices in $V$ independently with law $Q_1(z)$, and thus $P_1: Z_{<0} \to Z_{\leq 1}$ is $1$-independent.  (Note, however, that if $z \in Z_0$ is chosen randomly rather than deterministically, then the colours assigned to vertices in $V$ need not be independent any more.) More generally, in all of the examples discussed earlier in this section, the natural transformations $P_j: Z_{<j} \to Z_{\leq j}$ that appear in those examples are $j$-independent.
\end{remark}

\begin{example}\label{jep}  Let $Z$ be a contravariant functor and $j \geq 0$.  Let $Y_{=j}$ be a sub-Cantor palette with only the $j^\th$ component non-trivial, and let $Q^{([j])}: Z^{([j])} \rightsquigarrow Y_{=j}^{([j])}$ be a probability kernel which is $\Hom([j],[j])$-equivariant, thus $Y_{=j}^{(\phi)} \circ Q^{([j])} = Q^{([j])} \circ Z^{(\phi)}$ for all $\phi \in \Hom([j],[j])$.  If we then define the natural transformation $Q: Z \to Y_{=j}$ by
$$  Q^{(V)}(z) := \prod_{e \in \binom{V}{j}} Y_{=j}^{(\phi_e^{-1})} \circ Q^{([j])}( Z^{(\phi_e)}(z) )$$
for all vertex sets $V$ and all $z \in Z^{(V)}$, where we identify $Y_{=j}^{(V)}$ with $\prod_{e \in \binom{V}{j}} Y_{=j}^{(e)}$, and where we choose an arbitrary morphism $\phi_e \in \Hom([j],e)$ for each $e \in \binom{V}{j}$ (the exact choice of $\phi_e$ is irrelevant, thanks to the $\Hom([j],[j])$-equivariance of $Q^{([j])}$), then one verifies that $Q$ is a $j$-independent natural transformation.  Indeed, $Q$ is the unique $j$-independent natural transformation which agrees with $Q^{([j])}$ at $[j]$, and conversely every $j$-independent natural transformation from $Z$ to $Y_{=j}$ arises in this fashion.
\end{example}

\begin{definition}[Regular exchangeable recipes]\label{regexp}  Let $Z$ be a sub-Cantor palette of some order $k \geq 0$, and let $\mu: \pt \to Z$ be an exchangeable $Z$-recipe.  We say that $\mu$ is \emph{regular} if there exists a factorisation
$$ \mu = P_k \circ \ldots \circ P_0$$
where for each $0 \leq j \leq k$, $P_j: Z_{<j} \to Z_{\leq j}$ is a $j$-independent natural transformation which partially inverts the projection natural transformation $\pi_j: Z_{\leq j} \to Z_{<j}$ in the sense that $\pi_j \circ P_j = \id_{Z_{<j}}$.  (Here we identify $Z$ with $Z_{\leq k}$ in the obvious manner.)
\end{definition}

\begin{remark} If one sets $\mu_{\leq j} = \mu_{<j+1} := P_j \circ \ldots \circ P_0$, then the situation can be described by a commutative diagram whose $j^\th$ layer for $j=0,\ldots,k$ takes the form
\begin{equation}\label{mupj}
\begin{CD}
\pt  @>{\mu_{\leq j}}>> Z_{\leq j}             @>{\id_{Z_{\leq j}}}>>        Z_{\leq j} \\
@|                       @AA{P_j}A                                          @V{\pi_j}VV \\
\pt  @>{\mu_{<j}}>>       Z_{<j}                 @>{\id_{Z_{<j}}}>>            Z_{<j}
\end{CD}
\end{equation}
\end{remark}

\begin{example}  Let $Z$ be a sub-Cantor palette of order $k \ge 0$.  If $0 \leq j \leq k$ and $Q_j^{([j])}: Z_{<j}^{([j])} \rightsquigarrow Z_{=j}^{([j])}$ is a $\Hom([j],[j])$-equivariant probability kernel, and $Q_j: Z_{<j} \to Z_{=j}$ is the associated $j$-independent natural transformation, as defined by Example \ref{jep}, then the natural transformation $P_j: Z_{<j} \to Z_{\leq j}$ defined by $P_j^{(V)}(z) := \delta_z \times Q_j^{(V)}(z)$ for all vertex sets $V$ and $z \in Z^{(V)}$, where we identify $Z_{\leq j}^{(V)}$ with $Z_{<j}^{(V)} \times Z_{=j}^{(V)}$, is a $j$-independent natural transformation with $\pi_j \circ P_j = \id_{Z_{<j}}$.  Thus by selecting a $Q_j^{([j])}$ for each $0 \leq j \leq k$ and then composing the resulting $P_j$ together, one obtains a regular exchangeable $Z$-recipe $\mu$; conversely, all such regular exchangeable $Z$-recipes arise in this manner.
\end{example}

In terms of the notation set out above, we can now state the full structure
theorem that we need.

\begin{theorem}[Structure theorem]\label{struct}
Let $K$ be a finite palette of some order $k \geq 0$, let $\mu: \pt
\to K$ be an exchangeable $K$-recipe and let $S$ be a countably
infinite vertex set.  Then there exists a sub-Cantor palette $Z$, a
deterministically continuous natural transformation $\L: K^{\uplus S} \to Z$,
and a colouring map $\kappa:Z\to K$ such that the natural
transformation $\overline{\kappa} \circ \Lambda: K^{\uplus S} \to K$
is just the restriction map, thus
\begin{equation}\label{kappalam}
(\overline{\kappa} \circ \Lambda)^{(V)}(G) = G\downharpoonright_V
\end{equation}
for all vertex sets $V$ and all $G \in K^{(V \uplus S)}$, and such
that $\Lambda \circ \mu^{\uplus S}: \pt \to Z$ is a regular
exchangeable $Z$-recipe.
\end{theorem}

\begin{remark}
The situation in the structure theorem can be summarised by the following commutative diagram,
\begin{equation}\label{mustruct}
\begin{CD}
K^{\uplus S}             @>{\Lambda}>>   Z         @>{\overline{\kappa}}>>    K \\
@AA{\mu^{\uplus S}}A        @AA{\Lambda \circ \mu^{\uplus S}}A              @AA{\mu}A\\
\pt                       @=            \pt          @=                       \pt
\end{CD}
\end{equation}
with the map from $K^{\uplus S}$ to $K$ being the restriction map (by \eqref{kappalam}), and the middle vertical map being an exchangeable $Z$-recipe and thus factorable as $P_k \circ \ldots \circ P_0$ for some $j$-independent ingredients $P_j: Z_{<j} \to Z_{\leq j}$.
\end{remark}

\begin{proof}
See Theorems~3.15 (for the undirected case) and Theorem~3.22 (for
the general case) in~\cite{Aus1}.
\end{proof}

Informally, the above theorem asserts that any exchangeable recipe can (after adding a sufficient number of ``hidden variables'') be constructed from randomly colouring $0$-edges, then $1$-edges, then $2$-edges, etc. as in Examples \ref{rvc}-\ref{r3h}.  It is analogous to the hypergraph regularity lemma, which roughly speaking asserts that any $k$-uniform hypergraph $G = (V,E)$ can be regularised by first colouring (i.e. partitioning) the $1$-edges (i.e. vertices), and then on each pair of $1$-cells, colouring/partitioning the $2$-edges between those cells in a regular fashion (regularity being the analogue of $2$-independence), then on each triplet of $1$-cells and triplet of $2$-cells, colouring/partitioning the $3$-edges with vertices in the $1$-cells and $2$-edges in the $2$-cells in a (hypergraph)-regular fashion, and so forth.

\subsection{Infinitary reductions of main theorems}\label{reduce-sec}

In this section, we use the structure theorem to deduce the main positive results of this paper (Theorems \ref{rs-thm-dir}, \ref{lgr}, \ref{monotone}, \ref{part}) from infinitary counterparts (Proposition \ref{rs-prop}, \ref{lgr-prop}, \ref{monotone-prop}, \ref{part-prop}), which will then be proven in later sections.

We begin with some notation.

\begin{definition}[Entailment]  Let $K$ be a finite palette, let $\P$ be a hereditary $K$-property, and let $N: Z \to K$ be a natural transformation from some contravariant functor $Z$.  We say that $N$ \emph{almost entails} $\P$ if we have $N^{(V)}(z)(\P^{(V)}) = 1$ for all vertex sets $V$ and all $z \in Z^{(V)}$.  We say that $N$ \emph{entails} $\P$ if $N$ is deterministically continuous and almost entails $\P$.
\end{definition}

\begin{remark} If $N$ is deterministically continuous, then $N^{(V)}$ can be viewed as a continuous function from $Z^{(V)}$ to $K^{(V)}$, and then the assertion that $N$ entails $\P$ is equivalent to the claim that $N^{(V)}(Z^{(V)}) \subset \P^{(V)}$.  Note that this notation of entailment is consistent with that in Definition \ref{locmod} after using Example \ref{lmr-nat}.
\end{remark}

\begin{remark}[Alon-Shapira finitisation trick]\label{sat} From \eqref{kphi} in Definition \ref{hered} and \eqref{natprop} in Definition \ref{natdef} we see that to verify that $N$ almost entails $\P$, it suffices to verify $N^{(V)}(z)(\P^{(V)}) = 1$ for a single countably infinite vertex set $V$ and all $z \in Z^{(V)}$.  Actually, from countable additivity and the way $\P$ is extended from finite hypergraphs to infinite ones, it suffices to verify $N^{(V)}(z)(\P^{(V)}) = 1$ for all finite vertex sets $V$ and all $z \in Z^{(V)}$.  For similar reasons, to verify that a continuous natural transformation $N: Z \to K$ entails $\P$, it suffices to show that $N^{(V)}(Z^{(V)}) \subset \P^{(V)}$ for all finite $V$.  This ability to reduce entailment to verification on finite vertex sets is crucial to our arguments; not coincidentally, an analogous finitisation observation played a similarly central role in \cite{AloSha2}.
\end{remark}


\begin{definition}[Infinitary repairability and testability]\label{infitest} Let $K$ be a finite palette of some order $k \geq 0$, and let $\P$ be a hereditary $K$-property.  We say that $\P$ is \emph{infinitarily testable with one-sided error} (resp. \emph{infinitarily strongly locally repairable}) if given any sub-Cantor palette $Z$ of order $k$, any colouring $\kappa: Z \to K$, any regular exchangeable $Z$-recipe $\mu: \pt \to Z$ such that $\overline{\kappa} \circ \mu$ almost entails $\P$, and every $\varepsilon > 0$, there exists a weakly continuous (resp. deterministically continuous) natural transformation $T: Z \to K$ that almost entails (resp. entails) $\P$ and is close to $\overline{\kappa}$ in the sense that
\begin{equation}\label{mukan}
\int_{Z^{([k])}} T^{([k])}(z)( K^{([k])} \backslash
\{\overline{\kappa}^{([k])}(z)\} )\ d\mu^{([k])}(z) < \varepsilon.
\end{equation}
\end{definition}

\begin{remark}\label{jojo-rem} When $T$ is deterministically continuous, \eqref{mukan} simplifies to
\begin{equation}\label{jojo}
\mu^{([k])}( \{ z \in Z^{([k])}: T^{([k])}(z) \neq \overline{\kappa}^{([k])}(z) \} ) < \varepsilon.
\end{equation}
\end{remark}

\begin{example}[Testing and repair of the triangle-free property, I]\label{triangle-test}  Let $Z = (\pt, Z_1, \{0,1\})$ be a sub-Cantor palette, let $K := \{0,1\}_2$, and let $\kappa: Z \to K$ be the colouring map which is the identity on the order $2$ component and trivial on lower order components.  Let $Q_1 \in \Pr(Z_1)$; and let $P_1: \pt \to Z_{\leq 1}$ be as in Example \ref{rvc}.  Let $p: Z_1 \times Z_1 \to [0,1]$ be a symmetric measurable function, and let $P_2: Z_{\leq 1} \to Z$ be the $2$-independent natural transformation $P_2^{(V)}(z) := \delta_z \times \prod_{e \in
\binom{V}{2}} Q_2^{(e)}(z\downharpoonright_e)$ for all vertex sets
$V$ and $z \in Z_{\leq 1}^{(V)}$, where for each $e = \{v,w\}$,
$Q_2^{(e)}(z_v,z_w)$ is the law of the random graph on the doubleton $e$ which is complete with probability $p(z_v,z_w)$ and empty otherwise.  Then $\mu := P_2 \circ P_1$ is a regular exchangeable $Z$-recipe (closely related to the \emph{graphons} introduced in \cite{LovSze2}); it randomly colours any vertex set $V$ by assigning each vertex $v \in V$ a random colour $G_1(v)$ in $Z_1$ with law $Q_1$, and then assigns any edge $\{v,w\}$ the colour $1$ with probability $p(G_1(v),G_1(w))$, independently for all edges $\{v,w\}$ (once the colours $G_1(v)$ have all been picked).  Let $\P$ be the hereditary $K$-property of being undirected and triangle-free.  Observe that $\mu$ will almost entail $\P$ if we have $p(x,y)p(y,z)p(z,x)=0$ for $Q_1$-almost every $x,y,z \in Z_1$; suppose that this is the case.  Now we seek a weakly (resp. deterministically) continuous natural transformation $T: Z \to K$ that almost entails (resp. entails) $\P$, and is close to $\overline{\kappa}: Z \to K$ (observe that $\overline{\kappa}$ itself does not entail $\P$ at all) in the sense of \eqref{mukan}.  We know of two methods to achieve this, which we shall call the \emph{R\"odl-Schacht method} and the \emph{Alon-Shapira method}, being loosely based on the constructions in \cite{rs2} and \cite{AloSha2} respectively (we will also discuss finitary analogues of these schemes in the next remark).  Both methods proceed by first choosing a refinement $\alpha: Z \to A$ of $\kappa: Z \to K$, which amounts to subdividing the vertex space $Z_1$ into finitely many clopen ``cells'' $\alpha_1^{-1}(\{a\})$; the finer one takes the colouring $\alpha$, the better the value of $\varepsilon$ one will eventually obtain in \eqref{mukan}. The R\"odl-Schacht method then constructs the law $T^{(V)}(z) \in \Pr(K^{(V)})$ of a random $K$-coloured graph on a vertex set $V$, starting from a $Z$-coloured graph $z \in Z^{(V)}$, as follows.  For each vertex $v \in V$, one looks at the cell $C_v := \alpha_1^{-1}(\alpha(z_1(v))) \in Z_1$ that $z_1(v)$ lives in.  If this cell has positive measure with respect to $Q_1$, then we select a point $\zeta_v \in C_v$ at random with law $(Q_1|C_v)$ (see Appendix \ref{prob} for the definition of conditioned measure).  Otherwise, we select $\zeta_v \in Z_1$ with law $Q_1$.  Note that in either case, the law of $\zeta_v$ is absolutely continuous with respect to $Q_1$.  We perform this selection procedure independently for each $v \in V$.  One now selects $T^{(V)}(z)_2(v,w) = T^{(V)}(z)_2(w,v) \in \{0,1\}$ for each $(v,w) \in \Hom([2],V)$ separately by the following rule:
\begin{itemize}
\item If $z_2(v,w)=z_2(w,v)$ and $p(\zeta_v,\zeta_w) = p(\zeta_w,\zeta_v) \neq 1-z_2(v,w)$, then set $T^{(V)}(z)_2(v,w)=z_2(v,w)$ and $T^{(V)}(z)_2(w,v)=z_2(w,v)$.
\item Otherwise, set $T^{(V)}(z)_2(v,w)=T^{(V)}(z)_2(w,v)$ equal to $1$ with probability $p(\zeta_w,\zeta_v)$, and equal to 0 otherwise.
\end{itemize}
One can verify that $T$ is a weakly continuous natural transformation which almost entails $\P$, and that \eqref{mukan} is obeyed for sufficiently fine $\alpha$, which demonstrates that $\P$ is infinitarily testable with one-sided error.

Now we turn to the Alon-Shapira method, which is more complicated, but constructs a natural transformation $T: Z \to K$ which is deterministically continuous rather than weakly continuous.  To simplify matters we shall take advantage of the monotonicity of the property $\P$, and also make the additional assumption that the measure $Q_1$ is \emph{atomless} (i.e. $Q_1(\{z\})=0$ for all $z \in Z_1$).  Let $\alpha: Z \to A$ be as before.  For each $a \in A_1$ independently in turn, we construct the cell $C_a := \alpha_1^{-1}(\{a\})$, and select $\zeta_a \in Z_1$ at random with law $(Q_1|C_a)$ if $Q_1(C_a) > 0$, and with law $Q_1$ otherwise.  For each pair $\{a,a'\} \in \binom{A_1}{2}$, we then select $\zeta_{a,a'} = \zeta_{a',a} \in \{0,1\}$ independently at random, such that $\zeta_{a,a'}=1$ with probability $p( \zeta_a, \zeta_{a'} )$.  With all these choices, we then define the (random) deterministically continuous natural transformation $T: Z \to K$ by setting $T^{(V)}(z)_2(v,w)=T^{(V)}(z)_2(w,v)$ for vertex sets $V$, $z \in Z^{(V)}$, and $(v,w) \in \Hom([2],V)$ by the following rule:
\begin{itemize}
\item If $\alpha_1(v)\neq \alpha_1(w)$, $z_2(v,w)=z_2(w,v)$, and $p(\zeta_{\alpha_1(v)},\zeta_{\alpha_1(w)})=p(\zeta_{\alpha_2(v)},\zeta_{\alpha_1(w)}) \neq 1-z_2(v,w)$, then we set $T^{(V)}(z)_2(v,w)=T^{(V)}(z)_2(w,v) = z_2(v,w)$.
\item If $\alpha_1(v)\neq \alpha_1(w)$ but we are not in the previous case, we set $T^{(V)}(z)_2(v,w) = \zeta_{\alpha_1(v),\alpha_1(w)}$.
\item If we are in the ``diagonal case'' $\alpha_1(v)=\alpha_1(w)$ then we set $T^{(V)}(z)_2(v,w) = \zeta_{\alpha_1(v),\alpha_1(w)} = 0$.
\end{itemize}
One can verify that with probability 1, $T$ is a deterministically continuous transformation which entails $\P$; the monotonicity of $\P$ is used to ensure that the ``zeroing out'' of the diagonal case does not interfere with this entailment.  One can also verify \eqref{mukan} if the colouring $\alpha$ is sufficiently fine; the atomless nature of $Q_1$ is used to ensure that the contribution of the diagonal case can be made arbitrarily small.  (One can handle the diagonal contributions of any atoms in $Q_1$ by adding an additional case to the above rule; we leave the details to the interested reader.)  If $\P$ is not monotone, the diagonal case causes much more difficulty, and needs to be coloured according to a colour provided by an application of Ramsey's theorem; see \cite{AloSha2} for details (albeit in a rather different language).
\end{example}

\begin{example} Testing and repair of the triangle-free property, II]  We now adapt the above discussion to the finitary setting, to help provide a partial dictionary between the finitary and infinitary worlds.  Our discussion will be somewhat informal.  We start with a fixed \emph{graphon} - a measurable symmetric function $p: [0,1] \times [0,1] \to [0,1]$ by the following procedure.  Given such a graphon, and given a vertex set $V$, we construct a random graph $G = (V,E)$ by the following procedure.  First, randomly assign to each vertex $v \in V$ a colour $G_1(v) \in [0,1]$ using the uniform distribution on $[0,1]$, with each vertex being coloured independently.  (Note that the uniform distribution on $[0,1]$ is atomless, thus avoiding some of the technicalities alluded to in the previous example.)  Next, we define the edge set $E$ of $G$ by declaring each edge $\{v,w\}$ to lie in $E$ with probability $p(G_1(v),G_1(w))$, with these events being independent once the colours of the vertices have been chosen.  

The finitary analogue of the R\"odl-Schacht method involves two vertex sets\footnote{Of course, in the initial setup in \cite{rs2} no graphon is initially provided.  Instead, one takes a hypothetical sequence of increasingly large counterexamples to the local testability claim, passes to a subsequence which does converge to a graphon $p$ (cf. Lemma \ref{compact}), selects two widely separated elements of this sequence, and then applies the argument described here.}, a relatively small one $V$ and a very large one $V^*$, and generates two random graphs $G = (V,E)$, $G^* = (V^*,E^*)$ using the same graphon $p$.  We assume that the large graph $G^*$ is very close to being triangle-free, and in particular we assume that the triangle density of $G^*$ is extremely small compared to the size $|V|$ of the smaller graph.  On the other hand, the triangle density of $G^*$ is extremely close to the quantity
\begin{equation}\label{evo}
 \int_0^1 \int_0^1 \int_0^1 p(x,y) p(y,z) p(z,x)\ dx dy dz.
\end{equation}
We may thus assume (with high probability) that this quantity is very small - smaller than any quantity depending only on $|V|$.

We then use the nearly triangle-free nature of $G^*$ to obtain a genuinely triangle-free perturbation $G' = (V,E')$ of $G$ as follows.  Pick a large number $N$ (much larger than $|V|$) and subdivide the intervals $[0,1]$ into $N$ intervals $I_1,\ldots,I_N$ of equal length.  We then define a random map $\zeta: V \to [0,1]$ as follows.  For each $v \in V$, we look at the colour $G_1(v) \in [0,1]$ of $v$; this falls into one of the intervals $I_i$ of $[0,1]$.  We then pick an element of $I_i$ uniformly at random and call this $\zeta_v$.  (Note that different $v, v' \in V$ may correspond to the same $i$, but in such cases we pick $\zeta_v, \zeta_{v'}$ independently; in any event, such collisions will be rare if $N$ is chosen large enough depending on $|V|$.)  This gives rise to a random map $\zeta: V \to [0,1]$.  From the smallness of \eqref{evo} and the first moment method we see that the quantity
\begin{equation}\label{poof}
 \sum_{u,v,w \in V, \hbox{ distinct}} p( \zeta_u, \zeta_v ) p( \zeta_v, \zeta_w ) p( \zeta_w, \zeta_u )
\end{equation}
can be made (with high probability) to be as small as desired depending on $|V|$. 

We use this map $\zeta$ to construct $G'=(V,E')$ as follows.  We will need a small threshold $\sigma > 0$ depending on $|V|$.  Let $\{v,w\}$ be an edge in $V$.
\begin{itemize}
\item If $\{v,w\} \in E$, and $p( \zeta_v, \zeta_w ) \geq \sigma$, we place $\{v,w\}$ in $E'$.
\item If $\{v,w\} \in E$, and $p( \zeta_v, \zeta_w ) < \sigma$, we exclude $\{v,w\}$ from $E'$.
\item If $\{v,w\} \not \in E$ and $p( \zeta_v, \zeta_w ) \leq 1-\sigma$, we exclude $\{v,w\}$ from $E'$.
\item If $\{v,w\} \not \in E$ and $p( \zeta_v, \zeta_w ) < 1-\sigma$, we place $\{v,w\}$ in $E'$.
\end{itemize}
One can check (if \eqref{poof} is sufficiently small) that $G'$ is genuinely triangle-free; meanwhile, from the Lebesgue differentiation theorem we know that $p$ is approximately constant on most cells $I_i \times I_j$; since $G$ is generated using $p$, this can be used to show that $G$ and $G'$ differ in a relatively small number of edges if the parameters are selected correctly (it is here that it is crucial that $N$ is large compared with $|V|$).  Note however that the rule generating $G'$ from $G$ is not local in nature, as it requires an initial assignment of a real number $\zeta_v \in [0,1]$ to each vertex $v$ and thus requires far more ``memory'' than is available to a local modification rule.  Also, the ``complexity'' $N$ of the modification procedure here has to be large compared with $V$, and in particular this procedure would not work if $V$ were infinite.

Now we briefly sketch the Alon-Shapira approach to constructing $G'$.  Here we will not use the large graph $G^*$, and work solely with $G$.  We assume that $G$ is close to triangle-free, thus we may assume that \eqref{evo} is small; but now the bound is much weaker.  More precisely, for any $\delta > 0$, we may assume that \eqref{evo} is less than $\delta$, but only if $|V|$ is sufficiently large depending on $\delta$; we no longer have the luxury of assuming \eqref{evo} to be arbitrarily small depending on $V$.

We now construct the perturbation $G'$ by a variant of the R\"odl-Schacht method.  We pick an $N$ which is moderately large, but now \emph{independent} of $|V|$, and create the intervals $I_1,\ldots,I_N$ as before; this induces a partition $V = V_1 \cup \ldots \cup V_N$ of $V$ into cells which (with high probability) are of roughly equal size.  Rather than assign a number $\zeta_v \in [0,1]$ to each vertex $v \in V$, we now only assign a number $\zeta_i \in I_i$ for each $1 \leq i \leq N$, drawn uniformly at random from $I_i$ and independently for each $i$.  We then construct $G' = (V,E')$ as follows for $v \in V_i, w \in V_j$, this time with a threshold $\sigma > 0$ that is small compared with $N$, but independent of $|V|$:
\begin{itemize}
\item If $\{v,w\} \in E$, $i \neq j$, and $p( \zeta_i, \zeta_j ) \geq \sigma$, we place $\{v,w\}$ in $E'$.
\item If $\{v,w\} \in E$, $i \neq j$, and $p( \zeta_i, \zeta_j ) < \sigma$, we exclude $\{v,w\}$ from $E'$.
\item If $\{v,w\} \not \in E$, $i \neq j$, and $p( \zeta_i, \zeta_j ) \leq 1-\sigma$, we exclude $\{v,w\}$ from $E'$.
\item If $\{v,w\} \not \in E$, $i \neq j$, and $p( \zeta_i, \zeta_j ) < 1-\sigma$, we place $\{v,w\}$ in $E'$.
\item If $i=j$, we exclude $\{v,w\}$ from $E'$.
\end{itemize}
This procedure will (if $V$ is large enough to ensure \eqref{evo} sufficiently small depending on $N,\sigma$) create a triangle-free graph $G'$ which is close (with high probability) to $G$.  Technically, $G'$ is not obtained from $G$ from a local modification rule; however, the rule that decides when an edge $\{v,w\}$ belongs to $G'$ depends only on whether $\{v,w\}$ lies in $G$, and on the cells $V_i, V_j$ that $v,w$ lie in.  As mentioned before, the $V_i$ can be viewed as a Szemer\'edi partition of the graph $G$.  Another way to obtain a Szemer\'edi partition is to select a number of random vertices $v_1,\ldots,v_k$ and use the neighbourhoods of these vertices to determine a partition; see e.g. \cite{ishi-0}.  Using such a regularisation instead of the one based on the intervals $I_1,\ldots,I_n$, one can eventually obtain a local modification rule that repairs $G$ to a triangle-free graph and which only modifies a small number of edges on the average; we omit the details, which could in principle be extracted from the argument in \cite{AloSha2} using random vertex neighbourhoods to regularise graphs as in \cite{ishi-0}.
\end{example}

The connection of the notions in Definition \ref{infitest} to the those in Definitions \ref{Testdef}, \ref{locrep} is given by the following correspondence principle.

\begin{proposition}[Correspondence principle]\label{tprop} Let $K$ be a finite palette, and let $\P$ be a hereditary $K$-property.
\begin{itemize}
\item[(i)] If $\P$ is infinitarily testable with one-sided error, then $\P$ is testable with one-sided error.
\item[(ii)] If $\P$ is infinitarily strongly locally repairable, then $\P$ is strongly locally repairable.
\end{itemize}
\end{proposition}

\begin{proof}  Let $k$ denote the order of $K$.  We first prove (i).  Suppose for contradiction that $\P$ is infinitarily testable with one-sided error but not testable with one-sided error.  Carefully negating all the quantifiers, we conclude that there exists an error tolerance $\eps > 0$ and a sequence $G_n \in K^{(V_n)}$ of $K$-coloured hypergraphs on finite vertex sets $V_n$ with $|V_n| \geq \max(n,k)$, which increasingly locally obey $\P$ in the sense that
\begin{equation}\label{gnp}
\frac{1}{|\binom{V_n}{n}|} \left|\left\{ W \in \binom{V_n}{n}: G_n\downharpoonright_W \in \P^{(W)} \right\}\right| \geq 1-\frac{1}{n},
\end{equation}
but is far from $\P$ in the sense that for any $n$, there does not exist any $G'_n \in \P^{(V_n)}$ for which
\begin{equation}\label{joy}
\frac{1}{|\binom{V_n}{k}|} \left|\left\{ W \in \binom{V_n}{k}: G_n\downharpoonright_W \neq G'_n\downharpoonright_W \right\}\right| \leq \eps.
\end{equation}
From \eqref{gnp} and the hereditary nature of $\P$ we easily see that
\begin{equation}\label{gnp2}
\lim_{n \to \infty} \frac{1}{|\binom{V_n}{N}|} \left|\left\{ W \in \binom{V_n}{N}: G_n\downharpoonright_W \in \P^{(W)} \right\}\right| = 1
\end{equation}
for all fixed $N \ge 1$.

We now arbitrarily extend each $G_n \in K^{(V_n)}$ to a $K$-coloured continuous $\tilde G_n$ on $V_n$ as in Example \ref{extend}, endowing each $V_n$ with uniform probability measure $\nu_n$.  By Definition \ref{sampling}, we thus have a sequence of exchangeable $K$-recipes $\overline{\tilde G_n} \circ \overline{\nu_n}: \pt \to K$.  From \eqref{gnp2} (and the fact that $|V_n| \to \infty$ as $n \to \infty$) we see that the $\overline{\tilde G_n} \circ \overline{\nu_n}$ increasingly entail $\P$ in the sense that
\begin{equation}\label{nuf}
\lim_{n \to \infty} (\overline{\tilde G_n} \circ \overline{\nu_n})^{([N])}( \P^{([N])} ) = 1
\end{equation}
for any $N \ge 1$.

By Lemma \ref{compact}, and passing to a subsequence if necessary,
we may assume that $\tilde G_n \circ \nu_n$ converges vaguely to an
exchangeable $K$-recipe $\mu: \pt \to K$.  From \eqref{nuf} (and the
fact that $\P^{([N])}$ is clopen) we conclude that
$\mu^{([N])}(\P^{([N])}) = 1$ for all $N$.  By Remark \ref{sat}, we
conclude that $\mu$ almost entails $\P$.

Let $S$ be a countably infinite vertex set.  We now invoke Theorem
\ref{struct} to obtain a sub-Cantor palette $Z$, a natural
transformation $\Lambda: K^{\uplus S} \to Z$ and a colouring
$\kappa: Z \to K$ such that $\Lambda \circ \mu^{\uplus S}: \pt \to
Z$ is a regular exchangeable $Z$-recipe. From \eqref{mustruct} we
see that $\overline{\kappa} \circ \Lambda \circ \mu^{\uplus S} =
\mu$, thus $\overline{\kappa} \circ \Lambda \circ \mu^{\uplus S}$
almost entails $\P$.

Let $\delta$ be a small number (depending on $\eps$ and $k$) to be chosen later.  As $\P$ is infinitarily testable with one-sided error, we can find a weakly continuous natural transformation $T: Z \to K$ that almost entails $\P$ such that
\begin{equation}\label{mukan2}
\int_{Z^{([k])}} T^{([k])}(z)( K^{([k])} \backslash
\{\overline{\kappa}^{([k])}(z)\} )\ d\Lambda^{([k])} \circ \mu^{([k]
\uplus S)}(z) < \delta.
\end{equation}

The situation can be summarised by the commutative diagram
\begin{equation}\label{tstruct}
\begin{diagram}
                  &              &        K         & \lTo^{\iota}       &    \P && \\
                  &              &       \uTo^T     & \ruTo              &       &&\\
K^{\uplus S}      & \rTo^\Lambda &       Z          & \rTo^{\overline{\kappa}} & K  &\lTo^\iota & \P\\
\uTo^{\mu^{\uplus S}} &    &   \uTo^{\Lambda \circ \mu^{\uplus S}} & & \uTo^\mu & \ruTo &\\
\pt               & \rEq         &      \pt         & \rEq               & \pt & &
\end{diagram}
\end{equation}
where the two maps $T$, $\overline{\kappa}$ are close in the sense of \eqref{mukan}.
The fact that $\mu$ almost entails $\P$ means that it in fact factors through the inclusion map $\iota: \P \to K$, and similarly for $T$.

Fix this $T$, and let $n$ be a large integer to be chosen later.  We perform the following random construction.  Let $\N$ be the natural numbers (actually, we could use any countably infinite vertex set here).  Let $\psi \in V_n^{\N}$ be a point drawn at random with law $\nu_n^{\N}$ (or equivalently, $\psi: \N \to V_n$ is a random function from $\N$ to $V_n$).  Then the point
\begin{equation}\label{zDef}
z := \Lambda^{(\N)}(\overline{\tilde G_n}^{(\N)}(\psi))
\end{equation}
is a random point in $Z^{(\N)}$ with law $(\Lambda \circ \overline{\tilde G_n} \circ \overline{\nu_n})^{(\N)}$.

After choosing $\psi$ and hence $z$, let $G \in K^{(\N)}$ be drawn at random with law $T^{(\N)}(z)$.  By construction of $T$, we see that $G$ almost surely obeys $\P$.

We now claim that for $n$ sufficiently large we have
\begin{equation}\label{pze}
 \Prob( \overline{\kappa}^{(e)}(z\downharpoonright_e) \neq G\downharpoonright_e ) < \delta
\end{equation}
for all $e \in \binom{\N}{k}$.  As the joint distribution of $(z,G)$ is exchangable with respect to the action of $\Hom(\N,\N)$, we see that the probability on the left is independent of the choice of $e$, and so it suffices to verify \eqref{pze} for $e=[k]$.  Since $T$ is a natural transformation, we observe that for fixed $z$, $G\downharpoonright_{[k]}$ has the distribution $T^{([k])}(z\downharpoonright_{[k]})$.  Also, $z\downharpoonright_{[k]}$ has the distribution $\Lambda^{([k])} \circ (\overline{\tilde G_n}\circ \overline{\nu_n})^{([k] \uplus S)}$.  We can thus re-express the left-hand side of \eqref{pze} as
$$ \int_{Z^{([k])}} T^{([k])}(z)( K^{([k])} \backslash \{ \overline{\kappa}^{([k])}(z) \} )\ d\Lambda^{([k])} \circ (\overline{\tilde G_n}\circ \overline{\nu_n})^{([k] \uplus S)}(z).$$
But as $T$ is weakly continuous, the integrand here is continuous.  Since $\overline{\tilde G_n}\circ \overline{\nu_n}$ converges vaguely to $\mu$, the claim \eqref{pze} thus follows from \eqref{mukan2}.

Now let $M = M_n$ be a large integer (depending on $|V_n|$ and $\eps$) to be chosen later.  The vertices $\psi(1),\ldots,\psi(M) \in V_n$ are drawn uniformly at random, so by the law of large numbers we see (if $M$ is sufficiently large) that with probability at least $1/2$, that we have
\begin{equation}\label{mvn}
\frac{M}{2|V_n|} \leq |\{ m \in [M]: \psi(m) = v \}| \leq \frac{2M}{|V_n|}
\end{equation}
for all $v \in V_n$.  (Note that it is crucial here that $M$ is taken large compared to $|V_n|$; it is because of this that we only obtain testability here rather than local repairability.)

We now condition on the event that \eqref{mvn} holds.  Because this event has probability at least $1/2$, we see that after this conditioning, $G$ still continues to obey $\P$ almost surely, and from \eqref{pze} we have
\begin{equation}\label{kzkz}
 \Prob( \kappa^{([k])}(z\downharpoonright_e) \neq G\downharpoonright_e ) \ll \delta
\end{equation}
for all $e \in \binom{\N}{k}$.

For any $v \in V_n$, let $m_v \in [M]$ be chosen uniformly and independently at random from the set $\{ m \in [M]: \psi(m) = v \}$, which is non-empty by \eqref{mvn}.  This gives us a random function $m \in \Hom(V_n,\N)$ which partially inverts $\psi$.  We then define the hypergraph $G'_n \in K^{(V_n)}$ by the formula $G'_n := K^{(m)}(G)$.  From \eqref{zDef} we also have $G_n = K^{(m)}(\kappa^{(\N)}(z))$.

Since $G$ almost surely obeys $\P$, the hypergraph $G'_n = K^{(m)}(G)$ does also.  From \eqref{mvn}, \eqref{kzkz}, and the construction of $m$ we also see that
$$ \sum_{W \in \binom{V_n}{k}} \Prob( G_n\downharpoonright_W \neq G'_n\downharpoonright_W ) \ll_k \delta |\Hom([k],V_n)|;$$
by linearity of expectation, we thus have
$$ \E( \frac{1}{|\binom{V_n}{k}|} \left|\left\{ W \in \binom{V_n}{k}: G_n\downharpoonright_W \neq G'_n\downharpoonright_W \right\}\right| ) \ll_k \delta.$$
Thus by the the first moment method, there exists a deterministic hypergraph $G'_n \in \P^{(V_n)}$ such that
$$ \frac{1}{|\binom{V_n}{k}|} \left|\left\{ W \in \binom{V_n}{k}: G_n\downharpoonright_W \neq G'_n\downharpoonright_W \right\}\right| \ll_k \delta.$$
Choosing $\delta$ sufficiently small depending on $\eps$, we obtain \eqref{joy}, which is a contradiction.  This concludes the proof of (i).

Now we prove (ii).
Suppose for contradiction that $\P$ is infinitarily strongly locally repairable but not strongly locally repairable.  Carefully negating all the quantifiers, we conclude that there exists an error tolerance $\eps > 0$ and a sequence of $K$-coloured continuous hypergraphs $(G_n)_{n\geq 1}$, each on a different probability space $(V_n, \B_n, \nu_n)$, which increasingly obey $\P$ in the sense that
$$ \lim_{N \to \infty} \int_{V_n^{[N]}} \I(\overline{G_n}^{([N])}(v) \in \P^{([N])})\ d\nu^{[N]}(v) = 1$$
for every $N$, but such that for each $n$, there does not exist any modification rule $T = (A,T)$ entailing $\P$ with $|A| \leq n$ for which
\begin{equation}\label{joy2}
\int_{V^A} \int_{V^{[k]}} \I\left( \overline{T_v(G_n)}^{([k])}(w) \neq \overline{G_n}^{([k])}(w) \right)\ d\nu^A(v) d\nu^{[k]}(w) < \eps.
\end{equation}

As in the proof of (i), we may assume after passing to a subsequence that the exchangeable $K$-recipes $\overline{G_n} \circ \overline{\nu_n}: \pt \to K$ converge vaguely to an exchangeable $K$-recipe $\mu: \pt \to K$ which almost entails $\P$.

Let $S$ be a countably infinite vertex set.  As before, we invoke
Theorem \ref{struct} to obtain a sub-Cantor palette $Z$ and natural
transformations $\Lambda: K^{\uplus S} \to Z$ and $\kappa: Z \to K$,
with $\Lambda \circ \mu^{\uplus S}: \pt \to Z$ is a regular
exchangeable $Z$-recipe, and with $\overline{\kappa} \circ \Lambda
\circ \mu^{\uplus S}$ almost entailing $\P$ (by \eqref{mustruct}).

As $\P$ is infinitarily strongly locally repairable, we can find a deterministically continuous natural transformation $\tilde T: Z \to G$ entailing $\P$ such that
\begin{equation}\label{mukan-3}
\Lambda^{([k])} \circ \mu^{([k] \uplus S)}( \{ z \in Z^{([k])}: \tilde T^{([k])}(z) \neq \overline{\kappa}^{([k])}(z) \} ) < \eps.
\end{equation}
The situation is once again depicted by \eqref{tstruct}, except with the weakly continuous $T$ replaced by the deterministically continuous $\tilde T$.

Now consider the map
$$ (\tilde T \circ \Lambda)^{([k])}: K^{([k] \uplus S)} \to K^{([k])}.$$
This is a continuous map from the sub-Cantor space $K^{([k] \uplus S)}$ to the finite space $K^{([k])}$.  As such, all of its level sets are clopen, and thus factor through $K^{([k] \uplus A)}$ for some finite subset $A$ of $S$.  In other words, we can find a finite set $A \subset S$ and a continuous map $\overline{T}^{([k])}: K^{([k] \uplus A)} \to K^{([k])}$ such that
$$ (\tilde T \circ \Lambda)^{([k])} = \overline{T}^{([k])} \circ \pi_A^{([k])}$$
where $\pi_A: K^{\uplus S} \to K^{\uplus A}$ is the restriction natural transformation.  If we then define the natural transformation $\overline{T}: K^{\uplus A} \to K$ by requiring that
$$ K^{(\phi)} \circ \overline{T}^{(V)} = \overline{T}^{([k])} \circ K^{(\phi)}$$
for all vertex sets $V$ and all $\phi \in \Hom([k],V)$, one easily verifies that $\overline{T}$ is well-defined, is a deterministically continuous natural transformation, and that the diagram
\begin{equation}\label{tcirc}
\begin{CD}
K^{\uplus A}         @>{\overline{T}}>>    K \\
@AA{\pi_A}A                            @AA{\tilde T}A \\
K^{\uplus S}         @>{\Lambda}>>       Z
\end{CD}
\end{equation}
commutes.  (The reader may wish to connect this diagram together with \eqref{tstruct}, with $T$ again replaced by $\tilde T$ of course.) In particular, $(T, A)$ is a local modification rule in the sense of \ref{locmod}.

Since $\tilde T$ entails $\P$, we see from \eqref{tcirc} and the surjectivity of $\pi_A$ that $\overline{T}$ also entails $\P$.  By chasing all the definitions we conclude that the local modification rule $(T,A)$ also entails $\P$.

Now we turn to \eqref{mukan-3}.  From \eqref{tcirc} and the structure theorem (Theorem \ref{struct}) we can rewrite this as
$$
\mu^{([k] \uplus A)}( \{ H \in K^{([k] \uplus A)}: \overline{T}^{([k])}(H) \neq H\downharpoonright_{[k]} \} ) < \eps.
$$
Since the set here is clopen, and $\overline{G_n} \circ \overline{\nu_n}$ converges vaguely to $\mu$, we conclude that
$$
(\overline{G_n} \circ \overline{\nu_n})^{([k] \uplus A)}( \{ H \in K^{([k] \uplus A)}: \overline{T}^{([k])}(H) \neq H\downharpoonright_{[k]} \} ) < \eps
$$
for all sufficiently large $n$.  But the left-hand side can be rearranged using Definition \ref{contmap} and Definition \ref{locmod} as
$$ \int_{V_n^A} \int_{V_n^k} \I\left( \overline{T_v(G_n)}^{([k])}(w) \neq \overline{G_n}^{([k])}(w) \right)\ d\nu_n^k(w) d\nu_n^A(v).$$
But this contradicts \eqref{joy2} (for $n$ sufficiently large).  This concludes the proof of (ii).
\end{proof}

In view of the above correspondence principle, the Theorems \ref{rs-thm-dir}, \ref{lgr}, \ref{monotone}, \ref{part} now follow immediately from the following four infinitary counterparts respectively.

\begin{proposition}[Every hereditary directed hypergraph property is
testable]\label{rs-prop}  Let $K$ be a finite palette, and let $\P$ be a hereditary $K$-property.  Then $\P$ is infinitarily testable with one-sided error.
\end{proposition}

\begin{proposition}[Every hereditary undirected graph property is locally repairable]\label{lgr-prop} Let $K$ be a finite palette of order at most $2$, and let $\P$ be a hereditary undirected $K$-property.  Then $\P$ is infinitarily strongly locally repairable.
\end{proposition}

\begin{proposition}[Every weakly monotone directed hypergraph property is locally repairable]\label{monotone-prop} Let $K$ be an ordered finite palette, and let $\P$ be a weakly monotone $K$-property.  Then $\P$ is infinitarily strongly locally repairable.
\end{proposition}

\begin{proposition}[Every partite hypergraph property is locally repairable]\label{part-prop} Let $K$ be an finite palette of order $k \geq 1$, and let $\P$ be a partite hereditary $K$-property.  Then $\P$ is infinitarily strongly locally repairable.
\end{proposition}

We will prove these four propositions in future sections, with the proof of Propositions \ref{lgr-prop}, \ref{monotone-prop}, \ref{part-prop} being started in Section \ref{tech}, after some preliminaries in Section \ref{fine}, and Proposition \ref{rs-prop} being started in Section \ref{discred-sec}. For now, we reinterpret the negative results from Section \ref{negchap} by indicating why infinitary strong local repairability fails\footnote{The authors in fact discovered this failure at the infinitary level first, and only converted it to the finitary counterexamples in Section \ref{negchap} afterwards, and with some non-trivial effort.} for directed graph properties or undirected hypergraph properties of order $\leq 3$.

\subsubsection{Directed graph properties are not infinitarily strongly repairable}

We begin by recasting the argument in Section \ref{lrs} in the infinitary setting.  Let $Z_1  = C \subset \R$ be the standard middle-thirds Cantor set consisting of numbers in $[0,1]$ whose base $3$ expansion consists only of $0$s and $2$s with Cantor measure $Q_1 = \mu_C$ (which would be the law of a random base $3$ string in $[0,1]$ consisting of $0$s and $2$s); by Example \ref{rvc}, this induces a natural transformation $P_1: \pt \to Z_{\leq 1}$.  We set $Z := (\pt,Z_1,\{0,1\})$ and $k:=2$, and let $P_2: Z_{<2} \to Z$ be the natural transformation defined by $P_2^{(V)}(z) := \delta_z \times \prod_{e \in \binom{V}{2}} Q^{(e)}(z\downharpoonright_e)$ for all vertex sets $V$ and $z \in Z_{<2}^V$, where $Q^{(\{v,w\})}(z)$ is the law of the directed graph $G_2$ in $\{0,1\}_2^{(\{v,w\})}$ defined by $G_2(v,w) := \I( z(v) < z(w) )$ and $G_2(w,v) := \I( z(w) < z(v) )$.  Then $\mu := P_2 \circ P_1$ is a regular exchangeable $Z$-recipe.  We let $K := \{0,1\}_2$ and let $\kappa: Z \to K$ be the colouring map which is the identity on the second component and trivial on the zeroth and first components.  Then we easily check that $\overline{\kappa} \circ \mu$ almost entails the $\{0,1\}_2$-property $\P$ of being a total ordering (as in Section \ref{lrs}).  However, one cannot find any deterministically continuous natural transformation $T: Z \to K$ which entails $\P$, because any $Z$-coloured hypergraph $z \in Z^{(V)}$ which has a pair $v,w$ of vertices which are indistinguishable in the sense that $z_1(v)=z_1(w)$ and $z_2(v,w)=z_2(w,v)$, will necessarily map under $T$ to a directed graph $G \in K^{(V)}$ such that $G_2(v,w)=G_2(w,v)$, which implies that $G$ cannot obey $\P$.  Thus the $\{0,1\}_2$-property $\P$ is not infinitarily strongly repairable.

One can view the argument in Section \ref{lrs} that shows that $\P$ is not strongly repairable as the finitary analogue of the argument above. (The much more complicated demonstration that $\P$ is also not weakly repairable does not seem to have an easily describable infinitary counterpart.)

\subsubsection{$\leq 3$-uniform hypergraph properties are not infinitarily strongly repairable}

Now let $Z_1 = C \cup \{R\}$, where $C$ is the middle-thirds Cantor set and $R$ is an abstract ``red'' point, and let $Q_1 := \frac{1}{2} \mu_C + \frac{1}{2} \delta_R$, thus the red point has mass $1/2$ and the Cantor set has total mass $1/2$.  We set $Z := (\pt,Z_1,\{0,1\},\{0,1\})$ and $k := 3$, thus by Example \ref{rvc}, $Q_1$ induces a natural transformation $P_1: \pt \to Z_{\leq 1}$.  We then define a natural transformation $P_2: Z_{\leq 1} \to Z_{\leq 2}$ by $P_2^{(V)}(z) := \delta_z \times \prod_{e \in \binom{V}{2}} Q^{(e)}_2(z\downharpoonright_e)$, where $Q^{(\{v,w\})}_2(z)$ is the law of the random graph in $\{0,1\}_2^{(\{v,w\})}$ which is complete with probability $1/2$ and empty otherwise if $z_1(v) \neq z_1(w)$, and always empty when $z_1(v)=z_1(w)$.  We then define a natural transformation $P_3: Z_{\leq 2} \to Z$ by $P_3^{(V)}(z) := \delta_z \times \prod_{e \in \binom{V}{3}} Q^{(e)}_3(z\downharpoonright_e)$, where $Q^{(e)}_3(z)$ is the law of the random hypergraph in $\{0,1\}_3^{(e)}$ which is empty unless $e$ can be expressed as $\{r,b,b'\}$ where $z_1(r)=R$, $z_1(b) > z_1(b')$ lie in $C$, $z_2(r,b)=1$, and $z_2(r,b')=0$, in which case the hypergraph is complete.  Then $\mu := P_3 \circ P_2 \circ P_1$ is a regular exchangeable $Z$-recipe.  If we let $K := (\pt,\{0,1\},\{0,1\},\{0,1\})$, and let $\kappa: Z \to K$ be the colouring which is trivial on the zeroth component, the identity on the second and third components, and maps $C$ to $0$ and $R$ to $1$ on the first component, one verifies that $\overline{\kappa} \circ \mu$ almost entails the $K$-property $\P$ defined in Section \ref{leq3}.

Now let $V := \{r_1,r_2,b_1,b_2\}$ be an abstract set with four elements, and consider the $Z$-coloured hypergraphs $z \in Z^{(V)}$ such that
\begin{itemize}
\item $z_1(r_1)=z_1(r_2)=R$ and $z_1(b_1)=z_1(b_2) \in C$;
\item $z_2(r_1,b_1)=z_2(r_2,b_2)=1$ and $z_2(r_1,b_2)=z_2(r_2,b_1)=0$;
\item $z$ is symmetric with respect to the morphism $\phi \in\Hom(V,V)$ which swaps $r_1$ and $r_2$, and swaps $b_1$ and $b_2$.
\end{itemize}

If $T: Z \to K$ is a deterministically continuous natural transformation and $G := T^{(V)}(z) \in K^{(V)}$, then we see that $G$ is also symmetric with respect to the morphism $\phi$ mentioned above.  If $G_1 = \kappa_1 \circ z_1$ and $G_2 = z_2$, then we also have $G_1(r_1)=G_1(r_2)=1$, $G_1(b_1)=G_1(b_2)=0$, $G_2(r_1,b_1)=G_2(r_2,b_2)=1$ and $G_2(r_1,b_2)=G_2(r_2,b_1)=0$.  But this implies that either $b_1 >_{G,r_1} b_2$ and $b_2 >_{G,r_2} b_1$ are both true, or $b_2 >_{G,r_1} b_1$ and $b_1 >_{G,r_2} b_2$ are both true, which in either case is incompatible with $T$ entailing $\P$.  Thus the only way that $T$ can entail $\P$ is if we have $G_1 \neq \kappa_1 \circ z_1$ or $G_2 \neq z_2$ for all $z \in Z^{(V)}$ of the above form.  But this can be shown to be inconsistent with $T$ obeying \eqref{mukan} for $\eps$ sufficiently small, and so $\P$ is not infinitarily strongly locally repairable.

One can perform a similar infinitary translation of the scenario in Section \ref{eq3}; we leave this to the reader.


\subsection{The asymptotics of increasingly fine colourings}\label{fine}

Much of our analysis will revolve around the colouring of an infinite palette $Z$ by a finite palette $A$; such colouring is roughly analogous to that of dividing the vertices (or lower-order edges) of a graph (or hypergraph) into cells, as is done in the graph and hypergraph regularity lemmas.  We will need a notion of a statement becoming asymptotically true for ``sufficiently fine'' colourings, similar to how a graph becomes increasingly regular as one partitions the vertices into finer and finer cells, or how a measurable function increasingly resembles a continuous one when viewed at finer and finer scales.  In this section we set out some notation that will help us achieve these goals.

\begin{definition}[Colouring topology]\label{color-top}  Let $Z$ be a sub-Cantor palette of order at most $k$.  For each $0 \leq j \leq k$, we let $\Col_j(Z)$ denote the collection of all finite $\sigma$-algebras $\B$ of $Z_j$ that are generated by clopen sets, and let $\Col(Z) := \prod_{j=0}^k \Col_j(Z)$.  Note that every colouring $\alpha = (\alpha_j)_{j=0}^\infty: Z \to A$ generates an element $\B_\alpha = (\B_{\alpha_j})_{j=0}^k$ of $\Col(Z)$, where $\B_{\alpha_j}$ is the $\sigma$-algebra of $Z_j$ is generated by the level sets of $\alpha_j: Z_j \to A_j$.  (The maps $\alpha_j$ for $j>k$ are trivial and thus of no consequence.)

We endow $\Col(Z)$ with the topology whose sub-basic open sets take the form
\begin{equation}\label{sub-basic}
\{ (\B_j)_{j=0}^k \in \Col(Z): \B_i \supset F( \B_{i+1},\ldots, \B_k ) \}
\end{equation}
where $0 \leq i \leq k$ and $F: \Col_{i+1}(Z) \times \ldots \times \Col_k(Z) \to \Col_i(Z)$ is an arbitrary function.  Thus a set is open if it is the union of sets which are finite intersections of sets of the form \eqref{sub-basic}.  We make the simple but important observation that the intersection of finitely many non-empty open sets in $\Col(Z)$ is again a \emph{non-empty} open set.

Let $\alpha: Z \to A$ be a colouring.  A statement involving $\alpha$ is said to hold \emph{for sufficiently fine $\alpha$} if there exists a non-empty open set $U \subset \Col(Z)$ such that the statement holds whenever $\B_\alpha \in U$.  If $c(\alpha) \in \R$ is a real-valued quantity depending\footnote{Technically, the class of all colourings on a given palette $Z$ is not a set, so that $c$ here is a class function rather than a function, but one can rectify this by any number of artificial expedients, for instance by forcing all palettes to take values in the set of integers.} on $\alpha$, and $c_\infty$ is a real number, we say that $c(\alpha)$ \emph{tends to $c_\infty$ as $\alpha \to \infty$}, and write $\lim_{\alpha \to \infty} c(\alpha) = c_\infty$ or $c(\alpha) = c_\infty + o_{\alpha \to \infty}(1)$, if for every $\eps > 0$, the statement $|c(\alpha) - c_\infty| \leq \eps$ is true for sufficiently fine $\alpha$.
\end{definition}

\begin{remark} Readers familiar with the hypergraph regularity lemma may recall that in order to usefully regularise a hypergraph of order $k$ on a vertex set $V$, one must partition each of the edge classes $\binom{V}{j}$ for $1 \leq j \leq k$ into cells.  Typically, the regularisation will only be useful if the partitions for lower values of $j$ are sufficiently fine compared to higher values of $j$, as the lower order partitions are used to regularise the higher order ones.  Our notion of sufficiently fine colourings in the above definition captures the infinitary analogue of this phenomenon.
\end{remark}

One can treat the limit $\alpha \to \infty$ much like a sequential limit $n \to \infty$.  For instance, any finite linear combination of quantities which are $o_{\alpha \to \infty}(1)$ is also $o_{\alpha \to \infty}(1)$.  More generally, we have

\begin{lemma}[Dominated convergence theorem]\label{dct}  Let $Z$ be a sub-Cantor palette, and let $(X,\nu)$ be a probability space.  For each colouring $\alpha: Z \to A$, let $F_\alpha: X \to [-1,1]$ be a measurable function.  If we have
\begin{equation}\label{falpha}
\lim_{\alpha \to \infty} F_\alpha(x) = 0
\end{equation}
for $\nu$-almost every $x \in X$, then we have
$$ \lim_{\alpha \to \infty} \int_X F_\alpha\ d\nu(x) = 0.$$
\end{lemma}

\begin{proof}  By splitting $F_\alpha$ into positive and negative components we may assume that all the $F_\alpha$ are non-negative.  Let $\eps > 0$ be arbitrary.  Since
$$ \int_X F_\alpha\ d\nu(x) \leq \eps + \nu( \{ x \in X: F_\alpha(x) > \eps \} )$$
it will suffice to show that $F_\alpha$ converges to zero in measure, in the sense that
$\nu( \{ x \in X: F_\alpha(x) > \eps \} )\leq \eps$ for all sufficiently fine $\alpha$ (depending on $\eps$).

Since any sub-Cantor space has at most countably many clopen subsets, we see that the set $\Col(Z)$ is at most countable.  We can thus find a sequence $\alpha_n$ of colourings whose associated $\sigma$-algebras $\B_{\alpha_n}$ exhaust $\Col(Z)$.  From this and the hypothesis \eqref{falpha}, we see that
$$\nu( \bigcap_{n=1}^\infty \{ x \in X: F_{\alpha_n}(x) > \eps \} ) = 0.$$
By the monotone convergence theorem, we thus have
$$\nu( \bigcap_{n=1}^N \{ x \in X: F_{\alpha_n}(x) > \eps \} ) < \eps$$
for some finite $N$.  Taking $\alpha$ to be finer than $\alpha_1,\ldots,\alpha_N$, the claim follows.
\end{proof}

An important principle for us will be \emph{Littlewood's principle}, which asserts that measurable functions are almost continuous at sufficiently fine scales.  We shall need the following technical version of this principle.

\begin{lemma}[Littlewood's principle]\label{dom}  Let $Z$ be a sub-Cantor palette of order $k \geq 0$, and let $\alpha: Z \to A$ be a colouring of $Z$.  Let $V$ be a finite vertex set, and let $0 \le j \le k$.  Let $H$ be a finite-dimensional Hilbert space, let $\mu \in \Pr(Z_{=j}^{(V)})$ be a probability measure, and let $F: Z_{=j}^{(V)} \to H$ be a bounded measurable function; we allow $H, \mu, F$ to depend on $\alpha_{j+1},\ldots,\alpha_k,V$, but they must be independent of $\alpha_0,\ldots,\alpha_j$.  For any $a \in A_{=j}^{(V)}$, let $C_a \subset Z_{=j}^{(V)}$ be the set $C_a := (\overline{\alpha_{=j}}^{(V)})^{-1}(\{a\})$.
Then $F$ is almost continuous on most cells $C_a$ in the sense that
$$ \sum_{a \in A_{=j}^{(V)}} \mu(C_a)
\int_{Z_{=j}^{(V)}} \left\| F(z) - \int_{Z_{=j}^{(V)}} F(w)\ d(\mu|C_a)(w) \right\|_H\ d(\mu|C_a)(z) = o_{\alpha \to \infty}(1),$$
where the conditioning $(\mu|C_a)$ is defined in Appendix \ref{prob}, and we adopt the convention that the summand vanishes when $\mu(C_a)=0$.
\end{lemma}

\begin{proof}  Fix $V, j, \alpha_{j+1},\ldots,\alpha_k$, which then fixes $H,\mu,F$, and let $\eps > 0$.  It suffices to show that
$$ \sum_{a \in A_{=j}^{(V)}} \mu(C_a)
\int_{Z_{=j}^{(V)}} \left\| F(z) - \int_{Z_{=j}^{(V)}} F(w)\ d(\mu|C_a)(w) \right\|_H\ d(\mu|C_a)(z) \ll \eps$$
for all sufficiently fine $\alpha_{j}$.

As the topology of $Z_{=j}^{(V)}$ has a countable base of clopen sets, we can approximate the bounded measurable function $F$ to within $O(\eps)$ in $L^1(\mu)$ norm by a finite linear combination $G$ of indicator functions of clopen sets.  Then we have
$$ \sum_{a \in A_{=j}^{(V)}} \mu(C_a)
\int_{Z_{=j}^{(V)}} \| F(z) - G(z) \|_H\ d(\mu|C_a)(z) = \|F-G\|_{L^1(\mu)} \ll \eps$$
and similarly (by the triangle inequality)
$$ \sum_{a \in A_{=j}^{(V)}} \mu(C_a)
\int_{Z_{=j}^{(V)}} \| \int_{Z_{=j}^{(V)}} F(w)\ d(\mu|C_a)(w) - \int_{Z_{=j}^{(V)}} G(w)\ d(\mu|C_a)(w) \|_H\ d(\mu|C_a)(z) \leq \|F-G\|_{L^1(\mu)} \ll \eps$$
so by the triangle inequality again, it suffices to show that
$$ \sum_{a \in A_{=j}^{(V)}} \mu(C_a)
\int_{Z_{=j}^{(V)}} \| G(z) - \int_{Z_{=j}^{(V)}} G(w)\ d(\mu|C_a)(w) \|_H\ d(\mu|C_a)(z) \ll \eps$$
for all sufficiently fine $\alpha_{j}$.  But by the nature of $G$ we see that $G$ will constant on all of the cells $C_a$ if $\alpha_{j}$ is fine enough.  The claim follows.
\end{proof}

\subsection{Reduction of repairability to non-exchangeable repairability}\label{tech}

We need to prove three infinitary strong local repair results\footnote{Readers who are only interested in the testability result may skip ahead to Section \ref{discred-sec}.}, namely Proposition \ref{lgr-prop} which addresses undirected graph properties; Proposition \ref{monotone-prop}, which addresses monotone hypergraph properties; and Proposition \ref{part-prop}, which addresses partite hypergraph properties.  We shall deduce all three propositions from the following somewhat technical proposition that pertains to arbitrary hereditary hypergraph properties, which, instead of constructing a deterministically continuous natural transformation $T: Z \to K$ that entails $\P$, settles for constructing a single map $U: A^{(V)} \to \P^{(V)}$ on a very large but finite vertex set $V$, which satisfies the locality property \eqref{local} but not the exchangeability property \eqref{exchange}.  More precisely, we have

\begin{proposition}[Non-exchangeable repair of hereditary properties]\label{repair}
Let $K$ be a finite palette of order $k \geq 0$, let $\P$ be a
hereditary $K$-property, let $Z$ be a sub-Cantor palette, let
$\kappa: Z \to K$ be a colouring, and let $\mu: \pt \to Z$ be a
regular exchangeable $Z$-recipe such that $\overline{\kappa} \circ
\mu$ almost entails $\P$.  Then for any colouring $\alpha: Z \to A$
which refines $\kappa$ through $\sigma$ (as in Definition
\ref{colour}) and any finite vertex set $V$, there exists a map $U:
A^{(V)} \to \P^{(V)}$ which is \emph{local} in the sense that for
any $W \subset V$ and any $a, a' \in V$ with $a\downharpoonright_W =
a'\downharpoonright_W$, we have $U(a)\downharpoonright_W =
U(a')\downharpoonright_W$, and which locally resembles
$\overline{\sigma}^{(V)}$ in the sense that
\begin{equation}\label{repo}
(\overline{\alpha} \circ \mu)^{([k])}(\Omega_U) \geq 1 - o_{\alpha \to \infty}(1)
\end{equation}
where $\Omega_U \subset A^{([k])}$ is the set of all $b \in A^{([k])}$ such that $K^{(\phi)}(U(a)) = \overline{\sigma}^{([k])}(b)$ for all $\phi \in \Hom([k],V)$ and $a \in A^{(V)}$ with $ A^{(\phi)}(a)=b$, and the expression $o_{\alpha \to \infty}(1)$ is uniform in the choice of $V$.
\end{proposition}

\begin{remark} The fact that the error $o_{\alpha \to \infty}(1)$ is uniform in $V$ is crucial for establishing testability properties for general properties $\P$.  Without this uniformity, one would only be able to test properties that were equivalent to forbidding a finite number of induced hypergraphs.  (We will eventually be generating this finite set $V$ from the Alon-Shapira finitisation trick, Remark \ref{sat}, and as such there is no good control as to the size of $V$ other than that it is finite.)  The need to pass from the finite setting to the infinite setting, but then back again to the finite setting, is somewhat analogous to the presence of several regularisations in the Alon-Shapira argument \cite{AloSha2} at radically different scales; roughly speaking, the finest such regularisation corresponds to the infinitary framework here, but we still have to treat the remaining regularisations finitarily.
\end{remark}

In the remainder of this section we show how Proposition \ref{repair} implies the three infinitary strong local repair results.  We begin with the repairability of weakly monotone hypergraph properties.

\begin{proof}[Proof of Proposition \ref{monotone-prop} assuming Proposition \ref{repair}]
Let $K$ be an ordered finite palette of order $k \geq 0$, let $\P$ be a weakly monotone $K$-property, let $Z$ be a sub-Cantor palette, let $\kappa: Z \to K$ be a colouring, and let $\mu: \pt \to Z$ be a regular exchangeable $Z$-recipe such that $\overline{\kappa} \circ \mu$ almost entails $\P$, and let $\eps > 0$.  Our task is to locate a deterministically continuous natural transformation $T: Z \to K$ entailing $\P$ which obeys \eqref{mukan}.  Note (as observed in Remark \ref{jojo-rem}) that as $T$ is deterministically continuous, the left-hand side of \eqref{mukan} simplifies to
$$
\mu^{([k])}( \{ z \in Z^{([k])}: T^{([k])}(z) \neq \overline{\kappa}^{([k])}(z) \} ).$$

Let $\alpha: Z \to A$ be a sufficiently fine colouring to be chosen later; note that for $\alpha$ fine enough we may assume that $\kappa = \sigma \circ \alpha$ for some colouring $\sigma: A \to K$.  We will find a deterministically continuous natural transformation $S: A \to K$ entailing $\P$ with the property that
\begin{equation}\label{mussel}
(\alpha \circ \mu)^{([k])}( \{ b \in A^{([k])}: S^{([k])}(b) \neq \overline{\sigma}^{([k])}(b) \} ) = o_{\alpha \to \infty}(1).
\end{equation}
Once we do this, Proposition \ref{monotone-prop} follows by setting $T := S \circ \overline{\alpha}$ and taking $\alpha$ sufficiently fine.

It remains to locate a natural transformation $S$ with the required properties.  We first use a finitisation trick of Alon and Shapira.  Observe (from Remark \ref{sat}) that if a deterministically continuous natural transformation $S: A \to K$ does \emph{not} entail $\P$, then there exists a finite integer $N$ such that $S^{([N])}(A^{([N])}) \not \subset \P^{([N])}$.  This integer $N$ ostensibly depends on $S$; however, since $A$ and $K$ are both finite palettes, the number of deterministically continuous natural transformations $S: A \to K$ which do not entail $\P$ is also finite.  Thus (by enlarging $N$ if necessary) one can make $V$ independent of $S$.  In other words, there exists\footnote{Note however that this $N$ is	 \emph{ineffectively} finite, as one needs to solve a ``halting problem'' for $\P$ in order to compute it.} an $N = N_{A,K,\P}$ which serves as a \emph{certificate} for $\P$ in the following sense: if $S: A \to K$ is a deterministically continuous natural transformation such that
\begin{equation}\label{snap}
S^{([N])}(A^{([N])}) \subset \P^{([N])},
\end{equation}
then $S$ entails $\P$.

Fix this value of $N$; by increasing $N$ if necessary we may assume $N \geq k$.  Our objective is now to locate a deterministically continuous natural tranformation $S: A \to K$ that obeys \eqref{mussel} and \eqref{snap}.

Let $V := [N]$.  We apply Proposition \ref{repair} to obtain a local map $U: A^{([N])} \to \P^{([N])}$ obeying \eqref{repo}.

We will now use $U$ and the weakly monotone nature of $\P$, to build the deterministically continuous natural transformation $S: A \to K$.  We first define the map $S^{([N])}: A^{([N])} \to K^{([N])}$ by the formula
\begin{equation}\label{sm}
 S^{([N])}(a) := \bigwedge_{\phi \in \Hom([N],[N])} K^{(\phi)}(U(a))
 \end{equation}
for all $a \in A^{([N])}$,
where the meet of $K$-coloured hypergraphs was defined in Definition \ref{mono} (note that this operation is both commutative and associative).  Since $U(a) \in \P^{(V)}$ and $\P$ is weakly monotone, we see that \eqref{snap} holds.

The map $S^{([N])}$ is clearly $\Hom([N],[N])$-equivariant; since $U$ is local, $S^{([N])}$ is also. From this (and the assumption $N \geq k$) we see that $S^{([N])}$ extends uniquely to a deterministically continuous natural transformation $S: A \to K$.

Finally, it remains to verify \eqref{mussel}.  From \eqref{sm} we see that
$$ S^{([k])}(b) := \bigwedge_{a \in A^{([N])}, \phi \in \Hom([k],[N]): A^{(\phi)}(a)=b} K^{(\phi)}(U(a))$$
for all $b \in A^{([k])}$.  The claim \eqref{mussel} now follows from \eqref{repo}.
\end{proof}

Now we turn to the repairability of undirected graph properties.

\begin{proof}[Proof of Proposition \ref{lgr-prop} assuming Proposition \ref{repair}]  By increasing $k$ and adding some dummy palettes if necessary we can take $k=2$.  We then repeat the proof of Proposition \ref{monotone-prop}, with $K$ a finite palette of order at most $2$ and $\P$ a hereditary undirected $K$-property, and let $Z, \kappa, \mu, \alpha, A, \sigma, N$ be as in the previous proof.  As before, our objective is to locate a deterministically continuous natural transformation $S: A \to K$ obeying \eqref{mussel} and \eqref{snap}.  The main difference is that we will use Ramsey theory instead of monotonicity to construct $S$.

Let $V$ be a sufficiently large finite vertex set (depending on $N, A, K$) to be chosen later.  We apply Proposition \ref{repair} to obtain a local map $U: A^{(V)} \to \P^{(V)}$ obeying \eqref{repo}.

We now use Ramsey-theoretic tools to restrict $U$ to a smaller vertex set on which one has more monochromaticity; in these arguments we will rely crucially on the fact that $k$ is equal to $2$.

Since $U$ is local, we see that $U$ uniquely defines a map $U_W: A^{(W)} \to \P^{(W)} \subset K^{(W)}$ for all $W \subset V$, defined by requiring $U_W(a\downharpoonright_W) = U(a)\downharpoonright_W$ for all $a \in A^{(W)}$.  Applying this with $W=\emptyset$ we obtain a map $U_\emptyset: A_0 \to K_0$.  Applying this instead with $W = \{v\}$ equal to a singleton set, we obtain a map $U_v: A_0 \times A_1 \to K_0 \times K_1$.  The number of possible maps $U_v$ is finite, and so by the pigeonhole principle we can find a subset $V' \subset V$ and a map $U_1: A_0 \times A_1 \to K_0 \times K_1$ such that $U_v = U_1$ for all $v \in V'$.  Furthermore, we can make $V'$ as large as desired (depending on $N,A,K$) by making $V$ sufficiently large (depending on $N,A,K$).

We would like to perform the same analysis for doubleton sets $W = \{v,w\}$, but one runs into a difficulty that there is a $\Hom([2],W)$-ambiguity when trying to identify $A^{(W)}$ (for instance) with $A_0 \times A_1^2 \times A_2^2$.  We shall rectify this by \emph{ad hoc} combinatorial trickery when $k=2$ by exploiting the undirected nature of $\P$, but the ambiguity is much more serious\footnote{Specifically, the problem is that the colour in $K_3$ that $U(a)$ assigns to a $3$-edge $\{u,v,w\}$ depends not only on the vertex colours $a_1(u), a_1(v), a_1(w) \in A_1$ and the $3$-edge colour $a_3(u,v,w) \in A_3$, but also depends on the $2$-edge colours $a_2(u,v), a_2(v,w), a_2(w,u) \in A_2$, in a manner which may not be completely symmetric, even when $\P$ is undirected.  Unsurprisingly, it is this potential for asymmetry within an undirected hypergraph property which is exploited in Sections \ref{leq3} and \ref{eq3}.} when $k \geq 3$ (even for undirected $\P$) when one has to consider tripleton sets $W = \{u,v,w\}$ or worse, and indeed as we see from Theorem \ref{negate}, the analogue of Proposition \ref{lgr-prop} fails in this case.

We turn to the details.  Let $M$ be a large number depending on $N, A, K$ to be chosen later.  If $V$ (and hence $V'$) is chosen sufficiently large depending on $M, A, K$, we can find disjoint sets $V_{a_0,a_1}$ in $V'$ for $a_0 \in A_0$ and $a_1 \in A_1$ such that $|V_{a_0,a_1}| \geq M$.

Suppose $a_0 \in A_0$ and $a_1,a'_1 \in A_1$.  Then we can define a map $U_{v,v'}: A_2 \to K_2$ for any $v \in V_{a_0,a_1}$ and $v' \in V_{a_0,a'_1}$ by setting $U_{v,v'}(a_2) := U_{\{v,v'\}}(a)_2(v,v')$ for all $a_2 \in A_2$, where $a \in A^{(\{v,v'\})}$ is the undirected hypergraph defined explicitly by
$$ a_0() := a_0; a_1(1) := a_1; a_1(2) := a'_1; a_2(1,2) = a_2(2,1) = a_2.$$
Now we crucially use the fact that $\P$ is undirected to conclude that $U_{v,v'} = U_{v',v}$.  Thus $U_{v,v'}$ can be viewed as describing a $K_2^{A_2}$-coloured graph $G_{a_0}$ on the vertex set $\bigcup_{a_1 \in A_1} V_{a_0,a_1}$ for each $a_0 \in A_0$, and in particular defining bipartite graphs between $V_{a_0,a_1}$ and $V_{a_0,a'_1}$ when $a_1 \neq a'_1$.  Applying Ramsey's theorem (as well as the bipartite Ramsey theorem\footnote{See for instance \cite[\S 1.2, 5.1]{GraRotSpe} for statements and proofs of these theorems.  These theorems can be deduced from Theorem \ref{lgr} by a slight modification of the arguments used to prove Corollary \ref{ramsey}, but we of course cannot do so here as that would be circular.}) repeatedly, we thus conclude (if $M$ is sufficiently large depending on $N, A, K$) that we can find subsets $V'_{a_0,a_1} \subset V_{a_0,a_1}$ for $a_0 \in A_0$ and $a_1 \in A_1$ of size
\begin{equation}\label{vaa}
|V'_{a_0,a_1}| = N
\end{equation}
such that $G_{a_0}$ is monochromatic on $V'_{a_0,a_1} \times V'_{a_0,a'_1}$ for all $a_0 \in A_0$ and $a_1,a'_1 \in A_1$ (not necessarily distinct).  In other words, we can find maps $U_{a_0,a_1,a'_1}: A_2 \to K_2$ for $a_0 \in A_0$ and $a_1,a'_1 \in A_1$ with $U_{a_0,a_1,a'_1} = U_{a_0,a'_1,a_1}$ such that $U_{v,v'} = U_{a_0,a_1,a'_1}$ for all $v \in V'_{a_0,a_1}$ and $v' \in V'_{a_0,a'_1}$.

Let us place an arbitrary total ordering $<$ on $K_2$, which in particular defines a minimum function $\min: K_2 \times K_2 \to K_2$.  We now define a deterministically continuous natural transformation $S: A \to K$ by setting
\begin{align*}
S^{(W)}(a)_0() &:= U_0(a_0) \\
S^{(W)}(a)_1(w) &:= U_1(a_0,a_1(w))_1 \\
S^{(W)}(a)_2(w,w') &:= \min( U_{a_0, a_1(w), a_1(w')}(a_2(w,w')), U_{a_0, a_1(w), a_1(w')}(a_2(w',w)) )
\end{align*}
for all vertex sets $W$ and all $a \in A^{(W)}$. One easily verifies that $S$ is indeed a deterministically continuous natural transformation.  Now we verify \eqref{snap}.  If $a \in A^{([N])}$, observe (from \eqref{vaa}) that we can find a morphism $\Phi \in \Hom([N], V)$ such that $\Phi(n) \in V'_{a_0(), a_1(n)}$ for all $n \in [N]$.  Define the symmetrisation $\tilde G \in K^{([N])}$ of any $G \in K^{([N])}$ by defining $\tilde G_0 := G_0$, $\tilde G_1 := G_1$, and $\tilde G_2( n, m) := \min( G_2(n,m), G_2(m,n) )$ for all $(n,m) \in \Hom([2],[N])$; in particular, $\tilde G = G$ whenever $G$ is undirected.  By chasing all the definitions we see that
$$ S^{([N])}(a) = \widetilde{K^{(\Phi)}( U( b ) )}$$
for any $b \in A^{(V)}$ with $A^{(\Phi)}(b) = a$.  Since $U(b)$ obeys $\P$ and is thus undirected, we obtain \eqref{snap} as desired.

Finally, we verify \eqref{mussel}.  Let $b \in A^{([k])}$ be drawn
at random with law $(\overline{\alpha} \circ \mu)^{([k])}$, and let
$G := \overline{\sigma}^{([k])}(b) \in K^{([k])}$.  By \eqref{repo},
we see that with probability $1 - o_{\alpha \to \infty}(1)$ we have
\begin{equation}\label{kphu}
K^{(\phi)}(U(a)) = G
\end{equation}
whenever $\phi \in \Hom([k],V)$ and $a \in A^{(V)}$ satisfies
$A^{(\phi)}(a) = b$; let us now condition on this event.  Since
$U(a) \in \P^{(V)}$, we conclude that $G \in \P^{([k])}$; in
particular, $G$ is undirected.

To prove \eqref{mussel}, it will suffice to show that $S^{([k])}(b) = G$.  In view of the definition of $S$, it will suffice to show that
\begin{align*}
G_0() &= U_0(b_0) \\
G_1(i) &= U_1(b_0,b_1(i))_1 \\
G_2(i,j) &= U_{b_0,b_1(i), b_i(j)}(b_2(i,j)) \\
G_2(i,j) &= U_{b_0,b_1(i), b_i(j)}(b_2(j,i))
\end{align*}
for all distinct $i, j \in [k]$.  But these claims all follow from \eqref{kphu} and the definition of $U_0$, $U_1$, $U_{b_0,b_1(i),b_i(j)}$ by choosing $\phi$ appropriately.
\end{proof}

Finally, we establish the repairability of partite hypergraph properties.

\begin{proof}[Proof of Proposition \ref{part-prop} assuming Proposition \ref{repair}]  Once again we repeat the proof of Proposition \ref{monotone-prop}, with $K$ a finite palette of order $k \geq 1$ and $\P$ a partite hypergraph $K$-property, and let $Z, \kappa, \mu, \alpha, A, \sigma, N$ be as in the previous proof.  As before, our objective is to locate a deterministically continuous natural transformation $S: A \to K$ obeying \eqref{mussel} and \eqref{snap}.  In this case we will use partite Ramsey theory instead of Ramsey theory or monotonicity to construct $S$.

Let $M$ be a large integer (depending on $N,A,K$) to be chosen later.  We let $V := [M] \times A_1$, thus $W$ is the disjoint union of the sets $V_{a_1} := [M] \times \{a_1\}$ of cardinality $M$ for $a_1 \in A_1$.  We apply Proposition \ref{repair} to obtain a local map $U: A^{(V)} \to \P^{(V)}$ obeying \eqref{repo}.  From locality as before, we also have maps $U_W: A^{(W)} \to \P^{(W)}$ for all $W \subset V$.

Let $0 \leq j \leq k$, and let $\psi \in \Hom([j],A_1)$ be a morphism.  For any vertices $v_1 \in V_{\psi(1)}, \ldots, v_j \in V_{\psi(j)}$, one can define a map $U_{\psi; v_1,\ldots,v_j}: A^{([j])} \to \P^{([j])}$ by the formula
$$ U_{\psi; v_1,\ldots,v_j} := K^{(v)} \circ U_{\{v_1,\ldots,v_j\}} \circ A^{(v^{-1})}$$
where $v: [j] \to \{v_1,\ldots,v_j\}$ is the bijection that sends $i$ to $v_i$ for $i \in [j]$.  One can view this map as defining a $j$-partite $j$-uniform $(\P^{([j])})^{A^{([j])}}$-coloured hypergraph on the disjoint vertex classes $V_{\psi(1)},\ldots,V_{\psi(j)}$.

The number of $j$ and $\psi$ are finite (and independent of $M$), and the size of the palettes
$(\P^{([j])})^{A^{([j])}}$ are also finite and independent of $M$.  Thus by applying the partite hypergraph Ramsey theorem (see e.g. \cite[\S 5.1]{GraRotSpe}) repeatedly, we conclude (if $M$ is sufficiently large depending on $N,A,K$) that there exist sets $V'_{a_1} \subset V_{a_1}$ of cardinality $|V'_{a_1}| = N$ for all $a_1 \in A_1$ such that all the partite hypergraphs mentioned above are monochromatic, or in other words that for every $0 \leq j \leq k$ and $\psi \in \Hom([j],A_1)$ there exists a map $U_\psi: A^{([j])} \to \P^{([j])}$ such that $U_{\psi; v_1,\ldots,v_j} = U_\psi$ for all $v_1 \in V'_{\psi(1)},\ldots,v_j \in V'_{\psi(j)}$.

Fix the $V'_{a_1}$ and $U_\phi$.  We now introduce the deterministically continuous natural transformation $S: A \to K$ by defining $S^{(W)}(a)_j(\phi) \in K_j$ for vertex sets $W$, hypergraphs $a \in A^{(W)}$, integers $0 \leq j \leq k$, and morphisms $\phi \in \Hom([j],W)$ according to the following rule.  If $\phi$ is a partite edge for $a$ (thus the map $a_1 \circ \phi: [j] \to A_1$ is a morphism) then we set
\begin{equation}\label{swaj}
 S^{(W)}(a)_j(\phi) := U_{a_1 \circ \phi}( A^{(\phi)}(a) )_j(\phi);
 \end{equation}
otherwise, if $\phi$ is not a partite edge, we set
\begin{equation}\label{swaj-2}
S^{(W)}(a)_j(\phi) := \sigma_j( a_j(\phi) ).
\end{equation}
One easily verifies that $S$ is a strongly natural transformation.  Now we verify \eqref{snap}.  Let $a \in A^{([N])}$ be arbitrary.  Since each of the $V'_{a_1}$ have cardinality $N$, we can find a morphism $\Phi: [N] \to V$ such that $\Phi(n) \in V'_{a_1(n)}$ for all $n \in [N]$.  Let $b \in A^{(V)}$ be any hypergraph such that $A^{(\Phi)}(b) = a$.  By chasing all the definitions (and using the local nature of $U$), we conclude that
$$ S^{([N])}(a)_j(\phi) = U(b)_j( \Phi \circ \phi )$$
for all $0 \leq j \leq k$ and all partite edges $\phi \in \Hom([j],[N])$.  By Definition \ref{partite}, we conclude that $S^{([N])}(a)$ is partite equivalent to $K^{(\Phi)}(U(b))$.  Since $U(b) \in \P^{(V)}$ and $\P$ is partite, we obtain $S^{([N])}(a) \in \P^{([N])}$ as required.

Now we prove \eqref{mussel}.  Let $b \in A^{([k])}$ be drawn at
random with law $(\alpha \circ \mu)^{([k])}$, and let $G :=
\overline{\sigma}^{([k])}(b) \in K^{([k])}$.  By \eqref{repo}, we
see with probability $1-o_{\alpha \to \infty}(1)$ that \eqref{kphu}
holds whenever $\phi \in \Hom([k],V)$ and $a \in A^{(V)}$ satisfies
$A^{(\phi)}(a) = b$. Conditioning on this event, we conclude from
\eqref{swaj} and the definition of the $U_\psi$ that
$S^{([k])}(b)_j(\phi) = \sigma_j(b_j(\phi))$ whenever $0 \leq j \leq
k$ and $\phi \in \Hom([j],[k])$ is a partite edge of $b$.  Combining
this with \eqref{swaj-2} we obtain \eqref{mussel}.
\end{proof}

To conclude the proof of all our main theorems, it remains to establish Proposition \ref{rs-prop} and Proposition \ref{repair}.  This will be the purpose of the remaining sections.

\subsection{Reduction to discretisations of the identity}\label{discred-sec}

In the previous sections, we have reduced all of our testability and repair claims to two propositions, namely Proposition \ref{rs-prop} and \ref{repair}.  In this section, we show how these propositions will follow from the following two propositions, which assert the existence of two different ways to approximate the identity natural transformation $\id_Z: Z \to Z$ by more discrete natural transformations that factor through a colouring $\alpha: Z \to A$.

\begin{proposition}[First discretisation of the identity]\label{disc-ident}
Let $Z$ be a sub-Cantor palette of some order $k \geq 0$, and let $\mu: \pt \to Z$ be a regular exchangeable $Z$-recipe.  Then for any colouring $\alpha: Z \to A$ there exist $j$-independent natural transformations $Q_{\alpha,j}: Z_{<j} \times A_{\geq j} \to Z_{\leq j} \times A_{>j}$ for each $0 \leq j \leq k$ with the following properties:
\begin{itemize}
\item[(i)] ($Q_{\alpha,j}$ only modifies the $j$ component) For each $\alpha$ and each $0 \leq j \leq k$, the diagram
$$
\begin{CD}
Z_{<j} \times A_{\geq j}     @>{Q_{\alpha,j}}>>    Z_{\leq j} \times A_{>j} \\
@VVV                                               @VVV            \\
Z_{<j} \times A_{>j}         @=                    Z_{<j} \times A_{>j}
\end{CD}
$$
commutes, where the vertical arrows denote the obvious projection natural transformations.
\item[(ii)] (Absolute continuity)  For each $\alpha$, every finite vertex set $V$, and every $a \in A^{(V)}$, we have
$$ (Q_{\alpha,k} \circ \ldots \circ Q_{\alpha,0})^{(V)}(a) \ll \mu.$$
\item[(iii)] (Convergence to the diagonal)  Given any finite vertex set $V$ and any continuous function $F: Z^{(V)} \times Z^{(V)} \to \R$, we have
$$ \lim_{\alpha \to \infty} \int_{Z^{(V)}} \left( \int_{Z^{(V)}} F( z, z')\ (Q_{\alpha,k} \circ \ldots \circ Q_{\alpha,0} \circ \overline{\alpha})^{(V)}(z,dz') \right)\ d\mu^{(V)}(z) = \int_{Z^{(V)}} F(z,z)\ d\mu^{(V)}(z).$$
\end{itemize}
\end{proposition}

\begin{example}\label{d1}  Let $Z = (\pt, Z_1)$ and $k=1$ for some sub-Cantor space $Z_1$, and let $\mu \in \Pr(Z_1)$ be a probability measure, which can be identified with an exchangeable $Z$-recipe by Example \ref{rvc}.  Let $\alpha: Z \to A$ be a colouring of $Z$, let $Q_0: A \to A$ be the identity map, and let $Q_1: A \to Z$ be the natural transformation defined by $Q_1^{(V)}(a) := \prod_{v \in V} \mu_{a_1(v)}$ for any vertex set $V$ and $a \in A^{(V)}$, where we identify $Z^{(V)}$ with $Z_1^V$ and for any $a_1 \in A_1$, $\mu_{a_1} \in \Pr(Z_1)$ is the measure $(\mu|C_{a_1})$ if $\mu(C_{a_1})>0$ and $\mu$ otherwise, where $C_{a_1} := \alpha_1^{-1}(\{a_1\})$.  Then one easily verifies that $Q_0, Q_1$ obey the properties described above.  Roughly speaking, the map $Q_1$ maps points $z$ in $Z_1$ to the uniform distribution on the $A_1$-cell that $z$ lies in; as the colouring $\alpha$ gets finer and finer, this map converges to the identity in a weak sense, which corresponds to the property (iii) above. 
\end{example}

\begin{proposition}[Second discretisation of the identity]\label{disc-ident2}
Let $Z$ be a sub-Cantor palette of some order $k \geq 0$, and let $\mu: \pt \to Z$ be a regular exchangeable $Z$-recipe.  Then for any colouring $\alpha: Z \to A$  there exist a sub-Cantor space $X_\alpha$ (which we view as a sub-Cantor palette of order $0$) with a probability measure $\nu_\alpha \in \Pr(X_\alpha)$ (which we view as a natural transformation $\nu_\alpha: \pt \to X_\alpha$), together with a deterministically continuous natural transformation $\zeta_\alpha: A \times X_\alpha \to Z$, with the following properties:
\begin{itemize}
\item[(i)] (Asymptotic absolute continuity)  The measure
$$ \zeta_\alpha^{([k])} \circ ((\overline{\alpha} \circ \mu) \oplus \nu_\alpha)^{([k])} \in \Pr(Z^{([k])})$$
is $o_{\alpha \to \infty}(1)$-absolutely continuous
with respect to $\mu^{([k])}$ (see Definition \ref{epsac-def} for a definition of $\eps$-absolute continuity).
\item[(ii)] (Convergence to the diagonal)
Given any finite vertex set $V$ and any continuous function $F: Z^{(V)} \times Z^{(V)} \to \R$, we have
$$ \lim_{\alpha \to \infty} \int_{Z^{(V)}} \int_{X_\alpha}
F( z, \zeta_\alpha^{(V)}( \overline{\alpha}^{(V)}(z), x) )\ d\nu_\alpha(x) d\mu^{(V)}(z) = \int_{Z^{(V)}} F(z,z)\ d\mu^{(V)}(z).$$
\end{itemize}
\end{proposition}

\begin{remark}
The situation in the above proposition can be depicted by the following diagram,
$$
\begin{CD}
Z       @<<<  Z \times X_\alpha          @>{\overline{\alpha}\oplus \id}>> A \times X_\alpha  @>{\zeta_\alpha}>> Z\\
@A{\mu}AA   @A{\mu \oplus \nu_\alpha}AA                                    @A{(\overline{\alpha} \circ \mu) \oplus \nu_\alpha}AA    @.\\
\pt      @=   \pt @= \pt @.
\end{CD};
$$
informally speaking, the proposition asserts that the right map from $Z \times X_\alpha$ to $Z$ is asymptotically absolutely continuous and asymptotically convergent to the left map.
\end{remark}

\begin{example}\label{d2}  Let $Z = (\pt, Z_1)$ and $k=1$ for some sub-Cantor space $Z_1$, and let $\mu \in \Pr(Z_1)$ be a probability measure, which can be identified with an exchangeable $Z$-recipe by Example \ref{rvc}.  For technical reasons we also select an arbitrary element $z_*$ of $Z_1$.  Let $\alpha: Z \to A$ be a colouring of $Z$. For each $a_1 \in A_1$, we define the cell $C_{a_1} := \alpha_1^{-1}(\{a_1\})$, and draw $\zeta_{a_1} \in Z_1$ independently at random for each $a_1$ with law $(\mu|C_{a_1})$ if $\mu(C_{a_1}) > 0$, or with law $\delta_{z_*}$ otherwise.  We then define the natural transformation $\zeta_\alpha: A \to Z$ by setting $\zeta_\alpha^{(V)}(a)_1(v) := \zeta_{a_1(v)}$ for all vertex sets $V$, all $a \in A^{(V)}$, and all $v \in V$; one easily verifies that this transformation obeys all the required properties; compare this construction with that in Example \ref{d1}.  As we shall see in Section \ref{disc2-sec}, the situation becomes more complicated when $k>1$ due to the presence of ``indistinguishable'' pairs of elements of $A^{([j])}$ for $1 < j \leq k$ which are coupled together, which forces some modification to the above procedure of selecting each value of $\zeta$ independently.
\end{example}

In the rest of this section we show how Proposition \ref{rs-prop} follows from Proposition \ref{disc-ident}, and Proposition \ref{repair} follows by combining Proposition \ref{disc-ident} with Proposition \ref{disc-ident2}.

\begin{proof}[Proof of Proposition \ref{rs-prop} assuming Proposition \ref{disc-ident}]
Let $K$ be a finite palette of some order $k \geq 0$, let $\P$ be a hereditary $K$-property, let $Z$ be a sub-Cantor palette, let $\kappa: Z \to K$ be a colouring, and let $\mu: \pt \to Z$ be a regular exchangeable $Z$-recipe such that $\kappa \circ \mu$ almost entails $\P$, and let $\eps > 0$.  Our task is to construct a weakly continuous natural transformation $T: Z \to K$ which almost entails $\P$, and such that \eqref{mukan} holds.

Let $\alpha: Z \to A$ be a sufficiently fine colouring of $Z$ to be chosen later.  We apply Proposition \ref{disc-ident} to obtain natural transformations $Q_{\alpha,j}: Z_{<j} \times A_{\geq j} \to Z_{\leq j} \times A_{>j}$ for $0 \leq j \leq k$ with the stated properties.  We then define $T = T_\alpha: Z \to K$ to be the natural transformation
$$ T := \overline{\kappa} \circ Q_{\alpha,k} \circ \ldots \circ Q_{\alpha,0} \circ \overline{\alpha}.$$
Since $T$ factors through the natural transformation $\overline{\alpha}: Z \to A$, and $A$ is a finite palette, we see that $T$ must be weakly continuous.  Now we verify that $T$ almost entails $\P$.  If $V$ is a finite vertex set and $z \in Z^{(V)}$, then by Proposition \ref{disc-ident}(ii) we see that the probability measure $T^{(V)}(z)$ is absolutely continuous with respect to $(\overline{\kappa} \circ \mu)^{(V)}$.  Since $\overline{\kappa} \circ \mu$ almost entails $\P$, we see that $\P^{(V)}$ has full measure with respect to $(\overline{\kappa} \circ \mu)^{(V)}$ and hence $T^{(V)}(z)$.  The claim now follows from Remark \ref{sat}.

Finally, we need to verify \eqref{mukan}.  Let $F: Z^{([k])} \times Z^{([k])} \to \R$ be the indicator function
\begin{equation}\label{fdef}
F(z,z') := \I(\overline{\kappa}^{([k])}(z) \neq \overline{\kappa}^{([k])}(z')).
\end{equation}
Observe that $F$ is a continuous function which vanishes on the diagonal $z=z'$, and so by Proposition \ref{disc-ident}(iii) we have
$$ \int_{Z^{([k])}} F( z, (Q_{\alpha,k} \circ \ldots \circ Q_{\alpha,0} \circ \overline{\alpha})^{([k])}(z) )\ d\mu^{(V)}([k]) < \varepsilon$$
for sufficiently fine $\alpha$.  But by chasing all the definitions we see that this is equivalent to \eqref{mukan}.
\end{proof}

\begin{remark} Note that we did not use the full strength of Proposition \ref{disc-ident} in order to establish Proposition \ref{rs-prop}.  However we will need to exploit Proposition \ref{disc-ident} more thoroughly when establishing Proposition \ref{repair} below.
\end{remark}

\begin{proof}[Proof of Proposition \ref{repair} assuming Proposition \ref{disc-ident} and Proposition \ref{disc-ident2}]  Let $K$ be a finite palette of order $k \geq 0$, let $\P$ be a hereditary $K$-property, let $Z$ be a sub-Cantor palette, let $\kappa: Z \to K$ be a colouring, and let $\mu: \pt \to Z$ be a regular exchangeable $Z$-recipe such that $\overline{\kappa} \circ \mu$ almost entails $\P$.  Let $\eps > 0$.   Our task is to show that if colouring $\alpha: Z \to A$ is a sufficiently fine colouring which refines $\kappa$ in the sense that $\kappa = \sigma \circ \alpha$ for some $\sigma: A \to K$, then for any finite vertex set $V$ there exists a local map $U: A^{(V)} \to \P^{(V)}$ such that
\begin{equation}\label{repo2}
(\overline{\alpha} \circ \mu)^{([k])}(\Omega_U) \geq 1 - \eps.
\end{equation}

Let $\alpha$ be as above.  As in the proof of Proposition \ref{rs-prop}, we let $F: Z^{([k])} \times Z^{([k])} \to \R$ be the indicator function \eqref{fdef}.   If $\alpha$ is sufficiently fine, then by Proposition \ref{disc-ident2} we can find a sub-Cantor space $X_\alpha$ with a probability measure $\nu: \pt \to X_\alpha$ and a deterministically continuous natural transformation $\zeta_\alpha: A \times X_\alpha \to Z$ with
\begin{equation}\label{zzf}
\int_{Z^{(V)}} F( z, (\zeta_\alpha \circ (\overline{\alpha} \oplus \nu_\alpha))^{(V)}(z) )\ d\mu^{(V)}(z) < \eps/3
\end{equation}
such that the $\zeta_\alpha^{([k])} \circ ((\overline{\alpha} \circ \mu) \oplus \nu_\alpha)^{([k])}$ is $\eps/3$-absolutely continuous with respect to $\mu^{([k])}$.  By Proposition \ref{epsac}, we can find a compact set $E_{\alpha} \subset A^{([k])} \times X_\alpha$ such that
\begin{equation}\label{excep}
 ((\overline{\alpha} \circ \mu) \oplus \nu_\alpha)^{([k])}(E_{\alpha}) < \eps/3
\end{equation}
and
\begin{equation}\label{ac-out}
 \zeta_\alpha^{([k])} \circ \I(E_{\alpha}^c) ((\overline{\alpha} \circ \mu) \oplus \nu_\alpha)^{([k])} \ll \mu^{([k])}.
\end{equation}

Now let $V$ be an arbitrary finite vertex set.  We let $\alpha': Z \to A'$ be another colouring (it will depend on\footnote{This introduction of a second colouring has an analogue in \cite{AloSha}, \cite{AloSha2}, in which one uses a fine Szemer\'edi partition to decide how to colour a coarse Szemer\'edi partition.} $V$ and $\alpha$) to be chosen later.  We apply Proposition \ref{disc-ident} to obtain $j$-independent natural transformations $Q_{\alpha',j}: Z_{<j} \times {A'}_{\geq j} \to Z_{\leq j} \times {A'}_{>j}$ for $0 \leq j \leq k$ with the stated properties.

For each $-1 \leq j \leq k$ in turn, we use the $Q_{\alpha',j}$ to construct random local maps $U'_{\leq j}: {A'}_{\leq j}^{(V)} \to Z_{\leq j}^{(V)}$ recursively as follows.  The map $U'_{\leq -1}: \pt \to \pt$ is of course the trivial map.  Now suppose recursively that $0 \leq j \leq k$ and the local map $U'_{<j} := U'_{\leq j-1}: {A'}_{<j}^{(V)} \to Z_{<j}^{(V)}$ has already been chosen.  For any $e \in \binom{V}{j}$, the local map $U'_{<j}$ then induces a map $U'_{<j,e}: {A'}_{<j}^{(e)} \to Z_{<j}^{(e)}$.  We then randomly select, independently for each $e$, a map $U'_{\leq j,e}: {A'}_{\leq j}^{(e)} \to Z_{\leq j}^{(e)}$ by choosing $U'_{\leq j,e}(a)$ independently at random for each $a \in {A'}_{\leq j}^{(e)}$ with law $Q_{\alpha',j}^{(e)}( U'_{<j,e}(a_{<j}), a_j )$, where $a_{<j} \in {A'}_{<j}^{(e)}$ and $a_j \in {A'}_{=j}^{(e)} = {A'}_{\geq j}^{(e)}$ are the components of $a$.  From Proposition \ref{disc-ident}(i), we see that we almost surely have the commutative diagram

\begin{equation}\label{upcom}
\begin{CD}
{A'}_{\leq j}^{(e)} @>{U'_{\leq j,e}}>>   Z_{\leq j}^{(e)} \\
@VVV                                      @VVV            \\
{A'}_{<j}^{(e)}     @>{U'_{<j,e}}>>       Z_{<j}^{(e)}
\end{CD},
 \end{equation}
where the vertical arrows are the obvious projection maps.  We now condition on this probability $1$ event.

We then define the local map $U'_{\leq j}: {A'}_{\leq j}^{(V)} \to Z_{\leq j}^{(V)}$ to be the unique local map whose restrictions to each $e \in \binom{V}{j}$ are given by $U'_{\leq j,e}$; the condition
\eqref{upcom} (and the local nature of $U'_{<j}$) ensures that the local map $U'_{\leq j}$ is well-defined.

By the $j$-independent nature of the $Q_{\alpha',j}$ (see Definition \ref{j-indep} and Proposition \ref{disc-ident}(i), we see by induction on $j$ that for any $-1 \leq j \leq k$ and any $a \in A_{\leq j}^{(V)}$, the random variable $U'_{\leq j}(a) \in Z_{\leq j}^{(V)}$ is distributed with law $Q_{\leq j}(a)$, where $Q_{\leq j}: A_{\leq j} \to Z_{\leq j}$ is the unique natural transformation obeying the commutative diagram
\begin{equation}\label{upcom2}
\begin{CD}
A              @>{Q_{\alpha',j} \circ \ldots \circ Q_{\alpha',0}}>>   Z_{\leq j} \times A_{>j} \\
@VVV                                                                  @VVV            \\
A_{\leq j}     @>{\rlap{$\scriptstyle{\ \ \ \ \ \ Q_{\leq j}}$}\phantom{Q_{\alpha',j} \circ \ldots \circ Q_{\alpha',0}}}>>                                       Z_{\leq j}
\end{CD},
\end{equation}
where the vertical arrows are the obvious projection natural transformations.  Applying this with $j=k$, we conclude that for any $a \in A^{(V)}$, the random variable $U'_{\leq k}(a) \in Z^{(V)}$ is distributed with law $(Q_{\alpha',k} \circ \ldots \circ Q_{\alpha',0})^{(V)}(a)$.  In particular, we see from Proposition \ref{disc-ident}(ii) that the distribution of $U'_{\leq k}(a)$ is absolutely continuous with respect to $\mu^{(V)}$.  Since $\overline{\kappa} \circ \mu$ almost entails $\P$, we conclude that $\overline{\kappa}^{([k])} \circ U'_{\leq k}(a)$ obeys $\P$ almost surely.  In other words, we see with probability $1$ that the map $\overline{\kappa}^{([k])} \circ U'_{\leq k}$ maps $(A')^{(V)}$ to $\P^{(V)}$.
We now condition on this probability $1$ event.

We choose $x \in X_\alpha$ at random with law $\nu_\alpha$ (independently of all previous random choices), and define the (probabilistic) map $U = U_x: A^{(V)} \to \P^{(V)}$ by composing together the chain
\begin{equation}\label{uxdef0}
\begin{CD}
A^{(V)} @>{\id \times x}>> A^{(V)} \times X_\alpha @>{\zeta_\alpha^{(V)}}>> Z^{(V)} @>{\overline{\alpha'}^{(V)}}>> (A')^{(V)} @>{U'_{\leq k}}>> Z^{(V)} @>{\overline{\kappa}^{(V)}}>> K^{(V)}
\end{CD}
\end{equation}
or in other words by the formula
\begin{equation}\label{uxdef}
 U_x(a) := (\overline{\kappa}^{(V)} \circ U'_{\leq k} \circ \overline{\alpha'}^{(V)} \circ \zeta_\alpha^{(V)})(a,x)
 \end{equation}
for all $a \in A^{(V)}$.

By construction we see that (with probability $1$) $U_x$ does indeed map $A^{(V)}$ to $\P^{(V)}$; since $U'_{\leq k}$ is local and $\overline{\kappa}$, $\overline{\alpha'}$, $\zeta_\alpha$ are deterministically continuous natural transformations, we see that $U_x$ is also almost surely local.  To establish the claim \eqref{repo2}, it thus suffices by the probabilistic method to show that
$$ \E (\overline{\alpha} \circ \mu)^{([k])}(\Omega_{U_x}) \geq 1 - \eps.$$

Accordingly, let us select $b \in A^{([k])}$ at random with law $(\overline{\alpha} \circ \mu)^{([k])}$.  By \eqref{excep}, we see that $(b,x) \not \in E_{\alpha}$ with probability at least $1-\eps/3$.  Also, by \eqref{zzf}, we see that $\overline{\sigma}^{([k])}(b) = (\overline{\kappa} \circ \zeta_\alpha)^{([k])}(b,x)$ with probability $1-\eps/3$.  Thus it suffices to show that the event
$$ (b,x) \not \in E_{\alpha} \hbox{ and } K^{(\phi)}(U_x(a)) \neq (\overline{\kappa} \circ \zeta_\alpha)^{([k])}(b,x) \hbox{ for some } \phi \in \Hom([k],V) \hbox{ and } a \in A^{(V)} \hbox{ with } A^{(\phi)}(a)=b $$
has probability at most $\eps/3$.

Fix $\phi \in \Hom([k],V)$; by the union bound, it suffices to show that the event
$$ (b,x) \not \in E_{\alpha} \hbox{ and } K^{(\phi)}(U_x(a)) \neq (\overline{\kappa} \circ \zeta_\alpha)^{([k])}(b,x) \hbox{ for some } a \in A^{(V)} \hbox{ with } A^{(\phi)}(a)=b $$
has probability at most $\eps/3|\Hom([k],V)|$.

Write $z:= \zeta_\alpha^{([k])}(b,x)$, $a' := \overline{\alpha'}^{([k])}(z)$, and $e := \phi([k])$.  From \eqref{uxdef} (or \eqref{uxdef0}) we see that
$$ K^{(\phi)}(U(a)) = \overline{\kappa}^{([k])} \circ Z^{(\phi)} \circ U'_{\leq k, e} \circ (A')^{(\phi^{-1})}(a')$$
whenever $a \in A^{(V)}$ is such that $A^{(\phi)}(a) = b$, where $U'_{\leq k,e}: (A')^{(e)} \to Z^{(e)}$ is the localisation of the local map $U'_{\leq k}: (A')^{(V)} \to Z^{(V)}$.  Thus, if we write
$z' := U'_{\leq k, e} \circ (A')^{(\phi^{-1})}(a')$, it suffices to show that the event
$$ (b,x) \not \in E_{\alpha} \hbox{ and } \overline{\kappa}^{([k])}(z') \neq \overline{\kappa}^{([k])}(z) $$
has probability at most $\eps/3|\Hom([k],V)|$.  By \eqref{fdef}, it thus suffices to show that
\begin{equation}\label{evff}
\E \left( \I((b,x) \not \in E_{\alpha}) F(z,z') \right) < \frac{\eps}{3|\Hom([k],V)|}.
\end{equation}

Recall that for any $a \in A^{(V)}$, the random variable $U'_{\leq k}(a) \in Z^{(V)}$ is distributed with law $(Q_{\alpha',k} \circ \ldots \circ Q_{\alpha',0})^{(V)}(a)$.  This implies for fixed $a'$ that $z'$ is distributed with law $(Q_{\alpha',k} \circ \ldots \circ Q_{\alpha',0} \circ \overline{\alpha'})^{([k])}(z)$.  Thus we can write the left-hand side of \eqref{evff} as
$$
\int_{Z^{(V)}} f_{\alpha'}(z)\ d\mu_\alpha(z)$$
where $f_{\alpha'}: Z^{(V)} \to [0,1]$ is the measurable function
$$ f_{\alpha'}(z) := \int_{Z^{(V)}} F(z,z')\ (Q_{\alpha',k} \circ \ldots \circ Q_{\alpha',0} \circ \overline{\alpha'})^{([k])}(z, dz')$$
and $\mu_\alpha$ is the finite measure
$$ \mu_\alpha := \zeta_\alpha^{([k])} \circ \I(E_{\alpha}^c) ((\overline{\alpha} \circ \mu) \oplus \nu_\alpha)^{([k])}.$$
Now, by Proposition \ref{disc-ident}(iii), we have
$$ \lim_{\alpha' \to \infty} \int_{Z^{(V)}} f_{\alpha'}(z)\ d\mu^{(V)}(z) = 0;$$
thus (by Markov's inequality) $f_{\alpha'}$ converges in measure to zero with respect to $\mu^{(V)}$ as $\alpha' \to \infty$.  On the other hand, from \eqref{ac-out} we see that $\mu_\alpha$ is absolutely continuous with respect to $\mu^{(V)}$.  By Proposition \ref{epsac}, we conclude that $f_{\alpha'}$ also converges in measure to zero with respect to $\mu_\alpha$, and so
$$ \lim_{\alpha' \to \infty} \int_{Z^{(V)}} f_{\alpha'}(z)\ d\mu_\alpha(z) = 0.$$
Thus, by choosing $\alpha'$ sufficiently fine depending on $\alpha, V, \eps$, we obtain \eqref{evff} as required for every choice of $\phi$.
\end{proof}

To complete the proof of our testability and local repair results, it suffices to prove Proposition \ref{disc-ident} and Proposition \ref{disc-ident2}.  This is the purpose of the next two sections.

\subsection{Proof of Proposition \ref{disc-ident}}\label{disc-sec}

We now prove Proposition \ref{disc-ident}.  Let $Z, k, \mu, \alpha$ be as in that proposition.  By
Definition \ref{regexp}, we can factor
$$ \mu = P_k \circ \ldots \circ P_0$$
where for each $0 \leq j \leq k$, $P_j: Z_{<j} \to Z_{\leq j}$ is a $j$-independent natural transformation such that $\pi_{<j} \circ P_j = \id_{Z_{<j}}$, where $\pi_{<j}: Z_{\leq j} \to Z_{<j}$ is the projection natural transformation.  From Definition \ref{j-indep}, we conclude that
\begin{equation}\label{pjw}
 P_j^{(V)}(z) = \delta_z \times Q_j^{(V)}(z) = \delta_z \times \prod_{e \in \binom{V}{j}} Q_j^{(e)}(z\downharpoonright_e)
\end{equation}
for all vertex sets $V$ and all $z \in Z_{<j}^{(V)}$, and some $j$-independent natural transformation $Q_j: Z_{<j} \to Z_{=j}$, where we identify $Z_{\leq j}^{(V)}$ with $Z_{<j}^{(V)} \times \prod_{e \in \binom{V}{j}} Z_{=j}^{(e)}$ in the obvious manner.

Suppose that $e$ is a vertex set of size $|e|=j$, $z \in Z_{<j}^{(e)}$, and $a \in A_{=j}^{(e)}$.  We define the \emph{cell} $C_a \subset Z_{=j}^{(e)}$ associated to $a$ by the formula
$$ C_a := (\alpha_{=j}^{(e)})^{-1}(\{a\}) = \{ z \in Z_{=j}^{(e)}: \alpha_j(z(\phi)) = a(\phi) \hbox{ for all } \phi \in \Hom([j],e) \}$$
and then define the measure $\nu_{e,z,a} \in \Pr( Z_{=j}^{(e)} )$ to equal the conditioned measure $(Q_j^{(e)}(z)|C_a)$ (as defined in Appendix \ref{prob}) if $Q_j^{(e)}(z)(C_a) > 0$, or $Q_j^{(e)}(z)$ otherwise.

We then define the natural transformation $Q_{\alpha,j}: Z_{<j} \times A_{\geq j} \to Z_{\leq j} \times A_{>j}$ by the formula
\begin{equation}\label{qjw}
Q_{\alpha,j}^{(V)}( z_{<j}, a_j, a_{>j} ) := \delta_{z_{<j}} \times \prod_{e \in \binom{V}{j}} \nu_{e,z_{<j}\downharpoonright_e,a_j\downharpoonright_e} \times \delta_{a_{>j}}
\end{equation}
for all vertex sets $V$ and all $z_{<j} \in Z_{<j}^{(V)}$, $a_j \in A_{=j}^{(V)}$, $a_{>j} \in A_{>j}^{(V)}$, where we identify $Z_{<j}^{(V)} \times A_{\geq j}^{(V)}$ with
$Z_{<j}^{(V)} \times A_{=j}^{(V)} \times A_{>}^{(V)}$ and $Z_{\leq j}^{(V)} \times A_{>j}^{(V)}$ with $Z_{<j}^{(V)} \times \prod_{e \in \binom{V}{j}} Z_{=j}^{(e)} \times A_{>j}^{(V)}$ in the obvious manner.  Note that we can factor $Q_{\alpha,j} = {Q'}_{\alpha,j} \oplus \id_{A_{>j}}$, where ${Q'}_{\alpha,j}: Z_{<j} \times A_{=j} \to Z_{\leq j}$ is defined by
\begin{equation}\label{qjw2}
{Q'}_{\alpha,j}^{(V)}( z_{<j}, a_j ) := \delta_{z_{<j}} \times \prod_{e \in \binom{V}{j}} \nu_{e,z_{<j}\downharpoonright_e,a_j\downharpoonright_e}.
\end{equation}

(Compare this with Example \ref{d1}.)

One easily verifies that $Q_{\alpha,j}$ is a natural transformation, is $j$-independent, and obeys claim (i) of Proposition \ref{disc-ident}.  Also observe from construction that $\nu_{e,z,a}$ is absolutely continuous with respect to $Q_j^{(e)}(z)$ for all $0 \leq j \leq k$, $|e|=j$, $z \in Z_{<j}^{(e)}$, and $a \in A_{=j}^{(e)}$.  By \eqref{pjw}, \eqref{qjw}, and Lemma \ref{pac}, we conclude the absolute continuity relationship
$$
Q_{\alpha,j}^{(V)}( z_{<j}, a_j, a_{>j} ) \ll P_j^{(V)}(z_{<j}) \times \delta_{a_{>j}}$$
for all finite vertex sets $V$, all $0 \leq j \leq k$, and all $z_{<j} \in Z_{<j}^{(V)}$, $a_{=j} \in A_{=j}^{(V)}$, $a_{>j} \in A_{>j}^{(V)}$.  Iterating this using Lemma \ref{pac}, we obtain claim (ii) of Proposition \ref{disc-ident}.

It remains to prove claim (iii) of Proposition \ref{disc-ident}, which is the most difficult estimate.  The key tool will be Littlewood's principle (Lemma \ref{dom}).  For inductive reasons we need to prove the following rather technical statement.  For any $-1 \leq j \leq k$, we introduce the exchangeable $Z_{\leq j}$-recipe $\mu_{\leq j}: \pt \to Z_{\leq j}$ by the formula
$$ \mu_{\leq j} := P_j \circ \ldots \circ P_0$$
and the natural transformation $T_{\leq j}: Z_{\leq j} \to Z_{\leq j}$ to be the unique natural transformation such that
$$ T_{\leq j} \circ \pi_{Z \to Z_{\leq j}} = \pi_{Z_{\leq j} \times A_{>j} \to Z_{\leq j}} \circ Q_{\alpha,j} \circ \ldots \circ Q_{\alpha,0} \circ \overline{\alpha}$$
where $\pi_{Z \to Z_{\leq j}}$ and $\pi_{Z_{\leq j} \times A_{>j} \to Z_{\leq j}}$ are the projection natural transformations.

\begin{lemma}[Convergence to the diagonal]\label{coerce-conv} Let the notation and assumptions be as above.  Let $V$ be a finite vertex set, let $-1 \leq j \leq k$, Let $H$ be a finite-dimensional Hilbert space depending on $\alpha_{j+1},\ldots,\alpha_k,V$ but independent of $\alpha_0,\ldots,\alpha_j$, and let $F: Z_{\leq j}^{(V)} \to H$ be a bounded measurable function which can depend on $\alpha_{j+1},\ldots,\alpha_k,V$ but is independent of $\alpha_0,\ldots,\alpha_j$.  Then
\begin{equation}\label{zetaj}
 \int_{Z_{\leq j}^{(V)}}
\left[ \int_{Z_{\leq j}^{(V)}} \| F(z)-F(w) \|_H\ T_{\leq j}^{(V)}(z, dw) \right]\ d\mu_{\leq j}^{(V)}(z) = o_{\alpha \to \infty}(1).
\end{equation}
\end{lemma}

\begin{proof}  We induct on $j$.  The case $j =-1$ is vacuously true, so suppose that $0 \leq j \leq k$ and that the claim has already been proven for $j-1$.

Fix $V, H, F$; we may normalise $F$ to be bounded in magnitude by $1$.  It is convenient to use the language of probability rather than measure theory.
Let $z \in Z_{\leq j}^{(V)}$ be drawn at random with law $\mu_{\leq j}^{(V)}$, and then for fixed $z$, let $w$ be drawn at random with law $T_{\leq j}^{(V)}(z)$.  Our task is to show that
\begin{equation}\label{fzw}
\E \| F(z) - F(w) \|_H = o_{\alpha \to \infty}(1).
\end{equation}

We split $z = (z_{<j},z_j)$ and $w = (w_{<j},w_j)$ for $z_{<j}, w_{<j} \in Z_{<j}^{(V)}$ and $z_j, w_j \in Z_j^{(V)}$.  We similarly split $a := \overline{\alpha_{\leq j}}^{(V)}(z) \in A_{\leq j}^{(V)}$ as $a = (a_{<j}, a_j)$.  Observe from construction that

\begin{itemize}
\item $z_{<j} \in Z_{<j}^{(V)}$ has the distribution of $\mu_{\leq j-1}^{(V)}$;
\item Given $z_{<j}$, $a_{<j}$ is determined by the formula $a_{<j} = \overline{\alpha_{<j}}^{(V)}(z_{<j})$;
\item Given $z_{<j}$, $z$ is a random variable with law $P_j^{(V)}(z_{<j})$;
\item Given $z$, $a_j$ is determined by the formula $a_j = \overline{\alpha_{j}}^{(V)}(z_j)$;
\item Given $z_{<j}$, $w_{<j}$ is a random variable with law $T_{\leq j-1}^{(V)}(z_{<j})$;
\item Given $w_{<j}$ and $a_j$, $w$ is a random variable with law ${Q'}_{\alpha,j}^{(V)}(w_{<j},a_j)$ (defined in \eqref{qjw2}).
\end{itemize}

Now we write the left-hand side of \eqref{fzw} as
$$ \sum_{b \in A_{=j}^{(V)}} \E \left(\I( a_j = b ) |F(z) - F(w)| \right)$$
and estimate this using the triangle inequality as the sum of the three expressions
\begin{equation}\label{x1}
\sum_{b \in A_{=j}^{(V)}} \E \left(\I( a_j = b )  \| F(z) - G_{b}(z_{<j}) \|_H\right)
\end{equation}
\begin{equation}\label{x2}
\sum_{b \in A_{=j}^{(V)}} \E \left(\I( a_j = b ) \| G_{b}(z_{<j}) - G_{b}(w_{<j}) \|_H\right)
\end{equation}
and
\begin{equation}\label{x3}
\sum_{b_ \in A_{=j}^{(V)}} \E \left(\I( a_j = b ) \| F(w) - G_{b}(w_{<j}) \|_H\right)
\end{equation}
where $G_{b}: Z_{<j}^{(V)} \to H$ is the measurable function
$$ G_{b}( z_{<j} ) := \int_{Z_{\leq j}^{(V)}} F(z)\ {Q'}_{\alpha,j}^{(V)}( (z_{<j},b), dz ).$$

We will show that each of \eqref{x1}, \eqref{x2}, \eqref{x3} are $o_{\alpha \to \infty}(1)$.

By the induction hypothesis we have
$$ \E \sum_{b \in A_{=j}^{(V)}} \| G_{b}(z_{<j}) - G_{b}(w_{<j}) \|_H = o_{\alpha \to \infty}(1)$$
and so the contribution of \eqref{x2} is acceptable.

Now let us look at \eqref{x1}.  In view of the distribution of $z_{<j}$ and $z$, we can rewrite this expression as $\E f_{\alpha_j}( z_{<j} )$, where
where
$$ f_{\alpha_j}(z_{<j}) :=
\sum_{b \in A_{=j}^{(V)}} P_j^{(V)}(z_{<j})(C_b)
\int_{Z_{\leq j}^{(V)}} \left\| F(y) - \int_{Z_{\leq j}^{(V)}} F(u)\ (P_j^{(V)}(z_{<j},du)|C_b) \right\|_H\ (P_j^{(V)}(z_{<j},dy)|C_b)$$
and
$$ C_b := (\overline{\alpha_{=j}}^{(V)})^{-1}(\{b\}),$$
where we can of course ignore all summands on which $P_j^{(V)}(z_{<j})(C_b) = 0$.
By Lemma \ref{dom}, $f_{\alpha_j}(z_{<j}) = o_{\alpha \to \infty}(1)$ for each $z_{<j}$.  Applying Lemma \ref{dct}, we conclude that \eqref{x1} is $o_{\alpha \to \infty}(1)$ as desired.

Finally we look at \eqref{x3}.  For each $b_j \in A_{=j}^{(V)}$, let $\Omega_{b_j} \subset Z_{<j}^{(V)}$ be the set of all $z_{<j}$ such that the event $\{ a_j = b_j \}$ has non-zero measure with respect to $\overline{P_j}^{(V)}(z_{<j})$.  We split \eqref{x3} further into
\begin{equation}\label{x3-a}
\sum_{b_j \in A_{=j}^{(V)}} \E \left( \I( a_j = b_j ) \I( w_{<j} \in \Omega_{b_j} ) \| F(w) - G_{b_j}(w_{<j}) \|_H\right)
\end{equation}
and
\begin{equation}\label{x3-b}
\sum_{b_j \in A_{=j}^{(V)}} \E \left(\I( a_j = b_j ) \I( w_{<j} \not \in \Omega_{b_j} ) \| F(w) - G_{b_j}(w_{<j}) \|_H\right)
\end{equation}
Consider the expression \eqref{x3-a}.  If $w_{<j} \in \Omega$ and $a_j = b_j$ are fixed, then $w_j$ has the distribution of $\mu_{w_{<j}}$, where $\mu_{w_{<j}}$ was defined in the treatment of \eqref{x1}.  Thus we can bound \eqref{x3-a} by
$$ \E f_{\alpha_j}(w_{<j}).$$
By the induction hypothesis, we have $\E |f_{\alpha_j}(w_{<j}) - f_{\alpha_j}(z_{<j})| = o_{\alpha \to \infty}(1)$, and so the contribution of \eqref{x3-a} is acceptable by our analysis of \eqref{x1}.

Finally, we turn to \eqref{x3-b}.  As $F$ is bounded in magnitude by $1$, we may bound this expression crudely by
$$ 2 \sum_{b_j \in A_{=j}^{(V)}} \E \left(\I( a_j = b_j ) \I( w_{<j} \not \in \Omega_{b_j}^c)\right).$$
By the induction hypothesis we have
$$ 2 \sum_{b_j \in A_{=j}^{(V)}} \E \left|\I( w_{<j} \not \in \Omega_{b_j}) - \I( z_{<j} \not \in \Omega_{b_j})\right| = o_{\alpha \to \infty}(1)$$
so it suffices to show that
$$ 2 \sum_{b_j \in A_{=j}^{(V)}} \E \left(\I( a_j = b_j ) \I( z_{<j} \not \in \Omega_{b_j})\right) = o_{\alpha \to \infty}(1).$$
But if $z_{<j} \in \Omega_{b_j}^c$ then $a_j$ has a zero probability of equaling $b_j$, and so the left-hand side is zero. The claim follows.
\end{proof}

Now we prove Claim (iii) of Proposition \ref{disc-ident}.  Observe from the Stone-Weierstrass theorem that we can approximate any continuous function $F: Z^{(V)} \times Z^{(V)} \to \R$ uniformly by finite linear combinations of tensor products $f(z) g(z')$, where $f: Z^{(V)} \to \R$ and $g: Z^{(V)} \to \R$ are continuous.  By linearity, we may assume that $F$ itself is of this form; thus our task is to show that
$$ \lim_{\alpha \to \infty} \int_{Z^{(V)}} f(z) \left( \int_{Z^{(V)}} g( z')\ T_{\leq k}^{(V)}(z,dz') \right)\ d\mu^{(V)}(z) = \int_{Z^{(V)}} f(z) g(z)\ d\mu^{(V)}(z).$$
By the triangle inequality, it suffices to show that
$$ \lim_{\alpha \to \infty} \int_{Z^{(V)}} \left( \int_{Z^{(V)}} |g( z')-g(z)|\ T_{\leq k}^{(V)}(z,dz')\right)\ d\mu^{(V)}(z) = 0.$$
But this follows immediately from Lemma \ref{coerce-conv}.  The proof of Proposition \ref{disc-ident} (and thus also Theorem \ref{rs-thm-dir}) is now complete.

\subsection{Proof of Proposition \ref{disc-ident2}}\label{disc2-sec}

We now prove Proposition \ref{disc-ident2}, which is the most difficult proposition to establish in this paper.  In Example \ref{d2}, we already saw how the $k=1$ case of this proposition proceeded.  Unfortunately, this case does not capture the full complexity of this proposition, as it does not reveal the difficulty of dealing with ``indistinguishable'' pairs of inputs.  To illustrate the problem, let us informally consider a model case in which $k=2$, $Z = (\pt,Z_1,\{0,1\})$, and $\mu = P_2 \circ P_1$ where $P_1: \pt \to Z_{\leq 1}$ is given from a probability measure $Q_1 \in \Pr(Z_1)$ as in Example \ref{rvc}, and $P_2: Z_{\leq 1} \to Z$ takes the form $P_2^{(V)}(z) = \delta_z \times Q_2^{(V)}(z)$ for some \emph{deterministic} natural transformation $Q_2: Z_{\leq 1} \to Z_{=2}$. We will also assume that $A_2=\{0,1\}$ and that $\alpha_2: \{0,1\} \to \{0,1\}$ is the identity.

We can view the deterministically continuous natural transformation $\zeta_{\alpha}: A \times X_\alpha \to Z$ as a \emph{random} deterministically continuous natural transformation $\zeta: A \to Z$.  Such a natural transformation can be built out of two functions $\zeta_1: A^{([1])} \to Z_{=1}^{([1])}$ and $\zeta_2: A^{([2])} \to Z_{=2}^{([2])}$ by requiring that
$$ \zeta^{([j])}(a)_j(\phi) = \zeta_j(a)(\phi)$$
for $j=1,2$, $a \in A^{([j])}$, and $\phi \in \Hom([j],[j])$.  Any two functions $\zeta_1, \zeta_2$ will determine a deterministically continuous natural transformation, so long as $\zeta_2$ is $\Hom([2],[2])$-equivariant.  On the other hand, to get the convergence to the diagonal, we would like to have $\alpha_{=j}^{([j])}(\zeta_j(a)) = a_j$ for all $j=0,1$ and ``most'' $a_j \in A^{([j])}$ (with respect to the measure $\overline{\alpha} \circ \mu^{([j])}$).  We also need to select the $\zeta_j(a)$ in a suitably ``absolutely continuous'' manner.

We build $\zeta_1$ and $\zeta_2$ as follows.  For each $a_1 \in A_1 \equiv A^{([1])}$, define the cell $C_{a_1} := \alpha_1^{-1}(\{a_1\})$, and select $\zeta_1(a_1) \in Z_1^{([1])}$ independently at random with law $(Q_1|C_{a_1})$ if $Q_1(C_{a_1})>0$, and with law $Q_1$ otherwise; this already ensures that $\alpha_1 \circ \zeta_1$ converges to the diagonal (by Littlewood's principle).  Then, we can define $\zeta_2$ by $\zeta_2(a) := Q_2(\zeta_1 \circ a_1)$.  Note that as long as $1$ and $2$ are \emph{distinguishable} in the sense that $a_1(1)\neq a_1(2)$, the distribution of $\zeta_1 \circ a_1 \in Z_1^2$ will be absolutely continuous with respect to $Q_1^2$, as $\zeta_1(a_1(1))$ and $\zeta_1(a_2(2))$ are independent and individually absolutely continuous with respect to $Q_1$.  The required absolute continuity and convergence properties would be relatively easy to establish if the distinguishable case was the only case.  However, in the indistinguishable case $a_1(1)=a_1(2)$, the random variable $\zeta_1 \circ a_1$ is no longer absolutely continuous with respect to $Q_1^2$, being concentrated on the diagonal of $Z_1^2$ (which can have zero measure), and so convergence and absolute continuity in this case is not immediately clear.  To resolve this issue, observe that if $Z_1$ is atomless with respect to $Q_1$ then this indistinguishable case will be asymptotically negligible for sufficiently fine $A_1$; on the other hand, if $Z_1$ does contain atoms, then the diagonal of $Z_1^2$ acquires a positive measure with respect to $Q_1^2$, and so the difficulty again disappears.  Note however that our analysis had to take note of what symmetries were obeyed by the input $a$.  Later on we shall see that we will need to describe these symmetries in general by a certain \emph{groupoid} $R_a$.

We now begin the full proof of Proposition \ref{disc-ident2}.  Let $Z, k, \mu, \alpha$ be as in that proposition.  To simplify the notation slightly we shall omit some subscripts on $\alpha$.
By  Definition \ref{regexp}, we can factor
$$ \mu = P_k \circ \ldots \circ P_0$$
where for each $0 \leq j \leq k$, $P_j: Z_{<j} \to Z_{\leq j}$ is a $j$-independent natural transformation such that
\begin{equation}\label{pizz}
\pi_{Z_{\leq j} \to Z_{<j}} \circ P_j = \id_{Z_{<j}}.
\end{equation}

Our objective is to find a probability sub-Cantor space $(X,\nu)$
and a deterministically continuous natural transformation $\zeta: A \times X \to Z$ such that
$$ \zeta^{([k])} \circ ((\overline{\alpha} \circ \mu) \oplus \nu)^{([k])}$$
is $o_{\alpha \to \infty}(1)$-absolutely continuous with respect to $\mu^{([k])}$,
and which converges to the diagonal in the sense that
\begin{equation}\label{lima}
\lim_{\alpha \to \infty} \int_{Z^{(V)}}
\int_{X} F( z, \zeta^{(V)}( \overline{\alpha}^{(V)}(z),x) )\ d\nu(x) d\mu^{(V)}(z) = \int_{Z^{(V)}} F(z,z)\ d\mu^{(V)}(z)
\end{equation}
for all finite vertex sets $V$ and all continuous $F: Z^{(V)} \times Z^{(V)} \to \R$.

This will follow from the $j=k$ case of following inductive proposition.

\begin{proposition}[Inductive discretisation]\label{indiscrete} For any $-1 \leq j \leq k$, there exists a probability sub-Cantor space $(X_j,\nu_j)$ and
a deterministically continuous natural transformation $\zeta_{\leq j}: A_{\leq j} \times X_j \to Z_{\leq j}$ such that
\begin{equation}\label{ejc-ac}
 \zeta_{\leq j}^{(V)} \circ ((\overline{\alpha}_{\leq j} \circ \mu_{\leq j}) \oplus \nu_j)^{(V)} \ll_{o_{\alpha \to \infty}(1)} \mu_{\leq j}^{(V)}
\end{equation}
for all finite vertex sets $V$, and for which we have the convergence property
\begin{equation}\label{limz}
\lim_{\alpha \to \infty} \int_{Z_{\leq j}^{(V)}}
\left[ \int_{Z_{\leq j}^{(V)}} \| F(z)-F(w) \|_H\ T_{\leq j}^{(V)}(z, dw) \right]\ d\mu_{\leq j}^{(V)}(z) = 0
\end{equation}
for all finite vertex sets $V$, all finite-dimensional Hilbert spaces $H$, and all bounded measurable $F: Z_{\leq j}^{(V)} \to H$, where $T_{\leq j}: Z_{\leq j} \to Z_{\leq j}$ is the natural transformation
$$ T_{\leq j} := \zeta_{\leq j} \circ (\overline{\alpha}_{\leq j} \oplus \nu_j),$$
$\mu_{\leq j}: \pt \to Z_{\leq j}$ is the exchangeable $Z_{\leq j}$-recipe
$$ \mu_{\leq j} := P_j \circ \ldots \circ P_0,$$
and we allow $H$, $F$ to depend on $\alpha_{j+1},\ldots,\alpha_k$ (but must be independent of $\alpha_0,\ldots,\alpha_j$).
\end{proposition}

Indeed, to establish \eqref{lima} from the $j=k$ case of \eqref{limz} one repeats the arguments at the end of the previous section.

\begin{remark} It will be more convenient to interpret \eqref{limz} probabilistically, as the assertion that if $V$ is a vertex set, $F: Z_{\leq j}^{(V)} \to H$ is a bounded measurable map into a finite-dimensional Hilbert space, $z \in Z_{\leq j}^{(V)}$ is drawn at random with law $\mu_{\leq j}^{(V)}$, $x \in X_j$ is drawn at random with law $\nu_j$, $a := \overline{\alpha_{\leq j}}^{(V)}(z) \in A_{\leq j}^{(V)}$, and $w := \zeta_{\leq j}^{(V)}(a,x)$, then
\begin{equation}\label{limzo}
\E \|F(z) - F(w)\|_H = o_{\alpha \to \infty}(1)
\end{equation}
where the decay rate $o_{\alpha \to \infty}$ depends of course on $F$ and $H$.
\end{remark}

It remains to prove Proposition \ref{indiscrete}.  The case $j=-1$ is trivial, so suppose inductively that $0 \leq j \leq k$ and that the claim has already been proven for $j-1$.  To simplify the notation slightly we shall just consider the case $j=k$; actually, we can reduce to this case by discarding all components of $Z, \alpha, A$ of order greater than $j$, and then reducing $k$ to $j$.

Henceforth $j=k$. Let $(X_{k-1},\nu_{k-1})$ and $\zeta_{<k} := \zeta_{\leq k-1}$ be given by the inductive hypothesis.

\subsubsection{Construction of $X_k$ and $\zeta_{\leq k}$}

Let $\Xi$ denote the collection of all $\Hom([k],[k])$-equivariant maps $\xi: A^{([k])} \to Z_{=k}^{([k])}$; observe that $\Xi$ is a compact subset of $(Z_{=k}^{([k])})^{A^{([k])}}$ and is thus a sub-Cantor space.  We refer to elements $\xi$ of $\Xi$ as \emph{$k$-rules}.

We set $X_k := X_{k-1} \times \Xi$, and let $\zeta_{\leq k}: A \times X_k \to Z$ be the unique deterministically continuous natural transformation with the following two properties:
\begin{itemize}
\item ($\zeta_{\leq k}$ extends $\zeta_{<k}$) We have the identity
$$ \pi_{Z \to Z_{<k}} \circ \zeta_{\leq k} = \zeta_{<k} \circ \pi_{A \times X_k \to A_{<k} \times X_{k-1}}$$
where $\pi_{Z \to Z_{<k}}: Z \to Z_{<k}$ and $\pi_{A \times X_k \to A_{<k} \times X_{k-1}}: A \times X_k \to A_{<k} \times X_{k-1}$ are the projection natural transformations.
\item ($\zeta_{\leq k}$ extends $\xi$)  We have
$$ \zeta_{\leq k}^{([k])}(a,(x,\xi))_k := \xi(a)$$
for all $a \in A^{([k])}$, $x \in X_{k-1}$, and $\xi \in \Xi$.
\end{itemize}

More explicitly, $\zeta_{\leq k}$ is given by the formula
$$ \zeta_{\leq k}^{(V)}( (a_{<k}, a_k), (x,\xi) )
:= \left( \zeta_{<k}^{(V)}(a_{<k}, x) , \left( Z^{(\phi_e^{-1})}_{=k}(\xi( A^{(\phi_e)}(a_{<k},a_k) ) ) \right)_{e \in \binom{V}{k}} \right)$$
for all vertex sets $V$, all $(a_{<k}, a_k) \in A^{(V)} \equiv A_{<k}^{(V)} \times A_{=k}^{(V)}$, all $x \in X_k$, and $\xi \in \Xi$, where for each $e \in \binom{V}{k}$, $\phi_e$ is an arbitrary morphism from $[k]$ to $e$ (the exact choice of morphism is not relevant, thanks to the $\Hom([k],[k])$-invariance of $\xi$), and where we identify $Z^{(V)}$ with $Z_{<k}^{(V)} \times \prod_{e \in \binom{V}{k}} Z_{=k}^{(e)}$.  One easily verifies that $\zeta_{\leq k}$ is a deterministically continuous natural transformation.

\subsubsection{Construction of $\nu_k$}

To construct the measure $\nu_k \in \Pr(X_k)$ we will need some more notation.

\begin{definition}[Invariant space, stabiliser, indistinguishability]\label{indistinguished}  Let $Y$ be a sub-Cantor palette, and $V$, $W$ be vertex sets.
\begin{itemize}
\item If $G \leq \Hom(V,V)$ is a group, we define the \emph{$G$-invariant space} $(Y^{(V)})^G := \{ y \in Y^{(V)}: Y^{(\phi)}(y) = y \hbox{ for all } \phi \in G$\}; this is a compact subspace of $Y^{(V)}$.
\item If $y \in Y^{(V)}$, we define the \emph{stabiliser} $\Stab(y) := \{ \phi \in \Hom(V,V): Y^{(\phi)}(y) = y\}$; this is a subgroup of $\Hom(V,V)$.
\item We say that two elements $y \in Y^{(V)}$, $y' \in Y^{(W)}$ are \emph{indistinguishable} if there exists an invertible $\phi \in \Hom(V,W)$ such that $Y^{(\phi)}(y) = y'$ (in particular, this requires $V$ and $W$ to have equal cardinality), and \emph{distinguishable} otherwise.
\end{itemize}
\end{definition}

\begin{remark} Note that deterministically continuous natural transformations are forced to map indistinguishable elements to indistinguishable elements.  In particular, the images of indistinguishable elements cannot be set independently.  This lack of independence will cause significant technical difficulties in our arguments.  A similar difficulty will also be caused by the fact that deterministically continuous natural transformations must map $G$-invariant spaces into $G$-invariant spaces (or equivalently, they cannot decrease the stabiliser of an element).
\end{remark}

\begin{definition}[Vertical ingredient]\label{vertmes}  We define $Q: Z_{<k} \to Z_k$ to be the unique natural transformation such that
\begin{equation}\label{pjwq}
\delta_{z} \times Q^{(V)}(z) := P_k^{(V)}( z );
\end{equation}
for all vertex sets $V$ and all $z \in Z_{<k}^{(V)}$; this is well defined from \eqref{pizz}.
\end{definition}

\begin{definition}[Cell]\label{celldef}  If $V$ is a vertex set, $G \leq \Hom(V,V)$ and $a_k \in A_{=k}^{(V)}$, we define the \emph{cell}
$$C_{V,G,a_k} := \{ z \in (Z_{=k}^{(V)})^G: \overline{\alpha_{=k}}^{(V)}(z) = a_k\};$$
this is a compact subspace of $Z_{=k}^{(V)}$.
\end{definition}

\begin{definition}[Default point]  We arbitrarily select a \emph{default point} $z_* \in Z_k$.  For any $V$, we define $\overline{z_*}^{(V)} \in Z_k^{(V)}$ by setting $\overline{z_*}^{(V)}(\phi) := z_*$
for all $\phi \in \Hom([k],V)$.
\end{definition}

\begin{remark} The point $z_*$ is only needed for technical reasons, as a sort of ``error message'' to output when certain inputs are ``bad''.  The exact value of $z_*$ plays no role in our arguments.
\end{remark}

\begin{definition}[Quadruples]\label{quad}  If $V$ is a vertex set, $G \leq \Hom(V,V)$, $a_k \in A_{=k}^{(V)}$, and $z_{<k} \in Z_{<k}^{(V)}$, we say that $(V,G,a_k,z_{<k})$ is \emph{good} if $Q^{(V)}(z_{<k})(C_{V,G,a_k}) > 0$, and \emph{bad} otherwise.  We define the probability measure $\rho_{V,G,a_k,z_{<k}} \in \Pr( (Z_{=k}^{(V)})^G )$ to equal the conditioned measure $(Q^{(V)}(z_{<k})|C_{V,G,a_k})$ (as defined in Appendix \ref{prob}) if $(V,G,a_k,z_{<k})$ is good, and $\delta_{\overline{z_*}^{(V)}(\phi)}$ otherwise.
\end{definition}

By using the natural transformation properties heavily, we observe that the probability measures
$\rho_{V,G,a_j,w_{<j}}$ are invariant under relabeling in the sense that
for any $G \leq \Hom(V,V)$, $a_j \in A_{=j}^{(V)}$, $w_{<j} \in Z_{<j}^{(V)}$, and any bijection $\phi: V \to W$, that
\begin{equation}\label{zanu}
\rho_{\phi G \phi^{-1}, A_{=j}^{(\phi)}(a_j), Z_{<j}^{(\phi)}(w_{<j})} = Z_{=j}^{(\phi)} \circ \rho_{G,a_j,w_{<j}}.
\end{equation}

\begin{definition}[Random $k$-rules]\label{rkr} If $x \in X_{<k}$, we define $\eta_{x} \in \Pr( \Xi )$ to be the unique law for a random $k$-rule $\xi \in \Xi$ with the following properties:
\begin{itemize}
\item For each $a = (a_{<k},a_k) \in A^{([k])}$, the random variable $\xi(a) \in Z_{=k}^{([k])}$ has the law of $\rho_{[k], \Stab(a), a_k, \zeta_{<k}^{([k])}(a_{<k},x)}$.
\item If $a_1,\ldots,a_n \in A^{([k])}$ are pairwise distinguishable, then the random variables $\xi(a_1),\ldots,\xi(a_n) \in Z_{=k}^{([k])}$ are jointly independent.
\end{itemize}
\end{definition}

\begin{remark} The probability distribution $\eta_{x}$ can be constructed explicitly as follows.  The equivalence relation of indistinguishability partitions $A^{([k])}$ into finitely many equivalence classes.  For each equivalence class $O$, select a representative $a \in O$ arbitrarily, and draw $\xi(a)$ independently at random with law $\rho_{[k], \Stab(a), a_k, \zeta_{<k}^{([k])}(a_{<k},x)}$.  Then for any $\phi \in \Hom([k],[k])$, we set $\xi( A^{(\phi)}(a)) := Z_{=k}^{(\phi)}(\xi(a))$.  One easily verifies (using \eqref{zanu}) that this defines a random $k$-rule $\xi$, and that the law $\eta_{x}$ for $\xi$ has the desired properties; it is also easy to see that this law is unique.
\end{remark}

With all these definitions, we can now define the measure $\nu_{\leq k} \in \Pr(X_k)$ by the formula
$$ \nu_{\leq k} := \int_{X_{k-1}} \delta_x \times \eta_x\ d\nu_{<k}(x).$$
Informally, $\nu_{\leq k}$ is the law of the pair $(x,\xi)$, where $x \in X_{k-1}$ is selected with law $\nu_{<k}$, and then $\xi \in \Xi$ is selected with law $\eta_x$.

\begin{remark} When $k=1$, the construction simplifies substantially since it is not possible for two distinct elements of $A^{([1])}$ to be indistinguishable, and one essentially obtains the construction in Example \ref{d2}.
\end{remark}

It remains to verify the properties \eqref{ejc-ac}, \eqref{limz} (with $j=k$).

\subsubsection{Most quadruples are good}

The first task is to show that the bad quadruples (which output the ``error message'' $z_*$) are negligible.

\begin{proposition}[Most quadruples are good]\label{quadgood} Let $z \in Z^{([k])}$ be drawn at random with law $\mu^{([k])}$, let $a = (a_{<k},a_k) := \overline{\alpha}^{([k])}(z) \in A^{([k])}$, and let $x \in X_{k-1}$ be drawn independently at random with law $\nu_{k-1}$.  Let $w_{<k} := \zeta_{<k}^{([k])}(a_{<k},x) \in Z_{<k}^{([k])}$.  Then the quadruple $([k], \Stab(a), a_k, w_{<k})$ is good with probability $1 - o_{\alpha \to \infty}(1)$.
\end{proposition}

\begin{proof}  The key idea of the proof is to exploit the fact (essentially arising from the monotone (or dominated) convergence theorem) that most elements of the cells $C_{V,G,a_k}$ tend to inherit the symmetries of their colour $a_k$ when the colouring $\alpha$ is sufficiently fine.

By Definition \ref{quad}, our task is to show that
$$
\Prob( Q^{([k])}(w_{<k})(C_{[k],\Stab(a),a_k}) = 0 ) = o_{\alpha \to \infty}(1).$$
The number of possible stabiliser subgroups $\Stab(a) \leq \Hom([k],[k])$ is bounded (independently of $A$ or $\alpha$), so it suffices by the union bound to show that
$$
\Prob( \Stab(a) = G \hbox{ and } Q^{([k])}(w_{<k})(C_{[k],G,a_k}) = 0 ) = o_{\alpha \to \infty}(1)$$
for each fixed group $G \leq \Hom([k],[k])$.

Fix $G$.  If $\Stab(a)=G$, then $a \in (A^{([k])})^G$; thus (by the natural transformation properties of $\zeta_{<k}$) we see that $w_{<k}$ lies in $(Z_{<k}^{([k])})^G$.  From this and Definition \ref{celldef} we see that
$$ \{ w_{<k}\} \times C_{[k],G,a_k} = (\{w_{<k}\} \times Z_{=k}^{([k])}) \cap (Z^{([k])})^G \cap C'_{a_k}$$
where for any $b \in A_{=k}^{([k])}$, $C'_b \subset Z^{([k])}$ is the set
$$ C'_{b} := Z_{<k}^{([k])} \times (\overline{\alpha_{=k}}^{([k])})^{-1}(\{b\}).$$
From this and Definition \ref{vertmes}, we see that
$$ Q^{([k])}(w_{<k})(C_{[k],G,a_k}) = P_k^{([k])}(w_{<k})( (Z^{([k])})^G \cap C'_{a_k} )$$
so it suffices to show that for every $\eps > 0$, we have
$$
\Prob( \Stab(a) = G \hbox{ and } F_{a_k}(w_{<k}) = 0 ) \ll \eps$$
for all sufficiently fine $\alpha$ (depending on $\eps$), where for any $b \in Z_{=k}^{([k])}$, $F_b: Z_{<k}^{([k])} \to [0,1]$ is the measurable function
$$ F_b(y) := P_k^{([k])}(y)( (Z^{([k])})^G \cap C'_{b} ).$$

Fix $\eps$, and set $\delta := \eps/|A_{=k}^{([k])}|$.  Observe that
$$
\I( F_{a_k}(w_{<k}) = 0 )
\leq \sum_{b \in A_{=k}^{([k])}} \frac{1}{\delta} |F_{b}(w_{<k}) - F_{b}(z_{<k})|
+ \I( F_{a_k}(z_{<k}) \leq \delta ).$$
From the inductive hypothesis \eqref{limz} (or \eqref{limzo}) we have
$$ \E \sum_{b \in A_{=k}^{([k])}} \frac{1}{\delta} |F_{b}(w_{<k}) - F_{b}(z_{<k})| \ll \eps$$
for sufficiently fine $\alpha$, so it suffices to show that
\begin{equation}\label{astab}
\Prob( \Stab(a) = G \hbox{ and } F_{a_k}(z_{<k}) \leq \delta ) \ll \eps.
\end{equation}
Recall that $z_{<k}$ is distributed with law $\mu_{<k}^{([k])}$, and for fixed $z_{<k}$, $z_k$ is distributed with law $P_k^{([k])}(z_{<k})$.  In particular, for any $b \in A_{=k}^{([k])}$ and fixed $z_{<k}$, we have $a_k = b$ with probability $P_k^{([k])}(y)( C'_{b} )$.  Thus we may express the left-hand side of
\eqref{astab} as
$$
\E \sum_{b \in A_{=k}^{([k])}} P_k^{([k])}(z_{<k})( C'_{b} ) \I(\Stab(a_{<k},b) = G) \I( F_b(z_{<k}) \leq \delta ).$$
We can split
$$P_k^{([k])}(y)( C'_{b} ) = F_b(y) + P_k^{([k])}(y)( C'_{b} \backslash (Z^{([k])})^G ).$$
Since
$$
\E \sum_{b \in A_{=k}^{([k])}} F_b(z_{<k}) \I( F_b(z_{<k}) \leq \delta )
\leq |A_{=k}^{([k])}| \delta = \eps,$$
it thus suffices to show that
$$
\E \sum_{b \in A_{=k}^{([k])}} P_k^{([k])}(z_{<k})( C'_{b} \backslash (Z^{([k])})^G ) \I(\Stab(a_{<k},b) = G) \ll \eps.$$
Now observe that if $\Stab(a_{<k},b) = G$, then on the support $\{z_{<k}\} \times Z_{=k}^{([k])}$ of $P_k^{([k])}(z_{<k})$, the set $C'_{b}$ is contained in the set $(\overline{\alpha}^{([k])})^{-1}(( A^{([k])})^G)$.  Thus we have
$$ \sum_{b \in A_{=k}^{([k])}} P_k^{([k])}(z_{<k})( C'_{b} \backslash (Z^{([k])})^G ) \I(\Stab(a_{<k},b) = G) \leq P_k^{([k])}(z_{<k})( (\overline{\alpha}^{([k])})^{-1}(( A^{([k])})^G) \backslash (Z^{([k])})^G );$$
since $\mu^{([k])} = P_k^{([k])} \circ\mu_{<k}^{([k])}$, it thus suffices to show that
\begin{equation}\label{ag}
\mu^{([k])}( (\overline{\alpha}^{([k])})^{-1}(( A^{([k])})^G) \backslash (Z^{([k])})^G ) \ll \eps
\end{equation}
for sufficiently fine $\alpha$.

Now observe that if $z \in Z^{([k])}$ is not $G$-invariant (i.e. $z \not \in (Z^{([k])})^G$), then (since the algebra of clopen subsets in sub-Cantor spaces separate points) there exists a colouring $\alpha: Z \to A$ such that $\alpha^{([k])}(z)$ is also not $G$-invariant.  This property is then inherited by all refinements of $\alpha$.  As a consequence we see that
$$ \I( z \in (\overline{\alpha}^{([k])})^{-1}((A^{([k])})^G) \backslash (Z^{([k])})^G) = o_{\alpha \to \infty}(1)$$
for all $z \in Z^{([k])}$.  The claim \eqref{ag} then follows from the dominated convergence theorem (Lemma \ref{dct}).
\end{proof}

\subsubsection{Decoupling}

Let $V$ be a finite vertex set, let $z \in Z^{(V)}$ be drawn at random with law $\mu^{(V)}$, let $a := \overline{\alpha}^{(V)}(z) \in A^{(V)}$, and let $x \in X_{k-1}$ be drawn independently at random with law $\nu_{k-1}$.  We then draw $\xi \in \Xi$ with law $\eta_x$, and set $w \in Z^{(V)}$ to be the point $w := \zeta_{\leq k}(a,(x,\xi))$.  We split $z = (z_{<k},z_k)$, $a = (a_{<k},a_k)$, and $w = (w_{<k},w_k)$ in the usual manner.

Let us temporarily freeze $z, a, x$, so that the only remaining source of randomness comes from $\xi$.
The lower order components $w_{<k}$ of $w$ do not depend on $\xi$ and are now deterministic; indeed, we have $w_{<k} = \zeta_{<k}(a_{<k},x)$.
If we split the top component $w_k$ as $w_k = (w_k\downharpoonright_e)_{e \in \binom{V}{k}}$, then we see that each piece $w_k\downharpoonright_e$ depends on $\xi$ via the formula
$$ Z^{(\phi_e^{-1})}_{=k}(\xi( A^{(\phi_e)}(a_{<k},a_k) ) ).$$
From this and Definition \ref{rkr} (and \eqref{zanu}), we see that $w_k\downharpoonright_e$ is distributed (for fixed $z,a,x$) according to the law
$$\rho_{e, \Stab(a\downharpoonright_e), a_k\downharpoonright_e, \zeta_{<k}^{(e)}(a_{<k}\downharpoonright_e,x)}
= \rho_{e, \Stab(a\downharpoonright_e), a_k\downharpoonright_e, w_{<k}\downharpoonright_e}.$$
In particular, we almost surely have the constraint
\begin{equation}\label{wke}
 w_k\downharpoonright_e \in (Z_{=k}^{(e)})^{\Stab(a\downharpoonright_e)},
\end{equation}
thus $w_k\downharpoonright_e$ needs to inherit all the symmetries that $a\downharpoonright_e$ has.

From Definition \ref{rkr}, we also see that for $e_1,\ldots,e_n \in \binom{V}{k}$, the pieces $w_k\downharpoonright_{e_1}, \ldots, w_k\downharpoonright_{e_n}$ are jointly independent so long as the $a\downharpoonright_{e_1},\ldots,a\downharpoonright_{e_n}$ are pairwise distinguishable.

On the other hand, if $e, e' \in \binom{V}{k}$ are such that $a\downharpoonright_{e}$ and $a\downharpoonright_{e'}$ are indistinguishable, thus we have $a\downharpoonright_{e'} = A^{(\phi)}(a\downharpoonright_e)$ for some $\phi \in \Hom(e,e')$, then $w_k\downharpoonright_e$ and $w_k\downharpoonright_{e'}$ are coupled together via the constraint
\begin{equation}\label{wke2}
w_k\downharpoonright_e = Z_{=k}^{(\phi)}( w_k\downharpoonright_{e'} ).
\end{equation}
Note that \eqref{wke} can be viewed as the special case of $e=e'$ of \eqref{wke2}.

Motivated by the above discussion, for every $b \in A^{(V)}$, let $R_b$ be the set of all triples $(e,e',\phi)$, where $e,e' \in \binom{V}{k}$ and $\phi \in \Hom(e,e')$ is such that $A^{(\phi)}(b\downharpoonright_{e'}) = b\downharpoonright_e$, thus $R_b$ collects all the ways in which components of $b$ are indistinguishable from each other.  The set $R_b$ is a \emph{groupoid}, in the sense that
\begin{itemize}
\item For every $e \in \binom{V}{k}$, the triple $(e,e,\id_e)$ lies in $R_b$.
\item If $(e,e',\phi)$ lies in $R_b$, then $(e',e,\phi^{-1})$ lies in $R_b$.
\item If $(e,e',\phi)$ and $(e',e'',\psi)$ lie in $R_b$, then $(e,e'',\psi \circ \phi)$ lies in $R_b$.
\end{itemize}
Observe that for any $e \in \binom{V}{j}$, the stabiliser $\Stab(b\downharpoonright_e)$ of $b$ restricted to $e$ can be recovered from $R_b$ by the formula
$$\Stab(b\downharpoonright_e) = \{ \phi \in \Hom(e,e): (e,e,\phi) \in R_b \}.$$

Let us call $e,e' \in \binom{V}{k}$ \emph{$R$-indistinguishable} for some groupoid $R$ if there exists $\phi \in \Hom(e,e')$ with $(e,e',\phi) \in R$, and \emph{$R$-distinguishable} otherwise.  As $R$ is a groupoid, the property of being $R$-indistinguishable is an equivalence relation.

Given a groupoid $R$, an element $b \in A_{=k}^{(V)}$, and an element $y \in Z_{<k}^{(V)}$, we then define the probability measure $\sigma_{V,R,b,y} \in \Pr(Z^{(V)})$ to be the unique probability distribution of a random variable $w = (w_{<k},w_k) \in Z^{(V)}$ such that
\begin{itemize}
\item $w_{<k} = y$;
\item For each $e \in \binom{V}{j}$, $w_k\downharpoonright_e$ has the distribution of $\rho_{e,G_e,b\downharpoonright_e, y\downharpoonright_e}$, where $G_e \leq \Hom(e,e)$ is the group $G_e := \{ \phi \in \Hom(e,e): (e,e,\phi) \in R \}$;
\item For any $(e,e',\phi) \in R$, we have the constraint \eqref{wke2};
\item For any $e_1,\ldots,e_n \in \binom{V}{k}$ which are pairwise $R$-distinguishable, the
random variables $w_k\downharpoonright_{e_1},\ldots,w_k\downharpoonright_{e_n}$ are jointly independent.
\end{itemize}

One can construct $\sigma_{V,R,b,y}$ more explicitly by choosing one representative $e$ from each $R$-indistinguishable equivalence class, selecting $w_k\downharpoonright_e$ independently at random for each such representative with law $\rho_{e,G_e,b\downharpoonright_e, y\downharpoonright_e}$, and then extending to all other $e$ by \eqref{wke2}.

By the previous discussion, we see that for fixed $z,a,x$, $w$ is distributed according to the law $\sigma_{V,R_a,a_k,w_{<k}}$.

We would like to remove the couplings \eqref{wke}, \eqref{wke2} from
this distribution.  To this end, we define the \emph{trivial
groupoid} $R_0 := \{ (e,e,\id_e): e \in \binom{V}{k} \}$. We would
like to assert that the probability measure
$\sigma_{V,R_a,a_k,w_{<k}}$ is close to $\sigma_{V,R_0,a_k,w_{<k}}$
in the total variation norm $\|\cdot \|_{M(Z^{(V)})}$ on $Z^{(V)}$,
as defined in Appendix \ref{prob}.  This is accomplished by the
following key estimate.

\begin{proposition}[$\sigma_{V,R_a,a_k,w_{<k}}$ approximates $\sigma_{V,R_0,a_k,w_{<k}}$]\label{decouple}  Let $V$ be a finite vertex set, let $z \in Z^{(V)}$ be drawn at random with law $\mu^{(V)}$, let $a := \overline{\alpha}^{(V)}(z) \in A^{(V)}$, and let $x \in X_{k-1}$ be drawn independently at random with law $\nu_{k-1}$.  Set $w_{<k} := \zeta_{<k}(a_{<k},x)$.  Then
$$ \E \| \sigma_{V,R_a,a_k,w_{<k}} - \sigma_{V,R_0,a_k,w_{<k}} \|_{M(Z^{(V)})} = o_{\alpha \to \infty}(1).$$
\end{proposition}


\begin{proof}  From the inductive hypothesis \eqref{limz} (or \eqref{limzo}) we have
$$
\E \sum_{b \in A_{=k}^{(V)}} |\| \sigma_{V,R_{a_{<k},b},b,w_{<k}} - \sigma_{V,R_0,b,w_{<k}} \|_{M(Z^{(V)})} - \E \| \sigma_{V,R_{a_{<k},b},b,z_{<k}} - \sigma_{V,R_0,b,z_{<k}} \|_{M(Z^{(V)})}| =o_{\alpha \to \infty}(1)$$
so it suffices to show that
$$ \E \| \sigma_{V,R_a,a_k,z_{<k}} - \sigma_{V,R_0,a_k,z_{<k}} \|_{M(Z^{(V)})} = o_{\alpha \to \infty}(1).$$

The main difficulty here is to understand the effect of the constraints \eqref{wke}, \eqref{wke2} caused by $R_a$.  The number of possible groupoids $R_a$ is bounded independently of $\alpha$.  Thus it suffices to show that
\begin{equation}\label{rar}
 \E \I(R_a=R) \| \sigma_{V,R,a_k,z_{<k}} - \sigma_{V,R_0,a_k,z_{<k}} \|_{M(Z^{(V)})} = o_{\alpha \to \infty}(1)
 \end{equation}
for each groupoid $R$.

Fix $R$.  Recall that $z_{<k}$ is distributed with law $\mu_{<k}^{(V)}$, and for fixed $z_{<k}$, $z_k$ is distributed with law $Q^{(V)}(z_{<k})$, and so for any $b \in A_{=k}^{(V)}$ and fixed $z_{<k}$, $a_k$ will equal $b$ with probability $Q^{(V)}(z_{<k})(C_{V,\{\id\},b})$.  Thus we can rewrite the left-hand side of \eqref{rar} as
$$
\int_{Z_{<k}^{(V)}} \sum_{b \in A_{=k}^{(V)}} Q^{(V)}(y)(C_{V,\{\id\},b}) \I( R_{\overline{\alpha_{<k}}^{(V)}(y),b} = R )
\| \sigma_{V,R,b,y} - \sigma_{V,R_0,b,y} \|_{M(Z^{(V)})}
\ d\mu_{<k}^{(V)}(y).$$

Let $Z_R \subset Z^{(V)}$ denote the set
$$ Z_R := \{ y \in Z^{(V)}: Z^{(\phi)}(y \downharpoonright_{e'}) = y\downharpoonright_e \hbox{ for all } (e,e',\phi) \in R \}$$
and let $A_{R} \subset A^{(V)}$ denote the set
$$ A_R := \{ b \in A^{(V)}: A^{(\phi)}(b\downharpoonright_{e'}) = b\downharpoonright_e \hbox{ for all } (e,e',\phi) \in R \}$$
As in the proof of Proposition \ref{quadgood}, we can use the fact that clopen subsets in sub-Cantor spaces separate points to conclude that
$$ \I( y \in (\overline{\alpha}^{(V)})^{-1}(A_R) \backslash Z_R) = o_{\alpha \to \infty}(1)$$
for each $y \in Z^{(V)}$, and hence by Lemma \ref{dct}
$$ \mu^{(V)}( (\overline{\alpha}^{(V)})^{-1}(A_R) \backslash Z_R ) = o_{\alpha \to \infty}(1)$$
or in other words
$$ \int_{Z_{<k}^{(V)}} \sum_{b \in A_{=k}^{(V)}} Q^{(V)}(y)(C_{V,\{\id\},b} \backslash Z_R) \I( R_{\overline{\alpha_{<k}}^{(V)}(y),b} = R )\ d\mu_{<k}^{(V)}(y) = o_{\alpha \to \infty}(1).$$
Since we have
$$ \int_{Z_{<k}^{(V)}} \sum_{b \in A_{=k}^{(V)}} Q^{(V)}(y)(C_{V,\{\id\},b})\ d\mu_{<k}^{(V)}(y) = 1,$$
we may thus apply Markov's inequality and locate an exceptional set $E \subset Z_{<k}^{(V)} \times A_{=k}^{(V)}$ with
$$ \int_{Z_{<k}^{(V)}} \sum_{b \in A_{=k}^{(V)}} Q^{(V)}(y)(C_{V,\{\id\},b}) \I( (y,b) \in E ) =
o_{\alpha \to \infty}(1)$$
such that
\begin{equation}\label{qvy}
 Q^{(V)}(y)(C_{V,\{\id\},b}) > 0
\end{equation}
and
\begin{equation}\label{qvy2}
 Q^{(V)}(y)(C_{V,\{\id\},b} \backslash Z_R) = o_{\alpha \to \infty}(1) Q^{(V)}(y)(C_{V,\{\id\},b})
\end{equation}
for all $(y,b) \in Z_{<k}^{(V)} \times A_{=k}^{(V)} \backslash E$.

Fix this set $E$.  To finish the proof of \eqref{rar}, it thus suffices to show that
$$ \| \sigma_{V,R,b,y} - \sigma_{V,R_0,b,y} \|_{M(Z^{(V)})} = o_{\alpha \to \infty}(1)$$
uniformly for all $(y,b) \in Z_{<k}^{(V)} \times A_{=k}^{(V)} \backslash E$ with
\begin{equation}\label{rr}
R_{\overline{\alpha_{<k}}^{(V)}(y),b} = R.
\end{equation}

Fix $y,b$ as above.  From \eqref{qvy} (and the $k$-independence of $P_k$, and hence of $Q$) we see that $(e,\{\id\},b\downharpoonright_e,y\downharpoonright_e)$ is good for every $e \in \binom{V}{k}$.  Thus
by Definition \ref{quad}, we have
$$ \rho_{e, \{\id\}, b\downharpoonright_e, y\downharpoonright_e} = (Q^{(e)}(y\downharpoonright_e)|C_{e,\{\id\},b\downharpoonright_e})$$
and hence (by the $k$-independence of $Q$ again) we see from construction of $\sigma'$ that
$$ \sigma_{V,R_0,b,y} = (Q^{(V)}(y)|C_{V,\{\id\},b}).$$
From \eqref{qvy2} and Lemma \ref{cond} we thus have
$$ \| \sigma_{V,R_0,b,y} - (Q^{(V)}(y)|C_{V,\{\id\},b} \cap Z_R) \|_{M(Z^{(V)})} = o_{\alpha \to \infty}(1)$$
so by the triangle inequality it suffices to show that
\begin{equation}\label{sby}
\| \sigma_{V,R,b,y} - (Q^{(V)}(y)|C_{V,\{\id\},b} \cap Z_R) \|_{M(Z^{(V)})} = o_{\alpha \to \infty}(1).
\end{equation}

Let $w$ be drawn using law $\sigma_{V,R,b,y}$, and let $w'$ be drawn using law $(Q^{(V)}(y)|C_{V,\{\id\},b} \cap Z_R)$.  The lower order components $w_{<k}$, $w'_{<k}$ of $w, w'$ are both equal to $b$, so we focus on the top order components $w_k, w'_k$, which we split as $(w_k\downharpoonright_e)_{e \in \binom{V}{k}}$ and $(w'_k\downharpoonright_e)_{e \in \binom{V}{k}}$ respectively.

If $e_1,\ldots,e_n \in \binom{V}{k}$ are pairwise $R$-distinguishable, then by construction of $\sigma_{V,R,b,y}$ we have that $w_k\downharpoonright_{e_1}, \ldots, w_k\downharpoonright_{e_n}$ are jointly independent.  Conversely, if $e, e'$ are $R$-distinguishable, thus $(e,e',\phi) \in R$ for some $\phi \in \Hom(e,e')$, then from \eqref{wke2} we have the constraint
$$ w_k\downharpoonright_e = Z_{=k}^{(\phi)}( w_k\downharpoonright_{e'} ).$$

Now we observe that the random variables $w'_k \downharpoonright_e$ obey exactly the same independence and constraint properties.  Indeed, if $(e,e',\phi) \in R$, then the constraint
$$ w'_k\downharpoonright_e = Z_{=k}^{(\phi)}( w'_k\downharpoonright_{e'} ).$$
holds almost surely, since $w$ is constrained to lie in $Z_R$ almost surely.  On the other hand, if $e_1,\ldots,e_n \in \binom{V}{k}$ are pairwise $R$-distinguishable, and thus lie in disjoint equivalence classes of $R$-indistinguishability, then we claim that the random variables $w'_k\downharpoonright_{e_1}, \ldots, w'_k\downharpoonright_{e_n}$ are jointly independent.  Indeed, this claim is clearly true if $w'$ is drawn with law $Q^{(V)}$ (as all of the $w'_k\downharpoonright_e$ are jointly independent in this case), and the conditioning to $C_{V,\{\id\},b} \cap Z_R$ only couples together those pairs $w'_k\downharpoonright_e$, $w'_k\downharpoonright_{e'}$ which lie in the same equivalence class.

In view of the above discussion (and the fact that the cardinality of $\binom{V}{k}$ is independent of $\alpha$), we see that in order to conclude \eqref{sby}, it suffices to show that for each $e \in \binom{V}{k}$ separately, the laws of $w_k\downharpoonright_e$ and $w'_k\downharpoonright_e$ differ in $M(Z_{=k}^{(e)})$ norm by $o_{\alpha \to \infty}(1)$, uniformly in $y$ and $b$.

Fix $e$.  From the definition of $\sigma_{V,R,b,y}$, we see that $w_k\downharpoonright_e$ is distributed according to the law $\rho_{e, G_e, b\downharpoonright_e, y\downharpoonright_e}$.  The distribution of $w'_k\downharpoonright_e$ is more complicated.  However, by \eqref{qvy2} we know that this law differs from the measure $\pi_{V \to e} \circ (Q^{(V)}(y)|C_{V,\{\id\},b})$, where $\pi_{V \to e}: Z_{=k}^{(V)} \to Z_{=k}^{(e)}$ is the restriction map, by $o_{\alpha \to \infty}(1)$ in the total variation norm
$M(Z_{=k}^{(e)})$.  Thus it suffices to show that
$$ \| \rho_{e, G_e, b\downharpoonright_e, y\downharpoonright_e} - \pi_{V \to e} \circ (Q^{(V)}(y)|C_{V,\{\id\},b}) \|_{M(Z_{=k}^{(e)}} = o_{\alpha \to \infty}(1).$$
But since $P_k$ (and hence $Q$) is $k$-independent, we have
$$\pi_{V \to e} \circ (Q^{(V)}(y)|C_{V,\{\id\},b}) =
(Q^{(e)}(y\downharpoonright)|C_{e,\{\id\},b\downharpoonright_e}).$$
Meanwhile, from \eqref{qvy}, \eqref{qvy2} we have
$$  Q^{(V)}(y)(C_{V,\{\id\},b} \cap Z_R) > 0.$$
Using the inclusion
\begin{equation}\label{include}
\pi_{V \to e}(C_{V,\{\id\},b} \cap Z_R) \subset C_{e,G_e,b\downharpoonright_e}
\end{equation}
and using the $k$-independence of $Q$ once again, we conclude
$$ Q^{(e)}(y\downharpoonright_e)( C_{e,G_e,b\downharpoonright_e} ) > 0$$
and thus by Definition \ref{quad}
$$ \rho_{e, G_e, b\downharpoonright_e, y\downharpoonright_e} = ( Q^{(e)}(y\downharpoonright_e) | C_{e,G_e,b\downharpoonright_e} ).$$
Our task is thus to show that
$$ \| (Q^{(e)}(y\downharpoonright_e)|C_{e,\{\id\},b\downharpoonright_e}) -
(Q^{(e)}(y\downharpoonright_e)|C_{e,G_e,b\downharpoonright_e}) \| = o_{\alpha \to \infty}(1).$$
But from \eqref{qvy2}, the inclusion \eqref{include} and the $k$-independence of $Q$ once again, we have
$$ Q^{(e)}(y\downharpoonright_e)(C_{e,\{\id\},b\downharpoonright_e} \backslash C_{e,G_e,b\downharpoonright_e} ) = o_{\alpha \to \infty}(1)$$
and the claim follows from Lemma \ref{cond}.
\end{proof}

\subsubsection{Approximate absolute continuity}

We can now quickly prove \eqref{ejc-ac}.  We can phrase this claim in probabilistic language as follows. Let $z \in Z^{(V)}$ be drawn at random with law $\mu^{(V)}$, let $a := \overline{\alpha}^{(V)}(z) \in A^{(V)}$, let $x \in X_{k-1}$ be drawn independently with law $\nu_{k-1}$, let $\xi \in \Xi$ be drawn with law $\eta_x$, and let $w := \zeta_{\leq k}(a,(x,\xi)) \in Z^{(V)}$.  Let $\eps > 0$ be arbitrary.  Our task is to show that if $\alpha$ is sufficiently fine depending on $\eps$, then the distribution of $w$ is $\eps$-absolutely continuous with respect to $\mu^{(V)}$.  Thus, let $E \subset Z^{(V)}$ be a measurable set such that $\mu^{(V)}(E) = 0$.  Our task is to show that
\begin{equation}\label{pwe}
\Prob( w \in E ) \leq \eps.
\end{equation}

From \eqref{mupj} we have $\mu^{(V)} = P_k^{(V)} \circ
\mu_{<k}^{(V)}$.  Since $\mu^{(V)}(E) = 0$, we conclude that the set
$E' := \{ y \in Z_{<k}^{(V)}: P_k^{(V)}(y)(E)
> 0 \}$ has measure zero with respect to $\mu_{<k}^{(V)}$.  By the
inductive hypothesis \eqref{ejc-ac}, we already know that the
distribution of $w_{<k} \in Z_{<k}^{(V)}$ is $\eps/4$-absolutely
continuous with respect to $\mu_{<k}^{(V)}$ if $\alpha$ is
sufficiently fine. Thus we have
$$ \Prob( w_{<k} \in E' ) < \eps/4.$$

Furthermore, by Proposition \ref{quadgood}, we see that
$$ \Prob( ([k], \Stab(a), a_k, w_{<k}) \hbox{ bad } ) < \eps/4$$
for $\alpha$ sufficiently fine, which implies that
$$ \Prob( ([k], \{\id\}, a_k, w_{<k}) \hbox{ bad } ) < \eps/4.$$

Also, by Proposition \ref{decouple} and Markov's inequality, we have
$$ \Prob( \| \sigma_{V,R_a,a_k,w_{<k}} - \sigma_{V,R_0,a_k,w_{<k}} \|_{M(Z^{(V)})} > \eps/4 ) < \eps/4 $$
if $\alpha$ is sufficiently fine.

Now let us fix $z,a,x$ (and hence $w_{<k}$), and condition on the
events that $w_{<k} \not \in E'$ (so $P_k^{(V)}(w_{<k})(E) = 0$) and
that
\begin{equation}\label{events}
([k], \{\id\}, a_k, w_{<k}) \hbox{ good }; \quad \|
\sigma_{V,R_a,a_k,w_{<k}} - \sigma_{V,R_0,a_k,w_{<k}}
\|_{M(Z^{(V)})} \leq \eps/4.
\end{equation}
By the preceding discussion, the event \eqref{events} occurs with probability at least $1-3\eps/4$.
As discussed previously, the random variable $w$ now has the distribution of $\sigma_{V,R_a,a_k,w_{<k}}$.  By \eqref{events}, we thus have the conditional probability estimate
$$ \Prob( w \in E | z,a,x ) \leq \sigma_{V,R_0,a_k,w_{<k}}(E) + \eps/4.$$
But as $([k], \{\id\}, a_k, w_{<k})$ is good, we see from construction of $\sigma_{V,R_0,a_k,w_{<k}}$ (and the $k$-independence of $P_k$) that $\sigma_{V,R_0,a_k,w_{<k}}$ is absolutely continuous with respect to $P_k^{(V)}$, and thus by \eqref{events} we have $\sigma_{V,R_0,a_k,w_{<k}}(E) = 0$.  Integrating over $z,a,x$ and applying the union bound we obtain the claim \eqref{pwe}.

\subsubsection{Convergence to the diagonal}

Now, we verify \eqref{limz}.  We shall modify the argument used to establish Lemma \ref{coerce-conv}.  Fix $V, H, F$ as in the Proposition; we may normalise $F$ to be bounded in norm by $1$.  As in the proof of Lemma \ref{coerce-conv}, it is convenient to use the probabilistic formulation \eqref{limzo}.
Let $z \in Z^{(V)}$ be drawn at random with law $\mu^{(V)}$, and then for fixed $z$, let $x \in X_{k-1}$ be drawn independently at random with law $\nu_{k-1}$, $\xi \in \Xi$ drawn with law $\eta_x$, and set $a := \overline{\alpha}^{(V)}(z)$ and $w := \zeta_{\leq k}^{(V)}(a,x)$.  Our task is to show that
$$
\E \| F(z) - F(w) \|_H = o_{\alpha \to \infty}(1),
$$
where the decay rate $o_{\alpha \to \infty}(1)$ is allowed to depend on $V$, $F$ and $H$.

As usual, we decompose $z = (z_{<k},z_k)$, $a = (a_{<k},a_k)$, and $w = (w_{<k},w_k)$.
From the inductive hypothesis \eqref{limzo} we have
$$ \Prob( \overline{\alpha_{<k}}^{(V)}(z_{<k}) \neq \overline{\alpha_{<k}}^{(V)}(w_{<k}) ) = o_{\alpha \to \infty}(1).$$
Since $\overline{\alpha_{<k}}^{(V)}(z_{<k}) = a_{<k}$, it thus suffices to show that
\begin{equation}\label{efs}
\E \| F(z) - F(w) \|_H \I( S ) = o_{\alpha \to \infty}(1)
\end{equation}
where $S$ is the event that
\begin{equation}\label{aw}
\overline{\alpha_{<k}}^{(V)}(w_{<k}) = a_{<k}.
\end{equation}

For fixed $z,a,x$ (and hence $w_{<k}$), we recall that $w$ has the distribution of $\sigma_{V,R_a,a_k,w_{<k}}$.  Thus we can express the left-hand side of \eqref{efs} as
$$ \E (\int_{Z^{(V)}} \|F(z)-F(y)\|_H\ d\sigma_{V,R_a,a_k,w_{<k}}(y)) \I(S).$$
From Proposition \ref{decouple} (and the boundedness of $F$) we have
$$ \E (\int_{Z^{(V)}} \|F(z)-F(y)\|_H\ d\sigma_{V,R_a,a_k,w_{<k}}(y)) \I(S)
\leq \E (\int_{Z^{(V)}} \|F(z)-F(y)\|_H\ d\sigma_{V,R_0,a_k,w_{<k}}(y)) \I(S) + o_{\alpha \to \infty}(1)$$
and so it suffices to show that
$$ \E (\int_{Z^{(V)}} \|F(z)-F(y)\|_H\ d\sigma_{V,R_0,a_k,w_{<k}}(y)) \I(S) = o_{\alpha \to \infty}(1).$$
Using \eqref{aw}, we can bound the left-hand side by
$$ \E \int_{Z^{(V)}} \|F(z)-F(y)\|_H\ d\sigma_{V,R_0,(\overline{\alpha_{<k}}^{(V)}(w_{<k}),a_k),w_{<k}}(y)).$$
By the triangle inequality, we can bound this in turn by the sum of
$$ \E \| F(z) - G_{a_k}(w_{<k})\|_H$$
and
$$ \E J_{a_k}(w_{<k})$$
where $G_{a_k}: Z_{<k}^{(V)} \to H$, $J_{a_k}: Z_{<k}^{(V)} \to \R$ are the bounded measurable functions
$$ G_{a_k}(w_{<k}) := \int_{Z^{(V)}} F(y) d\sigma_{V,R_0,(\overline{\alpha_{<k}}^{(V)}(w_{<k}),a_k),w_{<k}}(y)$$
and
$$ J_{a_k}(w_{<k}) := \int_{Z^{(V)}} \|G_{a_k}(w_{<k})-F(y)\|_H\ d\sigma_{V,R_0,(\overline{\alpha_{<k}}^{(V)}(w_{<k}),a_k),w_{<k}}(y).$$
From the inductive hypothesis \eqref{limzo} we have
$$ \E \sum_{b \in A_{=k}^{(V)}} \| G_b(w_{<k}) - G_b(z_{<k}) \|_H = o_{\alpha \to \infty}(1)$$
and
$$ \E \sum_{b \in A_{=k}^{(V)}} |J_b(w_{<k}) - J_b(z_{<k})| = o_{\alpha \to \infty}(1)$$
and so by the triangle inequality it suffices to show that
\begin{equation}\label{fg}
 \E \| F(z) - G_{a_k}(z_{<k})\|_H = o_{\alpha \to \infty}(1)
\end{equation}
and
\begin{equation}\label{fg2}
 \E J_{a_k}(z_{<k}) = o_{\alpha \to \infty}(1).
\end{equation}

Let us temporarily freeze $z_{<k}$ (and thus $a_{<k}$), then $z$ has the distribution of $P_k^{(V)}(z_{<k})$.  In particular, for any $b \in A_{=k}^{(V)}$, the probability that $a_k = b$ (conditioning on $z_{<k}$) is equal to $P_k^{(V)}(z_{<k})( C'_b )$, where
$$ C'_{b} := Z_{<k}^{(V)} \times (\overline{\alpha_{=k}}^{(V)})^{-1}(\{b\}).$$
Thus we see that those $b$ for which $P_k^{(V)}(z_{<k})( C'_b )=0$ will almost surely not be equal to $a_k$; in other words, we almost surely have
$$ P_k^{(V)}(z_{<k})( C'_{a_k} ) > 0.$$
From Definition \ref{quad}, we conclude that $(V, \{\id\}, a_k,z_{<k})$ is almost surely good.
Since $P_k$ is $k$-independent, we conclude that $(e, \{\id\}, a_k\downharpoonright_e, z_{<k}\downharpoonright_e)$ is also almost surely good for all $e \in \binom{V}{k}$. From this, the $k$-independence of $P_k$ again, and the definition of $\sigma_{V,R_0,a_k,z_{<k}}$, we conclude that
$$ \sigma_{V,R_0,a_k,z_{<k}} = ( P_k^{(V)}(z_{<k}) | C'_{a_k} )$$
almost surely.  Also, note that for any $b \in A_{=k}^{(V)}$, the distribution of $z$ conditioned to the event $a_k=b$ is also given by $( P_k^{(V)}(z_{<k}) | C'_{a_k} )$.  From this, we see that the left-hand sides of \eqref{fg} and \eqref{fg2} are both equal to
$$ \int_{Z_{<k}^{(V)}} \sum_{b \in A_{=k}^{(V)}} P_k^{(V)}(v)(C'_b)
\int_{Z^{(V)}} \| F(y) - \int_{Z^{(V)}} F(u)\ ( P_k^{(V)}(v,du) | C'_b ) \|_H
( P_k^{(V)}(v,dy) | C'_b )\ d\mu_{<k}^{(V)}(v).$$
From Lemma \ref{dom}, we have
$$ \sum_{b \in A_{=k}^{(V)}} P_k^{(V)}(v)(C'_b)
\int_{Z^{(V)}} \| F(y) - \int_{Z^{(V)}} F(u)\ ( P_k^{(V)}(v,du) | C'_b ) \|_H
( P_k^{(V)}(v,dy) | C'_b )\ d\mu_{<k}^{(V)}(v) = o_{\alpha \to \infty}(1)$$
for all $v \in Z_{<k}^{(V)}$.  The claim then follows from Lemma \ref{dct}.

This (finally!) completes the proof of Proposition \ref{indiscrete} and thus Proposition \ref{disc-ident2}, which in turn completes the proof of all the local repairability results claimed in the introduction.

\appendix

\section{Some measure theory and probability}\label{prob}

In this appendix we recall some notions from measure theory and
probability which we will rely on to establish our positive results.

We will work throughout this paper with sub-Cantor spaces (as defined in Definition \ref{subcantor}).  
All of the notation here however extends to the larger category of standard Borel spaces, i.e. a Polish space (a complete separable metrisable space), together with their Borel $\sigma$-algebra, which is generated by the open sets.

If $X$ is a sub-Cantor space, we will write $\Pr(X)$ for the space of
all probability Borel measures on $\X$.  This is a convex subset of the space $M(X)$ of all real finite measures on $\X$, equipped with the usual total variation norm
$$\| \mu \|_{M(X)} := |\mu|(X) = \sup\{ |\mu(E)-\mu(F)|: E, F \subset X, \hbox{ disjoint} \}.$$

An important operation for us will be that of \emph{conditioning}: if $\mu \in \Pr(X)$ is a probability measure and $E \subset X$ is an event with $\mu(E) > 0$, we define the \emph{conditioning} $(\mu|E) \in \Pr(X)$ of $\mu$ to $E$ to be the probability measure defined by the usual formula
$$ (\mu|E)(F) := \frac{\mu(E \cap F)}{\mu(E)}.$$
The following computation is easily verified:

\begin{lemma}[Conditioning by high probability events is mild]\label{cond}  Let $\mu \in \Pr(X)$ and $E \subset X$ be such that $\mu(E) \geq 1-\eps$ for some $0 < \eps < 1/2$.  Then $\| \mu - (\mu|E) \|_{M(X)} \ll \eps$.
\end{lemma}

The space $M(X)$ (and hence $\Pr(X)$) comes equipped with the \emph{vague topology} (or \emph{weak-* topology}), defined as the topology induced by the functionals $\mu \mapsto \int_X f\ d\mu$ for all bounded continuous supported $f$.  The following lemma is well-known:

\begin{lemma}[Prokhorov's theorem]\label{seq}  Let $X$ be a sub-Cantor space, and let $\mu_n$ be a sequence of measures in $\Pr(X)$.   Then there is some subsequence $\mu_{n_j}$ of $\mu_n$ which converges vaguely to another measure $\mu \in \Pr(X)$.
\end{lemma}

The space $\Pr(X)$ also comes with a $\sigma$-algebra, induced by the evaluation mappings $\mu \mapsto \mu(A)$ for all measurable $A \subset X$.  This allows us to introduce the notion of a \emph{probability kernel}, which is fundamental to our arguments for our positive results:

\begin{definition}[Probability kernels]
Let $X,Y$ be sub-Cantor spaces.  A \emph{probability kernel from $Y$ to $X$} is a measurable function $P: Y \to \Pr(X)$ from $Y$ to $\Pr(X)$. We will use the notation $P:Y \rightsquigarrow X$ to denote the fact that $P$ is a probability kerne from $Y$ to $X$.  If $y \in Y$ and $f: X \to \R$ is measurable, we use $\int_X f(x)\ P(y, dx)$ to denote the integral of $f$ against the measure $P(y) \in \Pr(X)$.  We call a probability kernel $P: Y \rightsquigarrow X$ \emph{trivial} if $X$ is a point.
\end{definition}

\begin{remark} A probability kernel $P: Y \rightsquigarrow X$ can be viewed as describing the law for some random variable on $X$, where the distribution of that law depends on the value of a parameter $y$ in $Y$.  Indeed, one common way to construct probability kernels is to condition one random variable on the value of another; in measure-theoretic terms, this is closely related to the operation of \emph{disintegrating} a measure with respect to a factor.
\end{remark}

Two important special cases of a probability kernel arise from probability measures and from measurable functions.  Indeed, if $\mu \in \Pr(X)$ is a probability measure, we can (by abuse of notation) identify $\mu$ with a probability kernel $\mu: \pt \rightsquigarrow X$ which maps the point in $\pt$ to $\mu$.  Similarly, if $\phi: Y \to X$ is a measurable function, we can (by further abuse of notation) identify $\phi$ with a probability kernel $\phi: Y \rightsquigarrow X$ which maps any point $y \in Y$ to the Dirac mass $\delta_{\phi(y)}$ at $y$.  These abuses of notation shall be in entail throughout the paper.

We now define two important notions on probability kernels, namely composition and product.

\begin{definition}[Composition of kernels] If $P: Y \rightsquigarrow X$ and $Q:Z
\rightsquigarrow Y$ are probability kernels between sub-Cantor spaces, we define \emph{composition}
$P\circ Q: Z\rightsquigarrow X$ by the formula
by
\[P\circ Q(z)(E) := \int_YP(y)(E)\,Q(z,d y)\]
for all $z \in Z$ and all measurable $E \subset X$.
\end{definition}

\begin{example}[Special cases of composition]  Let $\phi: Y \to X$ and $\psi: Z \to Y$ be measurable maps, which we then identify with probability kernels, and let $\mu \in \Pr(Y)$ be a probability measure (which we also identify with a probability kernel).  Then $\phi \circ \psi$ is just the usual composition of $\phi$ and $\psi$, while $\phi \circ \mu$ is the pushforward of $\mu$ under $\phi$.
\end{example}

\begin{remark} For future reference we observe that a probability kernel $P: Y \rightsquigarrow X$ not only pushes forward probability measures $\mu \in \Pr(Y)$ to probability measures $P \circ \mu \in \Pr(X)$, but in fact can push forward arbitrary finite Borel measures $\mu$ on $Y$ to finite Borel measures $P \circ \mu$ on $X$, by the formula
$$ P \circ \mu(E) := \int_Y P_y(E)\ d\mu(y)$$
for all measurable $E \subset X$.
\end{remark}

\begin{definition}[Product of kernels] If $S$ is an at most countable set, and $P_s: Y \rightsquigarrow X_s$ is a probability kernel between sub-Cantor spaces for each $s \in S$, then we define the product $\bigotimes_{s \in S} P_s: Y \rightsquigarrow \prod_{s \in S} X_s$ by defining $\bigotimes_{s \in S} P_s(y)$ for each $y \in Y$ to be the product of the probability measures $P_s(y)$ for $s \in S$.  We also write $P^{\otimes S}$ for $\bigotimes_{s \in S} P$.
\end{definition}

Finally, we define the notion of one probability kernel being absolutely continuous with respect to another.

\begin{definition}[Absolute continuity]  If $\mu, \nu \in \Pr(X)$ are two probability measures on a sub-Cantor space, we say that $\mu$ is \emph{absolutely continuous with respect to $\nu$}, and write $\mu \ll \nu$, if for every measurable $E \subset X$ we have $\mu(E)=0$ whenever $\nu(E)=0$.  If $P, P': Y \rightsquigarrow X$ are probability kernels, we say that $P'$ is \emph{absolutely continuous with respect to $P$}, and write $P' \ll P$, if we have $P'(y) \ll P(y)$ for all $y \in Y$.
\end{definition}

\begin{example}\label{condit}  If $\mu \in \Pr(X)$ is a probability measure, and $E \subset X$ is such that $\mu(E) > 0$, then $(\mu|E) \ll \mu$.
\end{example}

The notion of absolute continuity is clearly a partial ordering on probability kernels between two given sub-Cantor spaces.  It also interacts nicely with both composition and finite products:

\begin{lemma}[Preservation of absolute continuity]\label{pac}
\begin{itemize}
\item Let $P, P': Y \rightsquigarrow X$ and $Q, Q': Z \rightsquigarrow Y$ be probability kernels. If $P' \ll P$ and $Q' \ll Q$, then $P' \circ Q' \ll P \circ Q$.
\item Let $S$ be a finite set, and for each $s \in S$ let $P_s, P'_s: Y \rightsquigarrow X_s$ be probability kernels such that $P'_s \ll P_s$.  Then $\bigotimes_{s \in S} P'_s \ll \bigotimes_{s \in S} P_s$.
\end{itemize}
\end{lemma}

\begin{proof} Both claims follow immediately from the Fubini-Tonelli theorem.
\end{proof}

In some of our arguments we will need a perturbed version of absolute continuity.

\begin{definition}[$\eps$-absolute continuity]\label{epsac-def}  Let $\eps \ge 0$.  If $\mu, \nu \in \Pr(X)$ are two probability measures on a sub-Cantor space, we say that $\mu$ is \emph{$\eps$-absolutely continuous with respect to $\nu$}, and write $\mu \ll_\eps \nu$, if for every measurable $E \subset X$ we have $\mu(E) \leq \eps$ whenever $\nu(E)=0$.
\end{definition}

From the Lebesgue-Radon-Nikodym theorem we have several equivalent characterisations of $\eps$-absolute continuity:

\begin{proposition}[Equivalent formulations of $\eps$-absolute continuity]\label{epsac} Let $\eps \ge 0$, and let $\mu, \nu \in \Pr(X)$ are two probability measures on a sub-Cantor space $X$.  Then the following statements are equivalent:
\begin{itemize}
\item $\mu$ is $\eps$-absolutely continuous with respect to $\nu$.
\item For every $\eps' > \eps$ there exists $\delta > 0$ such that we have $\mu(E) \leq \eps'$ for every measurable $E \subset X$ with $\nu(E) < \delta$.
\item There exists a compact set $E \subset X$ with $\mu(E) \leq \eps$ and $\nu(E)=0$ such that $\I(E^c) \mu \ll \nu$.
\end{itemize}
\end{proposition}


\parskip 0pt

\textsc{Department of Mathematics\\ University of California at
Los Angeles, Los Angeles, CA 90095, USA}

\vspace{7pt}

Email: \verb|timaustin@math.ucla.edu|

Web: \verb|www.math.ucla.edu/~timaustin|

\vspace{7pt}

Email: \verb|tao@math.ucla.edu|

Web: \verb|www.math.ucla.edu/~tao|

\end{document}